\documentclass[a4paper,11pt]{article}
\usepackage[utf8]{inputenc} 
\usepackage[english]{babel}
\usepackage[T1]{fontenc}   
\usepackage{mathpazo}
\usepackage{graphicx}
\usepackage{amsthm,amsmath,amsfonts,stmaryrd,mathrsfs,amssymb}
\usepackage{mathtools}
\usepackage{bbm}
\usepackage{enumitem}
\usepackage[top=2.5cm, bottom=2.5cm, left=2cm, right=2cm]{geometry}
\usepackage{subfig}
\usepackage{multirow}
\usepackage{array}
\usepackage{etoolbox}
\usepackage{mdframed}

\usepackage[dvipsnames]{xcolor}
\usepackage[unicode,colorlinks=true,linkcolor=NavyBlue,citecolor=BrickRed,pdfencoding=auto,psdextra]{hyperref}
\usepackage{tikz}

\setenumerate[1]{label=(\roman*), font = \normalfont}
\setitemize[1]{itemsep = 0pt}

\makeatletter

\@addtoreset{equation}{section}
\makeatother

\let\originalleft\left
\let\originalright\right
\renewcommand{\left}{\mathopen{}\mathclose\bgroup\originalleft}
\renewcommand{\right}{\aftergroup\egroup\originalright}

\newlength{\bibitemsep}
\newlength{\bibparskip}
\let\oldthebibliography\thebibliography
\renewcommand\thebibliography[1]{\oldthebibliography{#1}
	\setlength{\parskip}{\bibitemsep}
	\setlength{\itemsep}{\bibparskip}}

\newcommand{\ensemblenombre}[1]{\mathbb{#1}}
\newcommand{\N}{\ensemblenombre{N}}
\newcommand{\R}{\ensemblenombre{R}}
\newcommand{\Z}{\ensemblenombre{Z}}

\renewcommand{\P}{\mathbb{P}}
\newcommand{\E}{\mathbb{E}}

\newcommand{\Ec}[1]{\mathbb{E} \left[#1\right]}
\newcommand{\Pp}[1]{\mathbb{P} \left(#1\right)}
\newcommand{\Ecsq}[2]{\mathbb{E} \left[#1\mathrel{}\middle|\mathrel{}#2\right]}
\newcommand{\Ppsq}[2]{\mathbb{P} \left(#1\mathrel{}\middle|\mathrel{}#2\right)}

\newcommand{\intervalle}[4]{\mathopen{#1}#2\mathclose{}\mathpunct{},#3\mathclose{#4}}
\newcommand{\intervalleff}[2]{\intervalle{[}{#1}{#2}{]}}

\newcommand{\intervallefo}[2]{\intervalle{[}{#1}{#2}{)}}
\newcommand{\intervalleoo}[2]{\intervalle{(}{#1}{#2}{)}}
\newcommand{\intervalleentier}[2]{\intervalle\llbracket{#1}{#2}\rrbracket}
 
\newcommand{\petito}[1]{o\mathopen{}\left(#1\right)}
\newcommand{\grandO}[1]{O\mathopen{}\left(#1\right)}

\newcommand{\ind}[1]{\mathbb{1}_{\lbrace #1 \rbrace}}  
\newcommand{\1}{\mathbb{1}}

\newcommand{\ttu}{\mathtt{u}}
\newcommand{\ttv}{\mathtt{v}}

\newcommand{\ttT}{\mathtt{T}}

\DeclareMathOperator{\haut}{ht}
\DeclareMathOperator{\diam}{diam}

\DeclareMathOperator{\wrt}{WRT}

\newcommand{\diff}{\mathop{}\mathopen{}\mathrm{d}}
\newcommand{\abs}[1]{\left\lvert #1 \right\rvert}

\newmdtheoremenv{theorem}{Theorem}[section]
\newmdtheoremenv{proposition}[theorem]{Proposition}
\newtheorem{lemma}[theorem]{Lemma}

\newtheorem{remark}[theorem]{Remark}

\begin{document}

\title{Height of weighted recursive trees with sub-polynomially growing total weight}
\author{Michel Pain\footnote{Institut de Mathématiques de Toulouse (UMR 5219), Université de Toulouse, CNRS.
}
\and
Delphin Sénizergues\footnote{Department of Mathematics, University of British Columbia}}
\maketitle

\begin{abstract}
Weighted recursive trees are built by adding successively vertices with predetermined weights to a tree: each new vertex is attached to a parent chosen at  random with probability proportional to its weight.
In the case where the total weight of the tree at step $n$ grows polynomially in $n$, we obtained in \cite{painsenizergues2022} an asymptotic expansion for the height of the tree, which falls into the university class of the maximum of branching random walks.
In this paper, we consider the case of a total weight growing sub-polynomially in $n$ and obtain asymptotics for the height of the tree in several regimes,
showing that universality is broken and exhibiting new behaviors. 
\end{abstract}

\section{Introduction}

Weighted recursive trees (WRTs) are an example of a model of growing random trees that has attracted a lot of attention in the recent years, see for example \cite{eslava2021,fountoulakis2019,hiesmayrislak2020,iyer2020,lodewijks2020,lodewijks2021,mailleruribebravo2019,painsenizergues2022,senizergues2021}.
This model is parametrized by a whole sequence $(w_n)_{n\geq 1}$ of non-negative numbers that we call \emph{weights} and that account for some inhomogeneity between vertices in the growth dynamic of the tree. 
We are interested in the impact of these weights on the behavior of the height of the obtained trees.
In our previous paper \cite{painsenizergues2022}, we considered weight sequences for which the height of the trees was comparable to other well-known models of growing trees and fell in the same universality class as the maximal displacement for branching random walks.
We refer the reader to the introduction of~\cite{painsenizergues2022}, particularly Section~1.4, for a discussion containing a review of the literature concerning the asymptotic height of related models of growing random trees, links with branching processes and a review of the literature about WRTs.
In the current paper, we work with weight sequences that take us out of this universality class, 
in a regime where correlations due to the tree structure are stronger and have a more significant impact on the height of the tree.

\subsection{Presentation of the model}

\paragraph{Definition of WRTs.}
Let us define the model of weighted recursive trees, first introduced in \cite{borovkov2006} by Borovkov and Vatutin. 
For any sequence of non-negative real numbers $(w_n)_{n\geq 1}$ with $w_1>0$, we define the distribution $\wrt((w_n)_{n\geq 1})$ on sequences of growing rooted trees, which is called the \emph{weighted recursive tree with weights $(w_n)_{n\geq 1}$}. 
We construct a sequence of rooted trees $(\ttT_n)_{n\geq 1}$ starting from $\ttT_1$ containing only one root-vertex $\ttu_1$ 
by letting it evolve in the following manner: the tree $\ttT_{n+1}$ is obtained from $\ttT_n$ by adding a vertex $\ttu_{n+1}$ with label $n+1$. 
The parent of this new vertex is chosen to be the vertex with label $K_{n+1}$, where
\[
\forall k\in \{1,\dots,n\}, \qquad 
\Ppsq{K_{n+1}=k}{\ttT_n} = \frac{w_k}{W_n},
\]
denoting, for each $n\geq 1$, $W_n\coloneqq \sum_{i=1}^nw_i$ the sum of the $n$ first weights.
A key quantity is the sequence $(a_n)_{n\geq 1}$ defined as 
\begin{equation}\label{eq:def an}
a_n \coloneqq \sum_{i=1}^n \frac{w_i}{W_i}, \qquad \text{ for any }n\geq 1.
\end{equation}
Its role can be explained by the following fact (see Remark~\ref{rem:consequence_many-to-one}): the height of $\ttu_{n+1}$ has the same distribution as $\sum_{i=1}^n B_i$, where the $B_i$'s are independent Bernoulli r.v.\@ with parameter $w_i/W_i$.
In particular, the height of $\ttu_{n+1}$ has mean $a_n$ and
its variance is $\sum_{i=1}^n \frac{w_i}{W_i} (1-\frac{w_i}{W_i})$, which behaves like $a_n$ as $n \to \infty$ if $w_n/W_n \to 0$.
Hence, if moreover $a_n \to \infty$, the height of $\ttu_{n+1}$, and therefore the height of a vertex chosen in $\ttT_n$ proportionally to its weight, is close to $a_n$ with fluctuations of the order of $\sqrt{a_n}$.

\paragraph{Objective of the paper.}
In \cite{painsenizergues2022}, in the case where the total weight $W_n$ grows polynomially, that is $W_n = \lambda \cdot n^\gamma + O(n^{\gamma-\varepsilon})$ for some $\lambda,\gamma,\varepsilon > 0$,
we proved (under some weak additional assumption) the following asymptotic expansion for the height of $\ttT_n$,
\begin{equation} \label{eq:height_polynomial_case}
	\haut(\ttT_n) = c_1(\gamma) \log n - c_2(\gamma) \log \log n + O_\P(1),
\end{equation}
as $n \to \infty$, where $c_1(\gamma),c_2(\gamma)$ are positive constants depending only on $\gamma$, and the term $O_\P(1)$ denotes a tight sequence of random variables.
This extended a previous result \cite{senizergues2021} by one of the authors, which described the first order for the height. Moreover, because of the precise value of $c_2(\gamma)$, this showed that, in this regime, the height of $\ttT_n$ falls into the universality class of the maximum of branching random walks.
By heuristics presented in Section \ref{subsec:BRW}, we believe that the regime where this universal behavior holds is exactly the one where $W_n = n^{\gamma+o(1)}$ (or equivalently $a_n \sim \gamma \log n$).

In this paper, we show that new behaviors, outside of this universality class, appear when the total weight $W_n$ grows sub-polynomially or converges to a finite limit.
This corresponds to the case where $a_n$ grows sub-logarithmically or converges. 
%
To state our results, we distinguish two main cases.
In the first main case, we consider sequences of weights such that $a_n$ varies as powers of $\log n$: either $a_n$ grows like $(\log n)^p$ for some $p \in (0,1)$, or $a_n$ converges to some finite limit at speed $(\log n)^{-q}$ for some $q > 0$.
The case of a slower divergence or convergence is also included and the result in this first main case turns out to be unified.
%
In the second main case, we investigate the case where $a_n$ converges at a speed $\exp(-\log^\beta n)$ for some $\beta > 0$.
This leads to three sub-cases with different asymptotics for the height of $\ttT_n$.
Before proving these results, we first establish general criteria for obtaining some upper and lower bounds (Proposition~\ref{prop:upper_bound} and Proposition~\ref{prop:lower_bound}).
These bounds are valid in great generality and could also be used to obtain the asymptotic expansion for $\haut(\ttT_n)$ in regimes that are not treated in this paper.


Note that in this paper, we only state results on the height $\haut(\ttT_n)$ of the tree, whereas in \cite{painsenizergues2022} we also expressed an expansion similar to \eqref{eq:height_polynomial_case} for the diameter $\diam(\ttT_n)$, which can be obtained from the latter by just multiplying all the terms by 2.  
The same would hold in the setting studied in this paper by the same argument, see Remark \ref{rem:diameter}.


\paragraph{Slowly varying function.} In all cases mentioned above, we allow corrective factors in the behavior of $a_n$ and those are expressed 
using a slowly varying function satisfying some regularity conditions.
For a function $L \colon [1,\infty) \to (0,\infty)$, we consider the following family of assumptions \ref{hyp_SV_k} for $k\geq 0$:
\begin{enumerate}[label=$(\mathrm{SV}_\arabic*)$,leftmargin=*,topsep=8pt]
	\setcounter{enumi}{-1}
	\item\label{hyp_SV_0} 
	$L$ is slowly varying, that is $L$ is measurable and, for any $\lambda > 0$, $L(\lambda x)/L(x) \to 1$ as $x \to \infty$;
\end{enumerate}
or, for some integer $k \geq 1$,
\begin{enumerate}[label=$(\mathrm{SV}_k)$,leftmargin=*,topsep=8pt]
	\item\label{hyp_SV_k} 
	$L$ is $k$ times differentiable and,
	for any $1 \leq i \leq k$, $x^i L^{(i)}(x)/L(x) \to 0$ as $x \to \infty$.
\end{enumerate}
Note that \hyperref[hyp_SV_k]{$(\mathrm{SV}_1)$} implies \ref{hyp_SV_0}, see e.g. Lemma \ref{lem:slowly_varying}.
Moreover, \ref{hyp_SV_k} is the typical behavior that one would want for a $k$ times differentiable slowly varying function: \ref{hyp_SV_k} is a restriction to the $k$ first derivatives of the definition of a smoothly varying function of index 0, see e.g.\@ \cite[Equation (1.8.1')]{binghamgoldieteugels1989}.

\paragraph{Notation.}
Some of our assumptions and results are expressed using the Landau big-$O$ and small-$o$ notation: we write $x_n=\grandO{y_n}$ if there exists a constant $C$ such that $\abs{x_n}\leq C \abs{y_n}$ for all $n\geq 1$; we write $x_n=\petito{y_n}$ if for every $\epsilon>0$ there exists $N$ such that for all $n\geq N$ we have $\abs{x_n}\leq \epsilon \abs{y_n}$. 
Moreover, we write $x_n \sim y_n$ if $x_n=y_n(1+\petito{1})$.
%
Throughout the paper, we also denote $\N \coloneqq \{0,1,2,\dots\}$ the set of non-negative integers, and use the notation $\intervalleentier{a}{b} \coloneqq \intervalleff{a}{b}\cap \Z$ and $\llbracket a, b \mathclose{\llbracket} \coloneqq \intervallefo{a}{b}\cap \Z$ for integer intervals.
We denote by $C$ a positive constant that can change from line to line.

\paragraph{Relation between $a_n$ and $W_n$.}
We express our assumptions on the sequence of weights $(w_i)_{i\geq 1}$ in terms of the asymptotic behavior of $a_n$, defined in \eqref{eq:def an}, and a control on the quantity $\sum_{i=n}^{\infty} w_i^2/W_i^2$. 
In fact, we prove in Section \ref{sec:correspondence behavior an and log Wn} the following relation between $a_n$ and $W_n$:
if $\sum_{i=1}^{\infty} w_i^2/W_i^2 < \infty$, then there exists a constant $K$ such that, as $n \to \infty$,
\begin{equation}\label{eq:transfer assum Wn an}
a_n = \log W_n + K + O\left(\sum_{i=n+1}^{\infty}\left(\frac{w_i}{W_i}\right)^2\right).
\end{equation}
Hence, it is easy to reinterpret our assumptions in term of the asymptotic behavior of $W_n$ instead.
Note that the last display does not assume that $a_n$ (or equivalently $W_n$) tends to infinity with $n$, and is also useful when those quantities are bounded.

\subsection{Variance varying like powers of \texorpdfstring{$\log n$}{log n}}
\label{sec:powers_of_log_intro}

We will first state results about the case where $a_n$ can be written as 
\begin{equation} \label{eq:assumption_a_n_1}
a_n = \int_1^{\log n} x^{-\alpha} L(x) \diff x 
+ o\left( (\log n)^{-1-\alpha} (\log \log n)^2 L(\log n) \right),
\end{equation}
where $\alpha > 0$ and $L$ is a positive function on $[1,\infty)$ satisfying \hyperref[hyp_SV_k]{$(\mathrm{SV}_2)$}.
By Lemma \ref{lem:from_J_to_L}, such a function $L$ exists if, for some function $J \colon [1,\infty) \to (0,\infty)$ satisfying \hyperref[hyp_SV_k]{$(\mathrm{SV}_3)$}, we have
\begin{itemize}
	\item $a_n = (\log n)^{1-\alpha} J(\log n) + o\left( (\log n)^{-1-\alpha} (\log \log n)^2 J(\log n) \right)$ if $\alpha \in (0,1)$;
	\item $a_\infty - a_n = (\log n)^{1-\alpha} J(\log n) + o\left( (\log n)^{-1-\alpha} (\log \log n)^2 J(\log n) \right)$ if $\alpha > 1$, where $\displaystyle a_\infty = \lim_{n\to\infty} a_n$.
\end{itemize}
In the case $\alpha =1$, no such general criterion exists and $a_n$ can either converge or go to infinity.
We also assume that
\begin{equation} \label{eq:ass_somme_carre}
\sum_{i\geq n} \frac{w_i^2}{W_i^2} = O \left( \frac{1}{n} \right).
\end{equation}
This assumption guarantees some regularity for the weight sequence $(w_n)_{n\geq 1}$ and ensures for example that we cannot be in a degenerate case where most of the weight of the tree is concentrated on a very sparse subset of vertices.
For comparison, if $w_n$ was regular in $n$, that is $\frac{w_n}{W_n} \approx \frac{\diff}{\diff n} a_n \approx n^{-1} (\log n)^{-\alpha} L(\log n)$, then we would have
$\sum_{i\geq n} w_i^2/W_i^2 = O(n^{-1} (\log n)^{-2\alpha} L^2(\log n))$.
\begin{theorem} \label{thm:powers_of_log}
	Assume \eqref{eq:assumption_a_n_1} for $\alpha > 0$ and some function $L \colon [1,\infty) \to (0,\infty)$ 
	satisfying \hyperref[hyp_SV_k]{$(\mathrm{SV}_2)$}. 
	Also assume  \eqref{eq:ass_somme_carre}.
	Then, almost surely, as $n \to \infty$, we have
	\begin{align*}
	\haut(\ttT_n)
	& = \frac{\log n}{\alpha \log \log n} 
	+ \frac{\log n}{(\alpha \log \log n)^2} 
	\Biggl( 
	\sum_{k\geq 0} \left( \frac{\log L(\log n)}{\alpha \log \log n} \right)^k
	[\log L(\log n) + (k+1) \log \log \log n]  \\
	& \hspace{10.9cm} {} + 1+\alpha +\log \alpha + o(1) \Biggr).
	\end{align*}
\end{theorem}

\begin{remark} 
	Keeping only the term $k=0$ in the expansion of $\haut(\ttT_n)$ in Theorem~\ref{thm:powers_of_log} above, we get
	\begin{align*}
	\haut(\ttT_n)
	& = \frac{\log n}{\alpha \log \log n} 
	+ \frac{\log n}{(\alpha \log \log n)^2} 
	\biggl( \log L(\log n) + \log \log \log n + O \left( \frac{(\log L(\log n))^2}{\log \log n} \right) \\
	& \hspace{10.2cm} {} + 1+\alpha +\log \alpha + o(1) \biggr).
	\end{align*}
	The second order term for $\haut(\ttT_n)$ can be the one involving $\log L(\log n)$ or the one involving $\log \log \log n$ (or none of those two if for example $L(x) = (\log \log x)^{-1}$). 
	Depending on $L$, the terms in the series in the expansion of $\haut(\ttT_n)$ can be included in the $o(1)$ for $k$ large enough, but not necessarily (take for example $L(x) = \exp(\frac{\log x}{\log \log \log \log x})$).
\end{remark}

\begin{remark} \label{rem:comparison_iid}
	Recall $\haut(\ttu_{n+1})$ has the same distribution as $\sum_{i=1}^{n} B_i$, where the $B_i$'s are independent Bernoulli r.v.\@ with parameter $w_i/W_i$ (see Remark \ref{rem:consequence_many-to-one}).
	In order to see the influence of the correlations due to the tree structure,
	it is meaningful to compare the height of $\ttT_n$ with the maximum of $n$ independent variables with the same distribution as $\sum_{i=1}^{n} B_i$%
	\footnote{It can seem crude to take $n$ copies of $\sum_{i=1}^{n} B_i$ instead of independent variables $Z_1,\dots,Z_n$ where $Z_k \overset{\text{(d)}}{=} \sum_{i=1}^{k} B_i$. But note that $\max_{1 \leq k \leq n} Z_k$ is stochastically larger than the maximum of $n/2$ independent copies of $\sum_{i=1}^{n/2} B_i$ and replacing $n$ by $n/2$ in all expansions considered does not change anything.\label{footnote:crude}}.
	In Section~\ref{sec:crude upper bound}, we explain why this maximum always provides an upper bound for $\haut(\ttT_n)$, we compare the expansion of this upper bound with that of the actual behavior of $\haut(\ttT_n)$ and identify the first term for which they differ. 
	In the framework of Theorem~\ref{thm:powers_of_log}, an interesting transition occurs:
	\begin{itemize}
		\item If $\alpha < 1$, the first terms in the expansion in the i.i.d.\@ case are the same as for $\haut(\ttT_n)$ up to the term of order $\log n/(\log \log n)^2$ for which the coefficient differs.
		\item If $\alpha = 1$, the first term is still the same but the difference between the expansions grows faster than $\log n/(\log \log n)^2$.
		\item If $\alpha > 1$, the first term differs by a multiplicative constant.
	\end{itemize}
	We decided to stop the expansion at the order $\log n/(\log \log n)^2$, because it was enough to see the difference with the case of i.i.d.\@ variables in all cases.
	However, our method (and in particular Proposition \ref{prop:upper_bound} and \ref{prop:lower_bound}) could give the expansion up to the order $\log n/(\log \log n)^k$ for any $k \geq 1$, if we assume a small enough error term in \eqref{eq:assumption_a_n_1} and that $L$ is sufficiently smooth (see Section \ref{sec:strategy_main_result} for details).
\end{remark}

\paragraph{Examples.} Let us apply the above result to specific sequences of weights.
The computations needed to ensure that those weight sequences indeed satisfy the assumptions of Theorem~\ref{thm:powers_of_log} are available in Section~\ref{sec:application} of the appendix.
\begin{itemize}
	\item \textbf{ For $w_n=\frac{\lambda (1-\alpha)}{n(\log n)^\alpha}\cdot \exp(\lambda (\log n)^{1-\alpha})$ with $\alpha \in \intervalleoo{0}{1}$ and $\lambda >0$}, we can check that \eqref{eq:ass_somme_carre} is satisfied and that \eqref{eq:assumption_a_n_1} holds with a function $L$ such that $L(x)=\lambda (1-\alpha)$ for $x$ large enough. Applying Theorem~\ref{thm:powers_of_log} then yields
	 \begin{align*}
	\haut(\ttT_n)
	& = \frac{\log n}{\alpha \log \log n} 
	+ \frac{\log n}{(\alpha \log \log n)^2} 
	\left(\log \log \log n  + 1+\alpha +\log (\lambda \alpha(1-\alpha)) + o(1) \right).
	\end{align*}
	This example of behavior for the weight sequence is similar to that of the (random) sequences appearing in \cite{mailleruribebravo2019,bocimailler2021} that are constructed using ``memory kernel'' $\mu_1$. See  
	\cite[Section~1.1]{mailleruribebravo2019} for more details.
	\item \textbf{For $w_n=\frac{1}{n}$}, we can check that \eqref{eq:ass_somme_carre} is satisfied and \eqref{eq:assumption_a_n_1} holds with a function $L$ such that $L(x)=1+o(1)$ as $x\rightarrow \infty$. Applying Theorem~\ref{thm:powers_of_log} yields 
	\begin{align*}
	\haut(\ttT_n)
	& = \frac{\log n}{\log \log n} 
	+ \frac{\log n}{(\log \log n)^2} 
	\left(\log \log \log n + 2  + o(1) \right).
	\end{align*}
	\item \textbf{For $w_n=\frac{1}{n(\log n)^\alpha}$ with $\alpha>1$}, we can check that \eqref{eq:ass_somme_carre} is satisfied and \eqref{eq:assumption_a_n_1} holds with a function $L$ such that $L(x)=\frac{1}{W_\infty}+o(1)$ as $x\rightarrow \infty$, where $W_\infty \coloneqq \lim_{n\rightarrow\infty}W_n$.
	Using Theorem~\ref{thm:powers_of_log} we get
	\begin{align*}
	\haut(\ttT_n)
	& = \frac{\log n}{\alpha \log \log n} 
	+ \frac{\log n}{(\alpha \log \log n)^2} 
	\left(\log \log \log n + 1+\alpha +\log \alpha- \log W_\infty  + o(1) \right). 
	\end{align*}
	Remark that in this case, changing the value of a finite number of weights in the sequence would affect the asymptotic behavior of the height at the third order because of the presence of the term containing $W_\infty$.  
\end{itemize}
The sequence of weights could also be chosen randomly. For example, let $w_n \coloneqq X_n \widetilde{w}_n$, where $(X_n)_{n\geq 1}$ is a sequence of i.i.d.\@ non-negative random variables and $(\widetilde{w}_n)_{n\geq 1}$ is one of the sequences in the examples above. Then, if $\E[X_1^2]< \infty$, one can check that almost surely \eqref{eq:assumption_a_n_1} and \eqref{eq:ass_somme_carre} are satisfied with the same function $L$ as above and therefore the conclusion on $\haut(\ttT_n)$ still holds a.s. 
In particular, in the third example above, the third order term containing $\log W_\infty$ becomes random.

\subsection{Quickly converging variance} \label{sec:quickly converging intro}

Here, we assume that $a_n$ converges to some finite limit $a_\infty$ and that
\begin{equation} \label{eq:assumption_a_n_2}
a_\infty - a_n = \exp \left( -(\alpha-1) \log^\beta n \right) 
	\cdot J\left( \exp \left( \log^\beta n \right) \right) 
	\cdot \left( 1 + \petito{(\log n)^{(4\beta-2)\wedge 0}} \right),
\end{equation}
where $\alpha > 1$, $\beta > 0$ and $J$ is a positive function on $[1,\infty)$ satisfying \hyperref[hyp_SV_k]{$(\mathrm{SV}_2)$} if $\beta <\frac{1}{2}$ and \ref{hyp_SV_0} otherwise.
The assumption required for $\sum_{i=n}^{\infty} w_i^2/W_i^2$ depends on $\beta$. When $\beta < 1$ we use the same assumption \eqref{eq:ass_somme_carre} as in the previous cases.
If $\beta = 1$, we will further assume
\begin{equation} \label{eq:ass_somme_carre_beta=1}
	\forall \varepsilon>0: \quad
	\sum_{i\geq n} \frac{w_i^2}{W_i^2} 
	= O \left( \frac{1}{n^{2\alpha-1-\varepsilon}} \right).
\end{equation}
Compared to \eqref{eq:ass_somme_carre}, this last display assumes a faster convergence to $0$ for $\sum_{i\geq n} w_i^2/W_i^2$, but note that in this regime weights are also decreasing faster: if $w_n$ was regular in $n$, we would have
$\sum_{i\geq n} w_i^2/W_i^2 = O(n^{1-2\alpha} J^2(n))$.
Finally, if $\beta >1$ we assume
\begin{equation} \label{eq:ass_somme_carre_beta>1}
	\exists \varepsilon>0: \quad 
	\sum_{i\geq n} \frac{w_i^2}{W_i^2} 
	= O \left( \frac{1}{n^\varepsilon} \exp \left( -2(\alpha-1)\log^\beta n \right) 
	\cdot J^2\left( \exp \left(\log^\beta n \right) \right) \right).
\end{equation}
Here, if $w_n$ was regular in $n$, we would have
a similar bound with $1/n^\varepsilon$ replaced by $(\log n)^{\beta-1}/n$.
Also note that, for some $J$ slowly varying, $J(\exp(\log^\beta n))$ can grow faster than any polynomial.
\begin{theorem} \label{thm:quickly}
	Assume that \eqref{eq:assumption_a_n_2} holds for some $\beta >0$, some $\alpha > 1$
	 and some function $J \colon [1,\infty) \to (0,\infty)$ satisfying \hyperref[hyp_SV_k]{$(\mathrm{SV}_2)$} if $\beta <\frac{1}{2}$ and \ref{hyp_SV_0} otherwise.
	Then, almost surely, as $n \to \infty$,
	\begin{enumerate}
		\item\label{it:asymptotic height beta in (0,1)} if $ \beta \in (0,1)$ and  \eqref{eq:ass_somme_carre} holds, we have
		\[
			\haut(\ttT_n) \sim 	\frac{(\log n)^{1-\beta}}{(\alpha -1)(1-\beta)};
		\]
	    \item\label{it:asymptotic height beta=1} if $ \beta=1$ and \eqref{eq:ass_somme_carre_beta=1} holds, we have
		\[
		\haut(\ttT_n) \sim\frac{\log \log n}{\log \alpha};
		\]
		\item\label{it:asymptotic height beta>1} if $\beta\in \intervalleoo{1}{\infty}$ and \eqref{eq:ass_somme_carre_beta>1} holds, we have
		\[
		\haut(\ttT_n) \sim \frac{\log \log \log n}{\log \beta}.
		\]
	\end{enumerate}
\end{theorem}

\begin{remark} \label{rem:comparison_iid_2}
	As in Remark~\ref{rem:comparison_iid}, we can compare $\haut(\ttT_n)$ with the maximum of $n$ independent copies of $\sum_{i=1}^{n-1} B_i$.
	In the framework of Theorem~\ref{thm:quickly}, $\haut(\ttT_n)$ is always much smaller than this maximum a.s., see Section~\ref{sec:crude upper bound}.
\end{remark}

\begin{remark}
We see that in the setting of Theorem~\ref{thm:quickly}, the faster the weight sequence converges, the slower the height of the tree grows to infinity.
Remark that unless the weight sequence is such that $w_i=0$ for all $i$ large enough, the height of $\ttT_n$ almost surely tends to infinity with $n$, irrespective of the behavior of $a_n$.
\end{remark}

\paragraph{Example.} Consider the weight sequence $w_n=n^{-\alpha}$ for $\alpha>1$. 
Then we can apply Theorem~\ref{thm:quickly} with $\beta=1$, see Section~\ref{sec:application} in the appendix where we check that the assumptions are satisfied, to get that almost surely as $n\rightarrow\infty$ we have
\[
\haut(\ttT_n) \sim\frac{\log \log n}{\log \alpha}.
\]

\subsection{Link with time-inhomogeneous branching random walk}
\label{subsec:BRW}

As mentioned at the beginning of the introduction, the WRT can be compared to a BRW and we make here this analogy more precise.
First note that, as explained in \cite[Section 1.4]{painsenizergues2022}, the WRT can be linked rigorously to a particular BRW with types, which does not satisfy the branching property and therefore is not tractable (or at least not easier to study than the WRT itself).
Hence, we believe it is not a good strategy to try to prove results on the WRT through a direct comparison with a BRW.
The goal here is different: we present a \textit{non-rigorous} analogy, but which provided useful insights for the proofs.

A branching random walk on the real line is a discrete-time process defined as follows. Initially there is one individual at position $0$, which forms the $0$th generation. 
For any $t \in \N$, each individual of the $t$th generation has a random number of children who jump independently of each other from their parent's location according to some fixed jump distribution and these children form the $(t+1)$th generation.
In the following, we also consider the case of time-inhomogeneous BRW, where the jump distribution can depend on the generation number (but not the reproduction law).


We denote by $V_t$ the set of particles in generation $t$ and for a particle $u\in V_t$ we denote $X_u$ its position.
We introduce the measure
\[
	\mu_t^{\mathrm{BRW}} \coloneqq \frac{1}{\abs{V_t}}\sum_{u\in V_t} \delta_{X_u},
\]
the empirical measure of the displacement of particles in generation $t$, so that in particular, the maximal displacement at time $t$ corresponds to the supremum of the support of $\mu_t^{\mathrm{BRW}}$. 
A property that will be important in the comparison with the WRT is the following: Sample a random variable $Z_{t+1}$ with distribution $\mu_{t+1}^{\mathrm{BRW}}$. 
Then, conditional on $\mu_t^{\mathrm{BRW}}$, the conditional distribution of $Z_{t+1}$ can be described by the following equality in law
\begin{equation}\label{eq:equality in law random point BRW}
	(Z_{t+1}\ \vert \ \mu_t^{\mathrm{BRW}} ) \overset{(d)}{=} Z_t + \Delta_{t+1},
\end{equation} 
where the two terms on the right-hand-side are independent and $Z_t\sim \mu_t^{\mathrm{BRW}}$ and $\Delta_{t+1}$ has the jump distribution corresponding to the transition from time $t$ to $t+1$.

Now, for a given $n$, we would like to compare the height of the tree $\ttT_n$ with the maximal displacement at time $t$ in a BRW, with mean number of children $e$, say. 
With this choice, the total number of particles at time $t$ in the BRW behaves roughly as $e^t$, so it is natural to restrict ourselves to times of the form $n= \lfloor e^t \rfloor$ (in what follows we omit the floor notation so as to keep expressions simple).
Now, we consider the measure 
\[
	\mu_t^{\mathrm{WRT}}=\frac{1}{W_{e^t}}\sum_{i=1}^{e^t}w_i \delta_{\haut(u_i)}
\]
which is the (weighted) empirical distribution of height in the tree $\ttT_{e^t}$.
If we sample a random variable $H_{t+1}$ with distribution $\mu_{t+1}^{\mathrm{WRT}}$ then we obtain the following equality in distribution\footnote{The term involving a sum of independent Bernoulli r.v. arises from applying the many-to-one formula (Lemma~\ref{lem:many_to_one}) to the tree $\ttT_{e^{t+1}}^{(e^t)}$ where the $e^t$ first vertices are merged together (see definition in Section~\ref{sec:some_definitions}), conditionally on the tree $\ttT_{e^t}$. \vspace{-0.05cm}
The fact that the first term has distribution $\mu_t^{\mathrm{WRT}}$ is obtained by considering the distribution of $\ttT_{e^{t+1}}$ conditional on $\ttT_{e^{t}}$ and $\ttT_{e^{t+1}}^{(e^t)}$.}
\begin{equation}\label{eq:equality in law random point WRT}
(H_{t+1}\ \vert \ \mu_t^{\mathrm{WRT}} ) \overset{(d)}{=} H_t + \sum_{i=e^t+1}^{e^{t+1}} B_i,
\end{equation} 
where all the terms on the right-hand-side are independent with $H_t\sim \mu_t^{\mathrm{WRT}}$ and the $B_i$'s have respective distribution $\mathrm{Bernoulli}\left(w_i/W_i\right)$.
From the similarity between \eqref{eq:equality in law random point BRW} and \eqref{eq:equality in law random point WRT} it is natural for our comparison to fix the jump distribution between time $t$ and time $t+1$ to be close to that of 
\begin{equation} \label{eq:approximate_jumps_BRW}
\sum_{i=e^t+1}^{e^{t+1}} B_i
\overset{(d)}{\simeq} \mathrm{Poisson}\left( \sum_{i=e^t+1}^{e^{t+1}} \frac{w_i}{W_i} \right)
= \mathrm{Poisson}\left( a_{e^{t+1}} - a_{e^t} \right),
\end{equation}
at least for $t$ large enough so that the Poisson approximation can be justified.

To summarize the analogy: 
the (weighted) empirical distribution of heights for vertices in $\ttT_n$ can be compared with empirical distribution of  the positions of particles at time $t = \log n$ in a BRW with mean number of children $e$ which jumps between time $s$ and $s+1$ with distribution $\mathrm{Poisson}(a_{e^{s+1}}- a_{e^s})$.
In particular, the height of $\ttT_n$ should behave like the maximal displacement in the BRW.
One can expect this comparison to be more precise than the one with independent copies of the $\sum_{i=1}^n B_i$, because it relies not only on a many-to-one formula, but also on the hierarchical construction of the model.
Indeed, it is precise enough to predict all the expansions in Theorem~\ref{thm:powers_of_log} and Theorem~\ref{thm:quickly}.


In the case treated in \cite{painsenizergues2022} where $W_n$ grows polynomially, we have $a_n \simeq \gamma \log n$ and therefore the jumps approximately have distribution Poisson($\gamma$), which does not depend on $s$ in this case.
This is a realistic approximation: the asymptotic expansion of $\haut(\ttT_n)$ obtained in \cite{painsenizergues2022}, see \eqref{eq:height_polynomial_case}, is exactly the same as the one of the maximal position at time $t = \log n$ in a binary BRW with jump distribution Poisson($\gamma$).

In the framework of Section~\ref{sec:powers_of_log_intro} with $L=1$, it follows from \eqref{eq:approximate_jumps_BRW} that jumps at time $s$ in the BRW have approximately a Poisson$(s^{-\alpha})$ distribution, so the BRW becomes time-inhomogenous with decreasing variance.
To our knowledge, this time-inhomogeneous BRW does not fit into the framework studied so far in the literature, that we review briefly below.

A model of time-inhomogeneous BRW has been introduced by Bovier and Kurkova \cite{bovierkurkova2004} and is defined as follows. Fix a profile of variance $\sigma^2 \colon [0,1] \to (0,\infty)$. Then, for each given horizon $t$, consider the binary BRW where jumps at time $s$ have distribution $\mathcal{N}(0,\sigma^2(s/t))$.
They showed that the first order for the maximal position at time $t$ is linear in $t$ with an explicit speed in terms of $\sigma^2$. 
In particular, if $\sigma^2$ is decreasing, the speed 
is strictly smaller than the one of the maximum of $2^t$ independent random walks of length $t$ with jumps $\mathcal{N}(0,\sigma^2(s/t))$ at time $s$.
This is therefore called the strongly correlated case, since the branching structure has an effect even on the first order of the maximum.
In comparison, if $\sigma^2$ is non-decreasing, the speed is the same as for independent random walk (this is the weakly correlated case).

In the case of a decreasing variance profile, the next terms in the expansion of the maximum have been studied in
\cite{fangzeitouni2012b,maillardzeitouni2016} in a continuous-time setting and in \cite{mallein2015} for non-Gaussian jumps and also time-inhomogeneous reproduction law.
Fang and Zeitouni \cite{fangzeitouni2012b} showed that the second term is negative and of order $t^{1/3}$ and Mallein \cite{mallein2015} identified it precisely. This is much larger than the logarithmic correction appearing for the maximum of classical BRW.
Maillard and Zeitouni \cite{maillardzeitouni2016} proved tightness of the maximum after centering by an additional logarithmic term.
Tightness around the median had been obtained previously by Fang \cite{fang2012} in a general setting.
Other results have been proved in the case where $\sigma^2$ takes only two values \cite{fangzeitouni2012a,bovierhartung2014} and in the weakly correlated case \cite{bovierhartung2015}.

Now we come back to our framework with a BRW with 
Poisson$(s^{-\alpha})$ jumps.
For any fixed horizon $t$, one can rewrite the law of the jump at time $s$ as  Poisson$(t^{-\alpha} \cdot \left( s/t \right)^{-\alpha})$.
Hence, two effects can be expected here: a first one, coming from the inhomogeneity in the variance profile that is given by the function
$u \in [0,1] \mapsto u^{-\alpha}$ and
a second one, coming from the overall reduction of the variance by a factor $t^{-\alpha}$.
If the jumps had law $\mathcal{N}(0,t^{-\alpha} \cdot \left( s/t \right)^{-\alpha})$, a scaling argument would allow us to compare with the previously mentioned case and the maximum should be of order $t^{-\alpha/2} \cdot t$ for $\alpha < 2$%
\footnote{In the literature, the function $\sigma$ is not allowed to diverge at 0, but as long as it is integrable the result for the first order of the maximum should hold.}.
In the case of Poisson$(s^{-\alpha})$ jumps and mean number of children $e$, Theorem \ref{thm:powers_of_log} suggests that the first order of the maximum should be $t/(\alpha \log t)$.
The reason for this different behavior is that the overall reduction of the variance forces the maximal particles to behave in a very large deviations regime, for which there is no universality.
In particular, even if we are in the strongly correlated case, the difference with the maximum of independent random walks does not necessarily appear at the first order, see Remark~\ref{rem:comparison_iid}.

\subsection{Organization of the paper}
The paper is organized as follows:
In Section~\ref{sec:many-to-few} we introduce some notation and recall two results that hold in all generality for WRTs, a ``many-to-one'' and a ``many-to-two'' lemmas, which allow us to respectively re-express first and second moment estimates on some key quantities as probabilities of events involving an inhomogeneous random walk with Bernoulli increments.
Then, we use the many-to-one lemma to get some first moment estimate in Section~\ref{sec:upper-bound}, which leads to a general upper bound on the height of the tree, stated in Proposition~\ref{prop:upper_bound}.
In Section~\ref{sec:lower-bound}, we then use the many-to-two lemma to get some second moment estimate, which yields the general lower bound on the height of the tree contained in Proposition~\ref{prop:lower_bound}. 
Finally in Section~\ref{sec:proof of the main results}, we prove Theorem~\ref{thm:powers_of_log} and Theorem~\ref{thm:quickly} by applying Proposition~\ref{prop:upper_bound} and Proposition~\ref{prop:lower_bound} to the particular behavior of the weight sequences that we consider. 
The strategy for the proof is presented in Section~\ref{sec:strategy_upper_bound}, Section~\ref{sec:strategy_lower_bound} and Section~\ref{sec:strategy_main_result}.

In the appendix, Section~\ref{sec:crude upper bound} considers crude upper bounds on $\haut(\ttT_n)$ that one can make using the many-to-one lemma in a naive way and compares those to the actual behavior of $\haut(\ttT_n)$ obtained in Theorem~\ref{thm:powers_of_log} and Theorem~\ref{thm:quickly}.
Section \ref{sec:crude lower bound} presents a simple method to obtain lower bounds, at least when the sequence of weights is regular enough, which provides the correct first order in the regimes considered in this paper.
In Section~\ref{sec:slowly varying}, we prove a few technical results concerning slowly varying functions. 
Section~\ref{sec:correspondence behavior an and log Wn} contains the proof of \eqref{eq:transfer assum Wn an} which is helpful for checking that the assumption of our results are satisfied for particular weight sequences.
Section~\ref{sec:application} contains some computations ensuring that we can apply our results to the examples presented in the introduction.

\section{Definitions and many-to-few lemmas}
\label{sec:many-to-few}
In this section we recall some definitions and some notation concerning WRTs and state two useful lemmas: the many-to-one lemma (Lemma~\ref{lem:many_to_one}) and the many-to-two lemma (Lemma~\ref{lem:many-to-two}). 
\subsection{Some definitions}
\label{sec:some_definitions}

%

\paragraph{Labels and ancestors of a vertices.} 
For any $\ttu \in \ttT_n$, we write $\mathrm{lab}(\ttu)$ for the label of vertex $\ttu$ in the tree $\ttT_n$, which is an integer between $1$ and $n$.
For any $k\leq n $ we write $\ttu(k)$ for the most recent ancestor of $\ttu$ that has label smaller or equal to $k$. 
For any $\ttu,\ttv \in \ttT_n$, we denote $\ttu\wedge \ttv$ the most recent common ancestor of $\ttu$ and $\ttv$ in the tree $\ttT_n$.  


\paragraph{An auxiliary tree.} 
For any integer $N\geq 1$, we construct a new tree $\ttT_n^{(N)}$ from $\ttT_n$: we first remove all vertices with labels 2 through $N$ and then attach all of them and all of their children to the root.
Note that $\ttT_n^{(N)}$ has distribution $\wrt((w_i^{(N)})_{i\geq 1})$, 
where the sequence of weights $(w_i^{(N)})_{i\geq 1}$ is related to the sequence $(w_i)_{i\geq 1}$ as follows:
\begin{align} \label{eq:def_w^(N)}
w_i^{(N)}=
\begin{cases}
W_N & \text{if} \quad i=1,\\
0 & \text{if} \quad 2\leq i\leq N,\\
w_i & \text{if} \quad i\geq N+1.
\end{cases}
\end{align}
In other words, the sequence $(w_i^{(N)})_{i\geq 1}$ is obtained from $(w_i)_{i\geq 1}$ by transferring all the weight of vertices labelled $2$ through $N$ to the root and leaving the rest unchanged. 
We also write $W_n^{(N)} \coloneqq \sum_{i=1}^n w_i^{(N)}$.


\subsection{Many-to-few lemmas}

In this section we state the many-to-few lemmas, which allow to compute first and second moments of functionals of the heights along the lineage of a vertex chosen in $\ttT_n$ according to its weight.
They have been proved in \cite{painsenizergues2022}, using a specific construction of $\ttT_n$ coupled with two distinguished vertices introduced by Mailler and Uribe Bravo \cite[Proposition~9]{mailleruribebravo2019}.
\begin{lemma}[Many-to-one] \label{lem:many_to_one}
	For any function $F: \N^n\rightarrow \R$ we have
	\begin{align*}
	\Ec{\sum_{i=1}^{n}\frac{w_i}{W_n} \cdot F(\haut(\ttu_i(1)),\haut(\ttu_i(2)),\dots ,\haut(\ttu_i(n)))} = \Ec{F(H_1,H_2,\dots,H_n)},
	\end{align*}
	where, we defined for $j \geq 1$, $H_j=\sum_{i=2}^j B_i$ with $(B_i)_{i\geq 2}$ a sequence of independent random variables such that $B_i$ has distribution $\mathrm{Bernoulli}(w_i/W_i)$.
\end{lemma}
\begin{proof}
	This is Lemma~2.3 in \cite{painsenizergues2022} with $\theta = 0$ (note that $\theta$ has a specific non zero value in \cite{painsenizergues2022}, but all results in Section~2 of \cite{painsenizergues2022} stay true for any $\theta \in \R$, as mentioned in Remark 2.2 there).
\end{proof}
\begin{remark} \label{rem:consequence_many-to-one}
	One consequence of the many-to-one lemma, mentioned in the introduction, is that $\haut(\ttu_{n+1})$ has the same distribution as $H_{n}+1$. Indeed, $\ttu_{n+1}$ is attached to a vertex $\ttu_i$, for $1 \leq i \leq n$, chosen proportionally to its weight $w_i$ and in that case $\haut(\ttu_{n+1}) = \haut(\ttu_i(n))+1$.
	Therefore, for any test function $f \colon \N \to \R$, we have 
	\[
		\Ec{f(\haut(\ttu_{n+1}))} 
		= \Ec{\sum_{i=1}^{n} \frac{w_i}{W_n} \cdot f(\haut(\ttu_i(n))+1)}
		= \Ec{f(H_{n}+1)},
	\]
	where we used the many-to-one lemma in the last equality.
\end{remark}
\begin{lemma}[Many-to-two] \label{lem:many-to-two}
	For any $\ell \geq 1$, we introduce a sequence $(B_i^\ell,\overline{B}\vphantom{B}_i^\ell)_{i\geq 2}$ of independent couples of random variables such that
	\begin{itemize}
		\item if $i < \ell$, $B_i^\ell$ has distribution Bernoulli$(w_i/W_i)$ and $\overline{B}\vphantom{B}_i^\ell=B_i^\ell$;
		\item if $i = \ell$, $B_i^\ell=\overline{B}\vphantom{B}_i^\ell=1$;
		\item if $i > \ell$, $B_i^\ell$ and $\overline{B}\vphantom{B}_i^\ell$ have the same  distribution as two independent Bernoulli$(w_i/W_i)$ random variables conditioned not to be both equal to 1.
	\end{itemize}
	Then, for any $n\geq 1$ and any functions $F \colon \N^n \to \R$ and $f \colon \intervalleentier{1}{n} \to \R$, we have 
	\begin{align*}
	& \Ec{\sum_{1\leq i,j\leq n}
		\frac{w_i w_j}{W_n^2} \cdot f(\mathrm{lab}(\ttu_i\wedge \ttu_j)) 
		\cdot F(\haut(\ttu_i(1)),\dots ,\haut(\ttu_i(n)))
		\cdot F(\haut(\ttu_j(1)),\dots ,\haut(\ttu_j(n)))} \notag\\
	& = \sum_{\ell=1}^n \frac{w_\ell^2}{W_\ell^2} 
	\cdot \left( \prod_{i=\ell+1}^{n} \left(1-\frac{w_i^2}{W_i^2} \right) \right) \cdot f(\ell)\cdot 
	\Ec{ F(H^\ell_1,\dots , H^\ell_n)\cdot  F(\overline{H}\vphantom{H}_1^\ell, \dots, \overline{H}\vphantom{H}^\ell_n)},
	\end{align*}
	where $H_j^\ell=\sum_{i=2}^j B_i^\ell$ and $\overline{H}\vphantom{H}_j^\ell=\sum_{i=2}^j \overline{B}\vphantom{B}_i^\ell$ for $j \geq 1$.
\end{lemma}

\begin{proof}
	This is Lemma~2.4 in \cite{painsenizergues2022} with $\theta = 0$.
	In the reference, for $i > \ell$, the distribution of $(B_i^\ell,\overline{B}\vphantom{B}_i^\ell)$ is described as follows, with $q_i \coloneqq w_i/W_i$: the random variable 
	$B_i^\ell$ has distribution Bernoulli$(q_i/(1+q_i))$ and, given $B_i^\ell$, the other one $\overline{B}\vphantom{B}_i^\ell$ has distribution $\mathrm{Bernoulli}(q_i \1_{\{B_i^\ell = 0\}})$.
	One can check that this is equivalent to the description given in the statement of the lemma.
\end{proof}

\section{Upper bound for the height}\label{sec:upper-bound}

In this section, we provide a criterion for obtaining an upper bound for the height of $\ttT_n$ for general weight sequences $(w_i)_{i\geq 1}$. This is the content of Proposition~\ref{prop:upper_bound}. 

\subsection{Strategy}\label{sec:strategy_upper_bound}

We work with an increasing sequence of integers $(i_r)_{r\geq 0}$ with $i_0=1$. 
We define an associated nondecreasing sequence of integers $(t_n)_{n\geq 1}$ as follows: for $n \geq 1$, $t_n$ is the unique integer $r \geq 0$ such that $n\in \intervalleentier{i_{r-1}+1}{i_{r}}$, with the convention $i_{-1} = 0$.
In particular, $t_{i_r} = r$ for any $r \geq 0$.

This sequence $(i_r)_{r\geq 0}$ has to be thought as a time change such that $\haut(\ttT_{i_r}) \simeq r$. Then, proving that $\haut(\ttT_{i_r}) \leq r$ for $r$ large enough implies that $\haut(\ttT_n) \leq t_n$ for $n$ large enough (at least when $n=i_{t_n}$, but one can fill the gaps easily by monotonicity).
The sequence $(i_r)_{r\geq 0}$ will be chosen explicitly for each specific case in Section \ref{sec:proof of the main results}.
It has to be chosen so that $\haut(\ttT_{i_r})$ is sufficiently smaller than $r$ to allow the argument to work, but as close to $r$ as possible to get the best upper bound possible.

The upper bound relies on a first moment calculation on the number of high vertices.
However, we need first to introduce a barrier controlling the height along the ancestral line of these high vertices, otherwise this first moment calculation does not give the desired upper bound, see Section~\ref{sec:crude upper bound}.
This type of barrier argument is classical in the BRW literature and is used in the time-inhomogeneous setting, see e.g. \cite{mallein2015,maillardzeitouni2016}.
A reasonable barrier here consists in enforcing that the ancestor $\ttu_m(i_r)$ of a high vertex $\ttu_m$ has a height at most $r$, since we expect that $\haut(\ttT_{i_r}) \leq r$.
Therefore, we introduce the following quantity, for $n \geq 1$,
\begin{equation} \label{eq:def_Q_n}
Q_n \coloneqq \sum_{m=1}^n \frac{w_m}{W_n} \1_{\haut(\ttu_m) = t_n} \1_{\forall r \in \llbracket 0, t_n \llbracket,\, \haut(\ttu_m(i_r)) \leq r},
\end{equation}
which is the \textit{weighted} number of vertices of height $t_n$ in $\ttT_n$, whose ancestral line respected the barrier constraint.
Working with this weighted version is more convenient when applying the many-to-few lemmas.

The first moment of $Q_n$ is estimated in Lemma \ref{lem:bound first moment Q_n}: by the many-to-one lemma, it equals
\begin{equation} \label{eq:proba}
\Pp{H_n = t_n \quad \text{and} \quad \forall r \in \llbracket 0, t_n \llbracket, H_{i_r} \leq r}.
\end{equation}
The walk $(H_k)_{k\geq 1}$ on this event typically stays close to the barrier (at first order). 
This behavior becomes typical after the following time-inhomogeneous change of measure: the distribution of $B_i$ for $i_{r-1} < i \leq i_r$ is biased by $e^{\theta_r B_i}$, where $\theta_r$ is chosen such that $H_{i_r} -H_{i_{r-1}}$ has mean 1 under the new measure.
Actually, we do not need to specify precisely the choice of this sequence $(\theta_r)_{r\geq 1}$ in this section, but in Section \ref{sec:proof of the main results}, we always choose $\theta_r = - \log (a_{i_r} - a_{i_{r-1}})$ for $r$ large enough, which ensures that $H_{i_r} -H_{i_{r-1}}$ has approximately distribution Poisson$(1)$.
It is helpful to keep this in mind while reading the proofs.

Finally, a criterion for the upper bound is established in Proposition \ref{prop:upper_bound}.
Heuristically, the mean (non-weighted) number of vertices of height $t_n$ in $\ttT_n$ satisfying the barrier constraint is roughly $n \E[Q_n]$ and the criterion says that, if $i_t \E[Q_{i_t}]$ is summable in $t$, then $\haut(\ttT_n) \leq t_n$ for large $n$.

\subsection{First moment estimate}

\begin{lemma}\label{lem:bound first moment Q_n}
	Let $(\theta_r)_{r\geq 1}$ be a non-decreasing sequence of non-negative numbers and 
	\begin{equation} \label{eq:def_p_i}
	p_i \coloneqq \frac{e^{\theta_{t_i}} \frac{w_i}{W_i}}{1+(e^{\theta_{t_i}} -1)\frac{w_i}{W_i}}, \qquad i \geq 2.
	\end{equation}
	Let $n \geq 1$.
	Let $(Y_i)_{i\geq 2}$ be a sequence of independent Bernoulli$(p_i)$ random variables.
	We set $S_0 \coloneqq 0$ and, for $r \geq 1$,
	\begin{equation} \label{eq:def_S_r}
	X_r \coloneqq \sum_{i=i_{r-1}+1}^{i_{r}\wedge n} Y_i
	\qquad \text{and} \qquad 
	S_r \coloneqq \sum_{s = 1}^r (X_s-1).
	\end{equation}
	Then, we have
	\begin{align} 
	\E \left[ Q_n \right] 
	& = \left( \prod_{i=2}^{n} \left( 1+(e^{\theta_{t_i}} -1)\frac{w_i}{W_i} \right)  \right) 
	\exp\left(- \sum_{r=1}^{t_n} \theta_r  \right) 
	\Ec{\exp\left( \sum_{r=1}^{t_n-1} (\theta_{r+1}-\theta_r) S_r \right)
		\1_{S_{t_n} = 0, \, \forall r < t_n, S_r \leq 0}} \nonumber \\
	&\leq \exp \left( \sum_{i=2}^n (e^{\theta_{t_i}} -1) \frac{w_i}{W_i} - \sum_{r=1}^{t_n} \theta_r  \right). \label{eq:first_moment_Q_n}
	\end{align}
\end{lemma}
The calculation relies on a time-inhomogeneous change of measure similar to \cite[Lemma 3.3]{mallein2015} but here we do not need to estimate precisely the expectation appearing in the middle part of \eqref{eq:first_moment_Q_n}, since we are not aiming at the same level of precision.
\begin{proof}
	It follows from the many-to-one lemma (Lemma~\ref{lem:many_to_one}) that
	\begin{align*}
	\E \left[ Q_n \right] 
	& =\P \left( \sum_{i=2}^n B_i = t_n, \, \forall r < t_n, \sum_{i=2}^{i_r} B_i \leq r \right),
	\end{align*}
	where the $B_i$'s are independent Bernoulli r.v. with parameter $w_i/W_i$.
	Now note that $Y_i$ has the distribution of $B_i$ biased by $e^{\theta_{t_i} B_i}$, that is
	\[
	\forall \phi \colon \{0,1\} \to \R, \qquad 
	\Ec{\phi(Y_i)} 
	= \frac{\E[\phi(B_i) e^{\theta_{t_i} B_i}]}{\E[e^{\theta_{t_i} B_i}]}.
	\]
	Therefore, we have
	\begin{align*}
	\E \left[ Q_n \right] 
	= \Ec{\exp\left( \sum_{i=2}^{n} \theta_{t_i} B_i \right)}
	\Ec{\exp\left( -\sum_{i=2}^{n} \theta_{t_i} Y_i \right)
		\1_{\sum_{i=1}^n Y_i = t_n, \forall r < t_n, \sum_{i=1}^{i_r} Y_i \leq r}}.
	\end{align*}
	Then, we can rewrite
	\begin{align*}
	\sum_{i=2}^{n} \theta_{t_i} Y_i
	= \sum_{r=1}^{t_n} \theta_r \sum_{i=i_{r-1}+1}^{i_{r}\wedge n} Y_i
	= \sum_{r=1}^{t_n} \theta_r  + \sum_{r=1}^{t_n} \theta_r (X_r-1)
	= \sum_{r=1}^{t_n} \theta_r + \sum_{r=1}^{t_n-1} (\theta_r-\theta_{r+1}) S_r + \theta_{t_n} S_{t_n},
	\end{align*}
	using a summation by part and the fact that $S_0=0$.
	Hence, using that we are on the event $S_{t_n} = 0$, we get the equality in the lemma.
	Bounding the expectation by 1 because $(\theta_r)_{r\geq 1}$ is non-decreasing, and using the inequality $1+x \leq e^x$, the inequality follows.
\end{proof}

\subsection{General criterion for the upper bound}

\begin{proposition} \label{prop:upper_bound}
	Assume that is given a sequence $(\theta_r)_{r\geq 1}$ such that, for some $r_0 \geq 1$, $(\theta_r)_{r\geq r_0}$ is a non-decreasing sequence of non-negative numbers
	and such that the sequence
	\begin{align} \label{eq:ass_upper_bound}
	i_t \cdot \exp \left( \sum_{i=2}^{i_t} (e^{\theta_{t_i}} -1) \frac{w_i}{W_i} - \sum_{r=1}^t \theta_r \right) 
	\end{align}
	is summable in $t\geq 1$.
	Then, $\sup_{n\geq 1}(\haut(\ttu_{n})- t_n)<\infty$ almost surely.
\end{proposition}

\begin{proof}
For $n \geq 1$, first note that $\haut(\ttu_{n+1}) = \haut(\ttu_{n+1}(n))+1$, since by definition the vertex $\ttu_{n+1}$ is directly attached to the vertex $\ttu_{n+1}(n)$.
We write
\begin{align} \label{eq:inital_proba}
\{\exists n\geq 1,\ \haut(\ttu_{n+1})> t_{n+1}\} 
&= \bigcup_{n\geq 1}\{ \haut(\ttu_{n+1})= t_{n+1} +1 \text{ and } \forall i\leq n, \ \haut(\ttu_{n+1}(i))\leq t_i\},
\end{align}
where we used that if $\haut(\ttu_{n+1}(n)) \leq t_n$ then $\haut(\ttu_{n+1}) \leq t_n+1 \leq t_{n+1} +1$.
Given $\ttT_n$, vertex $\ttu_{n+1}$ is attached to $\ttu_m$ with probability $w_m/W_n$ for $1 \leq m \leq n$, and in that case $\haut(\ttu_{n+1}) = \haut(\ttu_m)+1$: therefore, we have
\begin{align} 
	\Pp{\haut(\ttu_{n+1})= t_{n+1} +1 \text{ and } \forall i\leq n,\ \haut(\ttu_{n+1}(i))\leq t_i}
	& = \Ec{ \sum_{m=1}^n \frac{w_m}{W_n} \1_{\haut(\ttu_m) = t_{n+1}} \1_{\forall i\leq n, \haut(\ttu_m(i)) \leq t_i} }. \label{eq:rewritting_proba}
\end{align}
If $n=i_t$, then $t_{n+1} = t_n +1$ and so the last expectation is zero, because the indicator functions both require $\haut(\ttu_m) = t_{n+1}$ and $\haut(\ttu_m) = \haut(\ttu_m(n)) \leq t_n$.
If $n\neq i_t$, then $t_{n+1} = t_n$ and, noting that
\[
	\{ \haut(\ttu_m) = t_n \text{ and } \forall i\leq n, \ \haut(\ttu_m(i)) \leq t_i\}
	= \{ \haut(\ttu_m) = t_n \text{ and } \forall r<t_n, \ \haut(\ttu_m(i_r)) \leq r\},
\]
it shows that the quantities displayed in \eqref{eq:rewritting_proba} are equal to $\Ec{Q_n}$.
Coming back to \eqref{eq:inital_proba} and using a union-bound, we hence get
\begin{align} \label{eq:upper bound proba ht un greater than xn}
	\Pp{\sup_{n\geq 1}(\haut(\ttu_{n})- t_n) >0}\leq \sum_{n\geq 1}\Ec{Q_n}.
\end{align}

Now we fix some large $N \geq 1$ and consider the trees $(\ttT_n^{(N)})_{n\geq 1}$ introduced in Section \ref{sec:some_definitions}. 
Let $Q_n^{(N)}$ be defined as in \eqref{eq:def_Q_n} but for the tree $\ttT_n^{(N)}$. 
For all $n\geq1$, we have $\haut(\ttu_n)\leq \haut(\ttu_n^{(N)})+N$, so
\begin{align} \label{eq:upper bound proba ht un greater than xn + N}
	\Pp{\sup_{n\geq 1}\left(\haut(\ttu_n)-t_n\right)> N}
	\leq \Pp{\sup_{n\geq 1}\left(\haut(\ttu_n^{(N)})-t_n\right)> 0}
	\leq \sum_{n\geq 1}\Ec{Q_n^{(N)}}
\end{align}
where the second inequality is obtained by applying \eqref{eq:upper bound proba ht un greater than xn} to the sequence $(\ttT_n^{(N)})_{n\geq 1}$ which is a weighted recursive tree with weight sequence $(w_n^{(N)})_{n\geq 1}$.
Now, going back to its definition \eqref{eq:def_Q_n}, we can check that $Q_n^{(N)}=0$ for any $2\leq n \leq N$. Indeed all the terms in the sum vanish: we have $t_n=1$, so the term corresponding to $m=1$ is zero because of the indicator $\ind{\haut(\ttu_1^{(N)})=1}$; the other ones vanish because $w_m^{(N)}=0$ for $2\leq m\leq n$. 
On the other hand, letting $\tilde{\theta}_r \coloneqq \theta_r \1_{r \geq r_0}$, so that 
$(\tilde{\theta}_r)_{r\geq 0}$ is a non-decreasing sequence of non-negative numbers, we apply Lemma~\ref{lem:bound first moment Q_n} to get
\begin{align*}
\E \left[ Q_n^{(N)} \right] 
\leq \exp \left( \sum_{i=2}^n (e^{\tilde{\theta}_{t_i}} -1) \frac{w_i^{(N)}}{W_i^{(N)}} 
- \sum_{r=1}^{t_n} \tilde{\theta}_r \right)
\leq \exp \left( \sum_{i=i_{r_0-1}+1}^{n} (e^{\theta_{t_i}} -1) \frac{w_i}{W_i} 
- \sum_{r=r_0}^{t_n} \theta_r  \right),
\end{align*}
using that, for all $i\geq 2$, we have $w_i^{(N)}/W_i^{(N)} \leq w_i/W_i$.
In the end, using \eqref{eq:upper bound proba ht un greater than xn + N} and the discussion that follows we get that 
\begin{align*}
\Pp{\sup_{n\geq 1} \left(\haut(\ttu_n)-t_n\right)> N} \leq \sum_{n\geq N+1}\Ec{Q_n^{(N)}}
& \leq \sum_{n\geq N+1} \exp \left( \sum_{i=i_{r_0-1}+1}^{n} (e^{\theta_{t_i}} -1) \frac{w_i}{W_i} 
- \sum_{r=r_0}^{t_n} \theta_r \right) \\
& \leq \sum_{t\geq t_{N+1}-1} 
i_t \cdot \exp \left( \sum_{i=i_{r_0-1}+1}^{i_t} (e^{\theta_{t_i}} -1) \frac{w_i}{W_i} 
- \sum_{r=r_0}^{t} \theta_r  \right),
\end{align*}
by grouping terms.
The summand in the last sum differs from \eqref{eq:ass_upper_bound} only by a finite factor independent of $t$ and therefore is summable by assumption of the lemma.
Hence, this last sum tends to $0$ as $N \to \infty$, because $t_{N+1} \to \infty$. This finishes the proof.
\end{proof}

\section{Lower bound for the height}\label{sec:lower-bound}

In this section, we provide a criterion for obtaining a lower bound for the height of $\ttT_n$ for general weight sequences $(w_i)_{i\geq 1}$. This is the content of Proposition~\ref{prop:lower_bound}. 
We additionally prove Lemma~\ref{lem:E(n)}, which provides us with some estimates on one of the quantities that appear in the assumptions of Proposition~\ref{prop:lower_bound}. 

\subsection{Strategy}\label{sec:strategy_lower_bound}

As for the upper bound, we work with an increasing sequence of integers $(i_r)_{r\geq 0}$ with $i_0=1$, which provides a time change such that $\haut(\ttT_{i_r}) \simeq r$.
Again, the sequence $(i_r)_{r\geq 0}$ will be chosen explicitly for each specific case in Section \ref{sec:proof of the main results}.
The main difference for the lower bound is that it has to be chosen so that $\haut(\ttT_{i_r})$ is slightly larger than $r$ (rather than slightly smaller for the upper bound).
This ensures that there are typically a large number of vertices at height $r$ in $\ttT_{i_r}$ and that we can expect some concentration in a first and second moment argument.

As in Section \ref{sec:upper-bound}, we consider the sequence of integers $(t_n)_{n\geq 1}$ such that for any $n\geq 1$ we have $n \in \intervalleentier{i_{t_n-1}+1}{i_{t_n}}$, and the quantity $Q_n$ defined in \eqref{eq:def_Q_n}.
We first establish a lower bound for $\E[Q_n]$ in Lemma \ref{lem:lower_bound_first_moment_Q_n} and an upper bound for $\E[Q_n^2]$ in Lemma \ref{lem:second_moment_Q_n}.
These bounds show that $\E[Q_n^2]$ and $\E[Q_n]^2$ are of the same order.
This implies that $\P(Q_n > 0) \geq c > 0$ for $n$ large enough, which imply in particular that $\P(\haut(\ttT_n) \geq t_n) \geq c$.
However, as for the case of polynomially growing $W_n$ treated in~\cite{painsenizergues2022}, we emphasize that we cannot proceed here as for the BRW and use the branching property at some large fixed time to conclude that $\haut(\ttT_n) \geq t_n-C$ with high probability.
Indeed, subtrees rooted at the first vertices of the WRT do not behave independently.

Instead, we use the same method as in \cite{painsenizergues2022}: we apply this moment calculation to the auxiliary tree $\ttT_n^{(N)}$, whose height provides a lower bound for $\haut(\ttT_n)$.
For large $N$, the benefit is the following: in $\ttT_n^{(N)}$, two particles contributing to $Q_n$ chosen independently have w.h.p.\@ the root as most recent common ancestor.
This results from two facts: 
(i) the barrier prevents from having too large groups of close cousins contributing to $Q_n$, so that the most recent common ancestor has to be close to the root, 
(ii) in $\ttT_n^{(N)}$, the root has a much larger weight than each of its close descendants, so it is unlikely that one of this descendant has two children with progeny contributing to $Q_n$ (this mimics the branching property argument in a rigorous way).
This guarantees that $\E[Q_n^2]$ and $\E[Q_n]^2$ are now asymptotically equivalent and therefore $\P(Q_n > 0) \to 1$.

The main differences with \cite{painsenizergues2022} are the following:
(i) the second moment calculation requires a time-inhomogeneous change of measure as in Section \ref{sec:upper-bound},
(ii) here we do not want to estimate precisely the remaining expectations involving the random walk $S_n$ (as the one appearing for the first moment),
(iii) because of the latter, $N$ has to grow appropriately with $n$, whereas it was fixed in \cite{painsenizergues2022}.


\subsection{Moments calculation}

The following lower bound for $\E[Q_n]$ is a consequence of Lemma~\ref{lem:bound first moment Q_n}. The sum $\sum_{i=2}^{n} e^{2\theta_{t_i}} \frac{w_i^2}{W_i^2}$ has to be thought as an error term. For now, we do not aim at estimating the expectation involving the walk $S$ appearing in the statement: although it plays a negligible role, it has to be handled carefully.
\begin{lemma} \label{lem:lower_bound_first_moment_Q_n}
	Let $(\theta_r)_{r\geq 1}$ be a non-decreasing sequence of non-negative numbers.
	Let $n \geq 1$ and $(S_r)_{r\geq 0}$ defined as in \eqref{eq:def_S_r}.
	Then, we have 
	\begin{align*}
	\E \left[ Q_n \right] 
	& \geq \exp\left( \sum_{i=2}^{n} (e^{\theta_{t_i}} -1)\frac{w_i}{W_i} 
	- \sum_{i=2}^{n} \frac{e^{2\theta_{t_i}}}{2} \frac{w_i^2}{W_i^2}
	- \sum_{r=1}^{t_n} \theta_r \right) 
	\Ec{\exp\left( \sum_{r=1}^{t_n-1} (\theta_{r+1}-\theta_r) S_r \right)
		\1_{S_{t_n} = 0, \, \forall r < t_n, S_r \leq 0}}.
	\end{align*}
\end{lemma}

\begin{proof}
	This follows directly from the equality in 
	Lemma~\ref{lem:bound first moment Q_n}, together with the inequality $1+x \geq \exp(x-\frac{x^2}{2})$ for $x \geq 0$.
\end{proof}
We now prove an upper bound for $\E[Q_n^2]$. This relies on the many-to-two lemma and a time-inhomogeneous change of measure. We keep aside the term $\ell = 1$ given by the many-to-two lemma, because it will be the main term once we work with the auxiliary tree $\ttT_n^{(N)}$.
In that case, the first term in the bound below is roughly $\E[Q_n]^2$ and the second one is an error term.
Note in the proof that we manage to bound the term $\ell=1$ in terms of $\E[Q_n]^2$ using Lemma \ref{lem:lower_bound_first_moment_Q_n}, but without estimating the expectation involving the walk $S$.
\begin{lemma} \label{lem:second_moment_Q_n}
	Let $(\theta_r)_{r\geq 1}$ be a non-decreasing sequence of non-negative numbers.
	Recall the definition of $p_i$ in~\eqref{eq:def_p_i}.
	Let $n \geq 1$ such that $n =i_{t_n}$. 
	Then, we have 
	\begin{align*}
	\Ec{Q_n^2} 
	& \leq \left( \prod_{i=2}^n (1-p_i^2)^{-1} \right) \cdot
	\exp\left( \sum_{i=2}^{n} e^{2\theta_{t_i}} \frac{w_i^2}{W_i^2} \right) \cdot \E \left[ Q_n \right]^2 \\
	& \qquad {}
	+ \exp \left( \sum_{i=2}^{n} (e^{\theta_{t_i}} -1) \frac{w_i}{W_i}
	- \sum_{r=1}^{t_n} \theta_r \right)
	\sum_{\ell=2}^n \frac{w_\ell^2}{W_\ell^2}  \cdot
	\exp \left( \theta_{t_\ell} 
	+ \sum_{i=\ell+1}^{n} (e^{\theta_{t_i}} -1) \frac{w_i}{W_i}
	- \sum_{r=t_\ell+1}^{t_n} \theta_r \right).
	\end{align*}
\end{lemma}
\begin{proof}
Applying the many-to-two lemma (Lemma~\ref{lem:many-to-two}), we get 
\begin{align} \label{eq:many-to-two_to_Q_n}
	\Ec{Q_n^2} & = \sum_{\ell=1}^n \frac{w_\ell^2}{W_\ell^2} 
	\cdot \left( \prod_{i=\ell+1}^{n} \left(1-\frac{w_i^2}{W_i^2} \right) \right) \cdot 
	\Ec{ \1_{H^\ell_n = t_n} \1_{\forall r < t_n, H^\ell_{i_r} \leq r}
	\cdot  \1_{\overline{H}\vphantom{H}^\ell_n = t_n} \1_{\forall r < t_n, \overline{H}\vphantom{H}^\ell_{i_r} \leq r}},
\end{align}
where $H^\ell$ and $\overline{H}\vphantom{H}^\ell$ are the walks associated to the jumps $B_i^\ell$ and $\overline{B}\vphantom{B}_i^\ell$, jointly defined in Lemma~\ref{lem:many-to-two}.
Let the $(Y_i^\ell,\overline{Y}\vphantom{Y}_i^\ell)$'s for $i \geq 2$ be independent couples of Bernoulli r.v., with a distribution defined by
\[
	\forall \phi \colon \{0,1\}^2 \to \R, \qquad 
	\Ec{\phi(Y_i^\ell,\overline{Y}\vphantom{Y}_i^\ell)} 
	= \frac{\Ec{\phi(B_i^\ell,\overline{B}\vphantom{B}_i^\ell) 
	\exp \left(\theta_{t_i} (B_i^\ell + \overline{B}\vphantom{B}_i^\ell \1_{i > \ell}) \right)}}
	{\Ec{\exp \left(\theta_{t_i} (B_i^\ell + \overline{B}\vphantom{B}_i^\ell \1_{i > \ell}) \right)}}.
\]
Recalling the definition of $p_i$ in \eqref{eq:def_p_i},
one can check that
\begin{itemize}
	\item if $i < \ell$, then $Y_i^\ell$ has distribution Bernoulli$(p_i)$ and $\overline{Y}\vphantom{Y}_i^\ell=Y_i^\ell$;
	\item if $i = \ell$, then $Y_i^\ell=\overline{Y}\vphantom{Y}_i^\ell=1$;
	\item if $i > \ell$, then $Y_i^\ell$ and $\overline{Y}\vphantom{Y}_i^\ell$ have the same distribution as two independent Bernoulli$(p_i)$ r.v.\@ conditioned not to be both equal to $1$.
\end{itemize}
We set $S_0^\ell \coloneqq 0$, $\overline{S}\vphantom{S}_0^\ell \coloneqq 0$ and, for $r\geq 1$,
\begin{align*}
	X_r^\ell \coloneqq \sum_{i=i_{r-1}+1}^{i_r} Y_i^\ell,
	\qquad 
	S_r^\ell \coloneqq \sum_{s = 1}^r (X_s^\ell-1),
	\qquad 
	\overline{X}\vphantom{X}_r^\ell \coloneqq \sum_{i=i_{r-1}+1}^{i_r} \overline{Y}\vphantom{Y}_i^\ell
	\qquad \text{and} \qquad 
	\overline{S}\vphantom{S}_r^\ell \coloneqq \sum_{s = 1}^r (\overline{X}\vphantom{X}_s^\ell-1).
\end{align*}
Then, we use this change of measure to write
\begin{align*}
	& \Ec{ \1_{H^\ell_n = t_n} \1_{\forall r < t_n, H^\ell_{i_r} \leq r}
	\cdot \1_{\overline{H}\vphantom{H}^\ell_n = t_n} \1_{\forall r < t_n, \overline{H}\vphantom{H}^\ell_{i_r} \leq r}} \\
	& = \Ec{\exp\left( \sum_{i=2}^n \theta_{t_i} (B_i^\ell + \overline{B}\vphantom{B}_i^\ell \1_{i > \ell}) \right)}
	\Ec{\exp\left( -\sum_{i=2}^n \theta_{t_i} \left( Y_i^\ell + \overline{Y}\vphantom{Y}_i^\ell \1_{i > \ell} \right) \right)
	\1_{S_{t_n}^\ell = 0}
	\1_{\forall r < t_n, S_r^\ell \leq 0}
	\cdot 
	\1_{\overline{S}\vphantom{S}_{t_n}^\ell = 0}
	\1_{\forall r < t_n, \overline{S}\vphantom{S}_r^\ell \leq 0}}.
\end{align*}
For the first expectation, using the definition of $(B_i^\ell,\overline{B}\vphantom{B}_i^\ell)$, we have 
\begin{align*}
	\Ec{\exp\left( \sum_{i=2}^n \theta_{t_i} (B_i^\ell + \overline{B}\vphantom{B}_i^\ell \1_{i > \ell}) \right)}
	& = \prod_{i=2}^{\ell-1} \left( 1+(e^{\theta_{t_i}} -1) \frac{w_i}{W_i} \right) 
	\cdot e^{\theta_{t_\ell}\1_{\ell \geq 2}}
	\cdot \prod_{i=\ell+1}^{n} \left( 1+(e^{\theta_{t_i}} -1) 
	\frac{2 \frac{w_i}{W_i}}{1+ \frac{w_i}{W_i}} \right) \\
	& \leq \exp \left( \sum_{i=2}^{\ell-1} (e^{\theta_{t_i}} -1) \frac{w_i}{W_i}
	+ \theta_{t_\ell} \1_{\ell \geq 2} 
	+ \sum_{i=\ell+1}^{n} 2(e^{\theta_{t_i}} -1) \frac{w_i}{W_i} \right) \\
	& \leq \exp \left( \sum_{i=2}^{n} (e^{\theta_{t_i}} -1) \frac{w_i}{W_i}
	+ \theta_{t_\ell} \1_{\ell \geq 2} 
	+ \sum_{i=\ell+1}^{n} (e^{\theta_{t_i}} -1) \frac{w_i}{W_i} \right).
\end{align*}
On the other hand, for the second expectation, we rewrite 
\begin{align*}
	\sum_{i=2}^n \theta_{t_i} Y_i^\ell
	= \sum_{r=1}^{t_n} \theta_r  + \sum_{r=1}^{t_n} \theta_r (X_r^\ell-1)
	= \sum_{r=1}^{t_n} \theta_r + \sum_{r=1}^{t_n-1} (\theta_r-\theta_{r+1}) S_r^\ell + \theta_{t_n} S_{t_n}^\ell - \theta_1 S_0^\ell,
\end{align*}
with a summation by part, and similarly (noting that $\ell + 1 \leq i_{t_\ell}+1$ and using here that $n=i_{t_n}$)
\begin{align*}
	\sum_{i=\ell+1}^n \theta_{t_i} \overline{Y}\vphantom{Y}_i^\ell
	& = \sum_{i=\ell+1}^{i_{t_\ell}} \theta_{t_i} \overline{Y}\vphantom{Y}_i^\ell
	+ \sum_{i=i_{t_\ell}+1}^n \theta_{t_i} \overline{Y}\vphantom{Y}_i^\ell 
	= \sum_{i=\ell+1}^{i_{t_\ell}} \theta_{t_i} \overline{Y}\vphantom{Y}_i^\ell
	+ \sum_{r=t_{\ell}+1}^{t_n} \theta_r  + \sum_{r=t_{\ell}+1}^{t_n} \theta_r (\overline{X}\vphantom{X}_r^\ell-1) \\
	& = \sum_{i=\ell+1}^{i_{t_\ell+1}} \theta_{t_i} \overline{Y}\vphantom{Y}_i^\ell
	+ \sum_{r=t_\ell+1}^{t_n} \theta_r 
	+ \sum_{r=t_\ell+1}^{t_n-1} (\theta_r-\theta_{r+1}) \overline{S}\vphantom{S}_r^\ell 
	+ \theta_{t_n} \overline{S}\vphantom{S}_{t_n}^\ell - \theta_{t_\ell+1} \overline{S}\vphantom{S}_{t_\ell}^\ell,
\end{align*}
Using the bounds $\sum_{i=\ell+1}^{i_{t_\ell+1}} \theta_{t_i} \overline{Y}\vphantom{Y}_i^\ell \geq 0$ and $- \theta_{t_\ell+1} \overline{S}\vphantom{S}_{t_\ell}^\ell \geq 0$ (on the event $S_{t_\ell}^\ell \leq 0$), and recalling that on the event of interest we have $S_{t_n}^\ell =\overline{S}\vphantom{S}_{t_n}^\ell= 0$  we get
\begin{align*}
	& \Ec{\exp\left( -\sum_{i=2}^n \theta_{t_i} \left( Y_i^\ell + \overline{Y}\vphantom{Y}_i^\ell \1_{i > \ell} \right) \right)
	\1_{S_{t_n}^\ell = 0}
	\1_{\forall r < t_n, S_r^\ell \leq 0}
	\cdot 
	\1_{\overline{S}\vphantom{S}_{t_n}^\ell = 0}
	\1_{\forall r < t_n, \overline{S}\vphantom{S}_r^\ell \leq 0}} \\
	& \leq \exp\left( - \sum_{r=1}^{t_n} \theta_r
	- \sum_{r=t_\ell+1}^{t_n} \theta_r \right)
	\Ec{\exp\left( \sum_{r=1}^{t_n-1} (\theta_{r+1}-\theta_r) S_r^\ell 
		+ \sum_{r=t_\ell+1}^{t_n-1} (\theta_{r+1}-\theta_r) \overline{S}\vphantom{S}_r^\ell \right)
	\1_{S_{t_n}^\ell = \overline{S}\vphantom{S}_{t_n}^\ell = 0, \forall r < t_n, S_r^\ell \leq 0, \overline{S}\vphantom{S}_r^\ell \leq 0}}.
\end{align*}
Coming back to \eqref{eq:many-to-two_to_Q_n} and using the bound $\prod_{i=\ell+1}^{n}(1-w_i^2/W_i^2) \leq 1$, we proved
\begin{align} 
	\Ec{Q_n^2} 
	& \leq \sum_{\ell=1}^n \frac{w_\ell^2}{W_\ell^2}  \cdot
	\exp \left( \sum_{i=2}^{n} (e^{\theta_{t_i}} -1) \frac{w_i}{W_i}
	+ \theta_{t_\ell} \1_{\ell \geq 2} 
	+ \sum_{i=\ell+1}^{n} (e^{\theta_{t_i}} -1) \frac{w_i}{W_i} 
	- \sum_{r=1}^{t_n} \theta_r
	- \sum_{r=t_\ell+1}^{t_n} \theta_r \right) \nonumber \\
	& \qquad \quad {} \cdot 
	\Ec{\exp\left( \sum_{r=1}^{t_n-1} (\theta_{r+1}-\theta_r) S_r^\ell 
		+ \sum_{r=t_\ell+1}^{t_n-1} (\theta_{r+1}-\theta_r) \overline{S}\vphantom{S}_r^\ell \right)
		\1_{S_{t_n}^\ell = \overline{S}\vphantom{S}_{t_n}^\ell = 0, \forall r < t_n, S_r^\ell \leq 0, \overline{S}\vphantom{S}_r^\ell \leq 0}}. \label{eq:bound_2nd_moment}
\end{align}
We now focus on the term $\ell = 1$ (so that $t_\ell = 0$).
Recall the definition of $(Y_i)_{i\geq 2}$ and $(S_r)_{r\geq0}$ in Lemma~\ref{lem:bound first moment Q_n} and let $(\overline{Y}\vphantom{Y}_i)_{i\geq 2}$ and $(\overline{S}\vphantom{S}_r)_{r\geq0}$ be independent copies of them.
Then, by definition of the $(Y_i^1,\overline{Y}\vphantom{Y}_i^1)$'s, the expectation on the RHS of \eqref{eq:bound_2nd_moment} equals (we use here that $n=i_{t_n}$ so that the indicator function appearing in \eqref{eq:def_S_r} does not play a role)
\begin{align*}
	& \Ecsq{\exp\left( \sum_{r=1}^{t_n-1} (\theta_{r+1}-\theta_r) S_r
		+ \sum_{r=1}^{t_n-1} (\theta_{r+1}-\theta_r) \overline{S}\vphantom{S}_r \right)
		\1_{S_{t_n} = \overline{S}\vphantom{S}_{t_n} = 0, \forall r < t_n, S_r \leq 0, \overline{S}\vphantom{S}_r \leq 0}}{\forall i \in \llbracket 2,n\rrbracket, (Y_i,\overline{Y_i}) \neq (1,1) } \\
	& \leq \frac{\Ec{\exp\left( \sum_{r=1}^{t_n-1} (\theta_{r+1}-\theta_r) S_r
			+ \sum_{r=1}^{t_n-1} (\theta_{r+1}-\theta_r) \overline{S}\vphantom{S}_r \right)
			\1_{S_{t_n} = \overline{S}\vphantom{S}_{t_n} = 0, \forall r < t_n, S_r \leq 0, \overline{S}\vphantom{S}_r \leq 0}}}
		{\Pp{\forall i \in \llbracket 2,n\rrbracket, (Y_i,\overline{Y_i}) \neq (1,1) }} \\
	& = \left( \prod_{i=2}^n (1-p_i^2)^{-1} \right) \cdot 
	\Ec{\exp\left( \sum_{r=1}^{t_n-1} (\theta_{r+1}-\theta_r) S_r \right)
		\1_{S_{t_n} = 0, \forall r < t_n, S_r \leq 0}}^2.
\end{align*}
Therefore, the $\ell =1$ term on the RHS of \eqref{eq:bound_2nd_moment} is at most
\begin{align*} 
& \exp \left( \sum_{i=2}^{n} 2(e^{\theta_{t_i}} -1) \frac{w_i}{W_i}
- \sum_{r=1}^{t_n} 2\theta_r \right) \cdot
\left( \prod_{i=2}^n (1-p_i^2)^{-1} \right) \cdot 
\Ec{\exp\left( \sum_{r=1}^{t_n-1} (\theta_{r+1}-\theta_r) S_r \right)
	\1_{S_{t_n} = 0, \forall r < t_n, S_r \leq 0}}^2 \\
& \leq \left( \prod_{i=2}^n (1-p_i^2)^{-1} \right) \cdot
\exp\left( \sum_{i=2}^{n} e^{2\theta_{t_i}} \frac{w_i^2}{W_i^2} \right) \cdot \E \left[ Q_n \right]^2,
\end{align*}
by Lemma~\ref{lem:lower_bound_first_moment_Q_n}.
For the $\ell \geq 2$ terms on the RHS of \eqref{eq:bound_2nd_moment}, we simply bound the expectation by 1, using that $(\theta_k)_{k\geq1}$ is non-decreasing.
This proves the result.
\end{proof}
%

\subsection{General criterion for the lower bound}

The proposition below is the main result of this section: it provides a general lower bound in terms of some quantities which we will have to control in each case.
\begin{proposition} \label{prop:lower_bound}
	Assume that are given
	\begin{itemize}
		\item a non-decreasing sequence $(T(n))_{n\geq 2}$ of non-negative integers such that $T(n) < t_n$ for any $n \geq 2$, and $T(n) \to \infty$ as $n \to \infty$;
		\item a sequence $(\theta_r)_{r\geq 1}$ such that, for some $r_0 \geq 1$, $(\theta_r)_{r\geq r_0}$ is a non-decreasing sequence of non-negative numbers.
	\end{itemize}
	Let $(S_r)_{r\geq 0}$ be defined as in \eqref{eq:def_S_r}.
	Define, for $n \geq 2$,
	\begin{align*}
	E(n) & \coloneqq \Ec{\exp\left( \sum_{r=1}^{t_n-T(n)-1} (\theta_{T(n)+r+1}-\theta_{T(n)+r}) (S_{T(n)+r} - S_{T(n)}) \right)
		\1_{S_{t_n} = S_{T(n)}, \, \forall r < t_n-T(n), S_{T(n)+r} \leq S_{T(n)}}} \\
	\delta_1(n) & \coloneqq \sum_{i \geq i_{T(n)}+1} e^{2\theta_{t_i}} \frac{w_i^2}{W_i^2} \\ 
	\delta_2(n) & \coloneqq 
	\sum_{\ell \geq i_{T(n)}+1} e^{2\theta_{t_\ell}} \frac{w_\ell^2}{W_\ell^2} 
	\cdot \exp \left( \sum_{r=T(n)+1}^{t_\ell-1} \theta_r 
	- \sum_{i=i_{T(n)}+1}^{i_{t_\ell-1}} (e^{\theta_{t_i}} -1) \frac{w_i}{W_i} \right).
	\end{align*}
	If the following series are convergent
	\begin{equation} \label{eq:ass_lower_bound}
	\sum_{t \geq 1} \delta_1(i_t) < \infty
	\qquad \text{and} \qquad \sum_{t \geq 1} \frac{\delta_2(i_t)}{E(i_t)^2} < \infty
	\end{equation}
	then almost surely, for $n$ large enough, $\haut(\ttT_n) \geq t_n - T(n) - 1$.
\end{proposition}
Before diving into the proof, we provide some insight concerning the quantities appearing above. 
First, $E(n)$ is the expectation involving the walk $S$ appearing in Lemma \ref{lem:lower_bound_first_moment_Q_n}, but for the auxiliary tree $\ttT_n^{(N)}$ with $N = i_{T(n)}$. It always plays a negligible role, but we need a lower bound for it, which is the content of Lemma~\ref{lem:E(n)} established in Section~\ref{sec:E(n)}.
Then, $\delta_1(n)$ and $\delta_2(n)$ are small as a consequence of our assumption on $\sum_{i=n}^\infty w_i^2/W_i^2$.
Moreover, for $\delta_2(n)$, we need the exponential term to be not too large. This amounts to requiring that the quantity
\[
i_t \cdot \exp \left( \sum_{i=2}^{i_t} (e^{\theta_{t_i}} -1) \frac{w_i}{W_i} - \sum_{r=1}^t \theta_r \right)
\]
diverges fast enough.
This has to be compared with the criterion for the upper bound (Proposition~\ref{prop:upper_bound}), where this quantity had to be summable.
Recall that this quantity is approximately $i_t \E[Q_{i_t}]$, which is roughly the number of particles contributing to $Q_{i_t}$.
This number needs to be large so that we have a good concentration in the first and second moment argument.
\begin{proof}[Proof of Proposition~\ref{prop:lower_bound}]
	We write $T = T(n)$ and $N = N(n) = i_T$. 
	We consider the trees $(\ttT_n^{(N)})_{n\geq 1}$ introduced in Section~\ref{sec:some_definitions}. 
	We also define $i_0^{(N)} \coloneqq 1$ and
	\[
		i_r^{(N)} \coloneqq i_{T+r}, \qquad 
		\theta_r^{(N)} \coloneqq \theta_{T+r}, \qquad
		\text{for } r \geq 1.
	\]
	We define $(t_i^{(N)})_{i\geq 1}$ as before, but associated to the sequence $(i_r^{(N)})_{r\geq 1}$. 
	One can check that
	\begin{equation}
		t_i^{(N)} = 
		\begin{cases}
		t_i-T, & \text{if } i \geq N+1, \\
		1, & \text{if } 2 \leq i \leq N, \\
		0, & \text{if } i =1.
		\end{cases}
	\end{equation}
	We also define $Q_n^{(N)}$ as in \eqref{eq:def_Q_n} but for the tree $(\ttT_n^{(N)})_{n\geq 1}$ and the sequences $(i_r^{(N)})_{r\geq 0}$ and $(t_i^{(N)})_{i\geq 1}$.
	By construction, we have $\haut(\ttT_n) \geq \haut(\ttT_n^{(N)})$ and therefore,
	\begin{align} \label{eq:cauchy-schwarz}
		\Pp{\haut(\ttT_n) \geq t_n - T}
		\geq \Pp{\haut(\ttT_n^{(N)}) \geq t_n^{(N)}}
		\geq \Pp{Q_n^{(N)} > 0} 
		\geq \frac{\Ec{Q_n^{(N)}}^2}{\Ec{(Q_n^{(N)})^2}},
	\end{align}
	where the last inequality follows from Cauchy--Schwarz inequality.
	
	First assume that $n = i_{t_n}$.
	We apply Lemma~\ref{lem:lower_bound_first_moment_Q_n} and Lemma~\ref{lem:second_moment_Q_n} to $Q_n^{(N)}$, using that $(\ttT_n^{(N)})_{n\geq 1}$ is a weighted recursive tree with weight sequence $(w_n^{(N)})_{n\geq 1}$.
	Recall the definition of $p_i$ and $(S_r)_{r\geq 0}$ in \eqref{eq:def_p_i} and \eqref{eq:def_S_r} and define similarly $p_i^{(N)}$ and $(S_r^{(N)})_{r\geq 0}$ but for the sequences $(w_n^{(N)})_{n\geq 1}$, $(i_r^{(N)})_{r\geq 0}$ and $(\theta_r^{(N)})_{r\geq 1}$.
	Note that, for $2 \leq i \leq N$, $w_i^{(N)} = 0$ and, for $i \geq N+1$,
	$\theta^{(N)}_{t_i^{(N)}} = \theta_{t_i}$ and $w_i^{(N)}/W_i^{(N)} = w_i/W_i$, so we get
	\[
		p_i^{(N)} = 
		\begin{cases}
		p_i, & \text{if } i \geq N+1, \\
		0, & \text{if } 2 \leq i \leq N,
		\end{cases}
		\qquad \text{and} \qquad 
		(S_r^{(N)})_{r\geq 0} \overset{\text{(d)}}{=} (S_{T(n)+r} - S_{T(n)})_{r\geq 0}.
	\]
	We consider $n$ large enough in such a way that $T \geq r_0$, so that $(\theta_r^{(N)})_{r\geq 1}$ is a non-decreasing sequence of non-negative numbers.
	Therefore, Lemma~\ref{lem:lower_bound_first_moment_Q_n} implies
	\begin{align*}
	\E \left[ Q_n^{(N)} \right] 
	& \geq \exp\left( \sum_{i=N+1}^{n} (e^{\theta_{t_i}} -1)\frac{w_i}{W_i} 
	- \sum_{i=N+1}^{n} \frac{e^{2\theta_{t_i}}}{2} \frac{w_i^2}{W_i^2}
	- \sum_{r=T+1}^{t_n} \theta_r \right) 
	\cdot E(n) \\
	& \geq \exp\left( \sum_{i=N+1}^{n} (e^{\theta_{t_i}} -1)\frac{w_i}{W_i} 
	- \sum_{r=T+1}^{t_n} \theta_r \right) 
	\cdot \frac{E(n)}{2},
	\end{align*}
	for $n$ large enough, using that $\sum_{i=N+1}^{n} e^{2\theta_{t_i}} w_i^2/W_i^2 \leq \delta_1(n)$ and $\delta_1(n) \to 0$ as a consequence of \eqref{eq:ass_lower_bound}.
	On the other hand, Lemma~\ref{lem:second_moment_Q_n} yields
	\begin{align*}
	& \Ec{(Q_n^{(N)})^2} \\
	& \leq \left( \prod_{i=N+1}^n (1-p_i^2)^{-1} \right) \cdot
	\exp\left( \sum_{i=N+1}^{n} e^{2\theta_{t_i}} \frac{w_i^2}{W_i^2} \right) \cdot \E \left[ Q_n^{(N)} \right]^2 \\
	& \qquad {}
	+ \exp \left( \sum_{i=N+1}^{n} (e^{\theta_{t_i}} -1) \frac{w_i}{W_i}
	- \sum_{r=T+1}^{t_n} \theta_r \right)
	\cdot \sum_{\ell=N+1}^n \frac{w_\ell^2}{W_\ell^2}  \cdot
	\exp \left( \theta_{t_\ell} 
	+ \sum_{i=\ell+1}^{n} (e^{\theta_{t_i}} -1) \frac{w_i}{W_i}
	- \sum_{r=t_\ell+1}^{t_n} \theta_r \right).
	\end{align*}
	The fact that $\delta_1(n) \to 0$ ensures that $e^{\theta_{t_i}} \frac{w_i}{W_i} \to 0$ as $i \to \infty$ and therefore $p_i \sim e^{\theta_{t_i}} \frac{w_i}{W_i}$.
	Hence, for $n$ large enough (and hence $N$ large enough), we have
	\begin{align*}
	\left( \prod_{i=N+1}^n (1-p_i^2)^{-1} \right) \cdot
	\exp\left( \sum_{i=N+1}^{n} e^{2\theta_{t_i}} \frac{w_i^2}{W_i^2} \right)
	\leq \exp \left( 3\sum_{k=N+1}^{n} e^{2\theta_{t_k}} \frac{w_k^2}{W_k^2} \right)
	\leq  \exp \left( 3\delta_1(n) \right)
	\leq 1 + 4\delta_1(n),
	\end{align*}
	for $n$ large enough, using again that $\delta_1(n) \to 0$.
	Combining what precedes, we get, for $n$ large enough,
	\begin{align*}
	\frac{\Ec{(Q_n^{(N)})^2}}{\Ec{Q_n^{(N)}}^2}
	& \leq 1 + 4\delta_1(n)
	+ \frac{4}{E(n)^2} 
	\sum_{\ell=N+1}^n \frac{w_\ell^2}{W_\ell^2}  \cdot
	\exp \left( \theta_{t_\ell} 
	- \sum_{i=N+1}^{\ell} (e^{\theta_{t_i}} -1) \frac{w_i}{W_i}
	+ \sum_{r=T+1}^{t_\ell} \theta_r \right) \\
	& \leq 1 + 4\delta_1(n) 
	+ \frac{4}{E(n)^2} \cdot \delta_2(n),
	\end{align*}
	bounding $\sum_{i=i_{t_\ell-1}+1}^{\ell} (e^{\theta_{t_i}} -1) \frac{w_i}{W_i} \geq 0$.
	Coming back to \eqref{eq:cauchy-schwarz}, this proves, for $n$ large enough satisfying $n=i_{t_n}$,
	\[
	\Pp{\haut(\ttT_n) \geq t_n - T(n)} 
	\geq \left( 1 + 4\delta_1(n) + \frac{4 \delta_2(n)}{E(n)^2} \right)^{-1}.
	\]
	By assumption \eqref{eq:ass_lower_bound} and the Borel--Cantelli lemma, it follows that almost surely, for $n$ large enough satisfying $n=i_{t_n}$, we have $\haut(\ttT_n) \geq t_n - T(n)$.
	
	We now want to get a lower bound for $\haut(\ttT_m)$ when $m \in \llbracket i_{t_n}+1, i_{t_n+1}-1 \rrbracket$ for some $n=i_{t_n}$. 
	In particular, $t_m = t_n+1$. 
	Since $n \leq m$, we have $\haut(\ttT_m) \geq \haut(\ttT_n)$ and $T(m) \geq T(n)$. 
	Therefore,
	\[
		\haut(\ttT_m) - (t_m-T(m)-1) 
		\geq \haut(\ttT_n) - (t_n-T(n)).
	\]
	The result follows.
\end{proof}

\begin{remark} \label{rem:diameter} Under the assumptions of Proposition \ref{prop:lower_bound}, one can also conclude that almost surely, for $n$ large enough, $\diam(\ttT_n) \geq 2(t_n - T(n) - 1)$, by following the same proof as the one of \cite[Theorem 1.3]{painsenizergues2022}.
Indeed, the key point in this argument is that, in the second moment of $Q_n^{(N)}$, the term $\ell=1$ is the dominant one: this means that two vertices chosen independently according to their weight in $Q_n^{(N)}$ typically have the root as most recent common ancestor.
Noting that the upper bound $\diam(\ttT_n) \geq 2 \haut(\ttT_n)$ is direct, this implies that Theorem \ref{thm:powers_of_log} and Theorem \ref{thm:quickly} also hold for $\diam(\ttT_n)$ by multiplying all the terms in the expansion by 2.  
\end{remark}

\subsection{Lower bound for the term \texorpdfstring{$E(n)$}{E(n)}}\label{sec:E(n)} 
We conclude this section by stating and proving Lemma~\ref{lem:E(n)} below which provides some lower bound on the term $E(n)$ that appears in the assumptions of Proposition~\ref{prop:lower_bound}. 

\begin{lemma} \label{lem:E(n)}
	Let $(\theta_r)_{r\geq 1}$ be a non-decreasing sequence of non-negative numbers such that, for $t$ large enough, we have $e^{\theta_t}(a_t-a_{t-1})=1$.
	Recall the definition of $E(n)$ from Proposition~\ref{prop:lower_bound}.
	Assume that $T(n) \to \infty$ and $t_n-T(n) \to \infty$ as $n \to \infty$.
	Also assume that 
	\begin{equation} \label{eq:summable}
		\sum_{i\geq 1} e^{2\theta_{t_i}}\frac{w_i^2}{W_i^2} < \infty.
	\end{equation}
	Then, there exists a constant $C > 0$, such that for all $n \geq 1$ satisfying $n =i_{t_n}$ we have
	\begin{enumerate}
	 	\item\label{it:lower bound walk follows barrier} $E(n)\geq \exp(-C (t_n-T(n)))$;
	 	\item\label{it:lower bound walk poisson} 
	 	$\displaystyle E(n)\geq \exp\left(-(\theta_{t_n}-\theta_{T(n)})\cdot (t_n-T(n))^{1/2} - C \log(t_n-T(n))\right)
	 	- C \sum_{i=i_{T(n)}+1}^{n} e^{2\theta_{t_i}}\frac{w_i^2}{W_i^2}$.
	\end{enumerate}
\end{lemma}
Note that it need not be the case that the second estimate is better than the first one as it depends on the rate of growth of $\theta_t$.
\begin{proof}
The two lower bounds are obtained in the same fashion, by identifying an event for the walk $(S_{T(n)+r} - S_{T(n)})_{r \geq 0}$ for which we can give a lower bound for the exponential term in the expectation. 
Recall the definition of $p_i$ from~\eqref{eq:def_p_i}.
It is easy to check that we have
\begin{equation} \label{eq:estimates_p_i}
	p_i= e^{\theta_{t_i}} \frac{w_i}{W_i}+ \grandO{e^{2\theta_{t_i}}\frac{w_i^2}{W_i^2}} 
	\qquad \text{ and } \qquad 
	p_i^2= \grandO{e^{2\theta_{t_i}}\frac{w_i^2}{W_i^2}},
\end{equation}
using in particular that $e^{\theta_{t_i}}\frac{w_i}{W_i}\rightarrow 0$ as a consequence of assumption \eqref{eq:summable}.

We start by proving \ref{it:lower bound walk follows barrier}. 
Noting that on the event $\{ \forall  r \leq t_n-T(n),\ S_{T(n)+r}=S_{T(n)} \}$ the exponential term in $E(n)$ equals $1$, we get
\begin{align}\label{eq:E(n) greater than product probability Xt equal 1}
	E(n) 
	& \geq \Pp{\forall  r \leq t_n-T(n),\ S_{T(n)+r}=S_{T(n)} }
	= \prod_{r=1}^{t_n - T(n)}\Pp{X_{T(n)+r}=1}.
\end{align}
Then, for any $t \geq 1$, recalling that $X_t = \sum_{i=i_{t-1}+1}^{i_t}Y_i$, where the $Y_i$'s are independent Bernoulli($p_i)$ r.v.\@, we have 
\begin{align}
	\Pp{X_t=1}
	& = \sum_{i=i_{t-1}+1}^{i_t} p_i \underset{j\neq i}{\prod_{j=i_{t-1}+1}^{i_t}}(1-p_j) 
	\geq \left(\sum_{i=i_{t-1}+1}^{i_t} p_i\right)\cdot   \left(\prod_{i=i_{t-1}+1}^{i_t}(1-p_i)\right).
	\label{eq:lower_bound_X_t}
\end{align}
Using \eqref{eq:estimates_p_i} and then \eqref{eq:summable}, we get that as $t\rightarrow\infty$
\begin{align*}
	\sum_{i=i_{t-1}+1}^{i_t}p_i =  \sum_{i=i_{t-1}+1}^{i_t} e^{\theta_{t}}\frac{w_i}{W_i}+ \sum_{i=i_{t-1}+1}^{i_t} \grandO{e^{2\theta_{t}}\frac{w_i^2}{W_i^2}}= e^{\theta_{t}}(a_t-a_{t-1}) + o(1) =1+ \petito{1},
\end{align*}
using that for $t$ large enough $e^{\theta_t}(a_t-a_{t-1})=1$.
Moreover, using the fact that $p_i \to 0$, the Taylor expansion $\log(1-x) \underset{x\rightarrow 0}{=} -x + \grandO{x^2}$ and the previous display we get
\begin{align*}
\left(\prod_{i=i_{t-1}+1}^{i_t}(1-p_i)\right)
= \exp\left(-\left(\sum_{i=i_{t-1}+1}^{i_t}p_i\right) 
+\grandO{\sum_{i=i_{t-1}+1}^{i_t}p_i^2}\right)
= \exp(-1+o(1)).
\end{align*}
Coming back to \eqref{eq:lower_bound_X_t}, this shows that $\Pp{X_t=1}$ is bounded below by some positive value for $t$ sufficiently large. 
Plugging this back into \eqref{eq:E(n) greater than product probability Xt equal 1}, we get part \ref{it:lower bound walk follows barrier} of the lemma.

In order to prove \ref{it:lower bound walk poisson}, we first compare $(S_r)$ with a time-homogeneous random walk $(\widehat{S}_r)$ that has $\mathrm{Poisson}(1)-1$ increments. 
Applying \cite[Lemma~A.1]{painsenizergues2022} to our case, we can couple $(S_r)$ and $(\widehat{S}_r)$ in such a way that 
\begin{align*}
\Pp{\exists r \in \intervalleentier{0}{t_n-T(n)}, \ S_{T(n)+r} - S_{T(n)} \neq  \widehat{S}_r} &\leq 2 \sum_{t=T(n)+1}^{t_n} \abs{\sum_{i=i_{t-1}+1}^{i_t} p_i -1} + 2 \sum_{i=i_{T(n)}+1}^{n} p_i^2\\
&\leq 2 \sum_{s=T(n)+1}^{t_n} \abs{e^{\theta_s}(a_{i_s}-a_{i_{s-1}}) - 1} + C \sum_{i=i_{T(n)}+1}^{n} e^{2\theta_{t_i}}\frac{w_i^2}{W_i^2},
\end{align*}
where the first sum on the right-hand side is zero for $n$ large enough.
Then, setting $r_n = t_n-T(n)$ for brevity, it is enough to prove that
\begin{equation} \label{eq:new_goal}
\Ec{\exp\left( \sum_{r=1}^{r_n-1} (\theta_{T(n)+r+1}-\theta_{T(n)+r}) \widehat{S}_r \right)
	\1_{\widehat{S}_{r_n} = 0, \, \forall r < r_n, \widehat{S}_r \leq 0}} 
\geq \exp\left(-(\theta_{t_n}-\theta_{T(n)})\cdot r_n^{1/2} - C \log r_n \right).
\end{equation}
For this, we will prove that 
\begin{equation} \label{eq:new_goal2}
\liminf_{n\to\infty} r_n^{3/2} \cdot \Pp{\widehat{S}_{r_n} = 0 \ \text{ and }\  \forall r < r_n,\ -r_n^{1/2} \leq \widehat{S}_r \leq 0} >0,
\end{equation}
which implies \eqref{eq:new_goal} by restricting ourselves to this event and bounding $\widehat{S}_r \geq -r_n^{1/2}$ inside the exponential.
We now prove \eqref{eq:new_goal2}.
First, according to \cite[Equation (4.6)]{caravennachaumont2013} (note that the probability below equals $q_{r_n}^+(0,0)$ for the walk $-\widehat{S}$ with their notation), we have, for some constant $C_1 > 0$, as $n \to \infty$,
\begin{equation} \label{eq:RW1}
\Pp{\widehat{S}_{r_n} = 0 \ \text{ and }\  \forall r < r_n,\  \widehat{S}_r \leq 0} 
\sim \frac{C_1}{r_n} \cdot \Pp{\widehat{S}_{r_n} = 0}
\sim \frac{C_1}{\sqrt{2\pi} \cdot r_n^{3/2}}.
\end{equation}
Now, let $(\widehat{\mathbf{S}}_u^{(n)})_{u\in[0,1]}$ denotes the process defined by linear interpolation of the points $(r/r_n,\widehat{S}_r/r_n^{1/2} )$ for $0 \leq r \leq r_n$.
Then, \cite[Corollary 2.5]{caravennachaumont2013} proves that the process $-\widehat{\mathbf{S}}^{(n)}$, conditionally on the event $\{ \widehat{S}_{r_n} = 0 \ \text{ and }\  \forall r < r_n,\ -r_n^{1/2} \leq \widehat{S_{r}} \leq 0\}$, converges in distribution in $(C([0,1],\R),\lVert \cdot \rVert_\infty)$ toward the normalized Brownian excursion $(\mathbf{e}_u)_{u\in \intervalleff{0}{1}}$.
In particular,
\begin{equation} \label{eq:RW2}
\Ppsq{\forall r < r_n,\ \widehat{S}_r \geq -r_n^{1/2}}
{\widehat{S}_{r_n} = 0 \ \text{ and }\  \forall r < r_n,\  \widehat{S}_r \leq 0} 
\xrightarrow[n\to\infty]{} 
\Pp{ \max_{u\in[0,1]} \mathbf{e}_u \leq 1} > 0,
\end{equation}
using that the map $f \in C([0,1],\R) \mapsto \max_{[0,1]} f$ is continuous.
Combining \eqref{eq:RW1} and \eqref{eq:RW2} shows \eqref{eq:new_goal2} and this concludes the proof. 
\end{proof}

\section{Application to different regimes for the weight sequence}\label{sec:proof of the main results}

In this section, we apply the general results of Section~\ref{sec:upper-bound} and Section~\ref{sec:lower-bound} to specific regimes for the weight sequence.
We first present the general ideas and then study the different regimes separately.

\subsection{Strategy}\label{sec:strategy_main_result}

The proofs of Theorem~\ref{thm:powers_of_log} and Theorem~\ref{thm:quickly} follow from applying Proposition~\ref{prop:upper_bound} and Proposition~\ref{prop:lower_bound} to well-chosen sequences $(i_t)_{t\geq 0}$ and $(\theta_t)_{t\geq 1}$. 
As discussed before, $(i_t)_{t\geq 0}$ should be chosen such that $\haut(\ttT_{i_t})$ is either slightly less (upper bound) or slightly more (lower bound) than $t$.
For the reasons explained in Section~\ref{sec:strategy_upper_bound}, we will always choose $\theta_t = - \log(a_{i_t} - a_{i_{t-1}})$ for $t$ large enough.
Then, the proof is mainly computational. We get an asymptotic expansion for $\theta_t$ and use it to check the assumptions of Proposition~\ref{prop:upper_bound} and Proposition~\ref{prop:lower_bound}.
In particular, this expansion has to be precise enough to ensure that the sequence $(\theta_t)$ is eventually non-decreasing.

Choosing $(i_t)_{t\geq 0}$ is not difficult if one already knows the expansion of $\haut(\ttT_n)$ that one tries to prove.
However, it may not be immediate to find a good conjecture for the terms in this expansion in the first place so we explain here how we guessed the behavior of $\haut(\ttT_n)$ in the different cases that we studied.
For Theorem~\ref{thm:powers_of_log}, when $\alpha \leq 1$, we started with the crude upper and lower bounds presented in Section~\ref{sec:crude_bounds}, which match at the first order.
This first order implies the choice $i_t = \exp(\alpha t \log t (1+o(1)))$, but this is already enough to deduce the expansion \eqref{eq:expansion_theta_t_1} for $\theta_t$.
Then, recall we want \eqref{eq:ass_upper_bound} to be summable for the upper bound and sufficiently quickly divergent for the lower bound, but to get precise bounds we want the first orders to cancel out.
Hence the precise definition of $i_t$ in \eqref{eq:def_ell_1} and \eqref{eq:def_i_t_1} includes exactly the right terms to cancel those coming from the exponential in \eqref{eq:ass_upper_bound}.
When $\alpha > 1$, the crude upper bound does not give the correct first order, but we simply used the same definition for $i_t$.
Also note that, as claimed in Remark \ref{rem:comparison_iid}, one could get the next orders in $\haut(\ttT_n)$.
Indeed, this precise definition of $i_t$ implies a more precise expansion for $\theta_t$, so plugging it in \eqref{eq:ass_upper_bound}, one can find what the next orders in the definition of $i_t$ should be, and so on.

For Theorem~\ref{thm:quickly}, we trusted the first order given by the crude lower bound given in Section~\ref{sec:crude lower bound} to find the definition of $i_t$. 
Indeed, it is rather intuitive that the method used for this crude lower bound should work better when $w_n$ tends faster to $0$.
Hence, since it is yielding the right first order for Theorem~\ref{thm:powers_of_log}, it had to be also correct for Theorem~\ref{thm:quickly}.

\subsection{Variance varying like powers of \texorpdfstring{$\log n$}{log n}}

Our aim in this section is to prove Theorem~\ref{thm:powers_of_log}. We work under assumption \eqref{eq:assumption_a_n_1} and \eqref{eq:ass_somme_carre}.
We first define and study a sequence $\ell_t$, which is then used to set $i_t \simeq \exp(\ell_t)$ later.
\begin{lemma} \label{lem:def_ell_t}
	Let $\eta \in \R$. 
	For $t \geq 2$, we define 
	\begin{equation} \label{eq:def_ell_1}
	\ell_t \coloneqq 
	\alpha t \log t - (1-\alpha) t \log \log t - t \log L(\alpha t \log t) - t \left[ 1+\alpha+(1-\alpha)\log \alpha +\eta \right].
	\end{equation}
	Then, we have the following asymptotics, as $t \to \infty$,
	\begin{enumerate}
		\item\label{it:ell} $\ell_t \sim \alpha t \log t$;
		\item\label{it:ell_diff} $\ell_t - \ell_{t-1} \sim \alpha \log t$;
		\item\label{it:log_ell_diff} $\log \ell_t - \log \ell_{t-1} \sim \frac{1}{t}$;
		\item\label{it:log_ell_diff_diff} $\log(\ell_{t+1} - \ell_t) - \log(\ell_t - \ell_{t-1}) = o \left( \frac{1}{t} \right)$.
	\end{enumerate}
\end{lemma}

\begin{proof}
	Part \ref{it:ell} follows directly from the fact that $\log L(x) = o(\log x)$ (see Lemma~\ref{lem:slowly_varying}\ref{it:slowly_varying_0}).
	
	For Part \ref{it:ell_diff}, we prove a more precise expansion that will be useful for Part \ref{it:log_ell_diff_diff}.
	We introduce $\tilde{\ell}_t \coloneqq \ell_t + t \log L(\alpha t \log t)$. One can check through explicit calculation that 
	\begin{align} \label{eq:ell_tilde}
	\tilde{\ell}_t - \tilde{\ell}_{t-1} = \alpha \log t - (1-\alpha) \log \log t 
	- \left[ 1+(1-\alpha)\log \alpha +\eta \right] - \frac{1-\alpha}{\log t} + O \left(\frac{1}{t} \right),
	\end{align} 
	so it remains to deal with the part involving $L$.
	We decompose
	\begin{align} 
	& t \log L(\alpha t \log t) - (t-1) \log L(\alpha (t-1) \log (t-1)) \nonumber  \\
	& = \log L(\alpha t \log t) + (t-1) \left[ \log L(\alpha t \log t) - \log L(\alpha (t-1) \log (t-1)) \right]. \label{eq:diff_L}
	\end{align} 
	Then, using Lemma~\ref{lem:slowly_varying}\ref{it:slowly_varying_2}, the function $\widetilde{L}(x) = \exp(xL'(x)/L(x))$ satisfies \hyperref[hyp_SV_k]{$(\mathrm{SV}_1)$} and we have
	\begin{align}
	& \log L(\alpha t \log t) - \log L(\alpha (t-1) \log (t-1)) \nonumber \\
	& = \log \widetilde{L}(\alpha t \log t)
	\cdot \log \left( \frac{t \log t}{(t-1) \log (t-1)} \right)
	+ \petito{\log^2 \left( \frac{t \log t}{(t-1) \log (t-1)} \right)} \nonumber \\
	& = \log \widetilde{L}(\alpha t \log t) \cdot
	\left[ \frac{1}{t} + \frac{1}{t \log t} + O\left( \frac{1}{t^2} \right) \right]
	+ \petito{\frac{1}{t^2}}.
	\label{eq:diff_log_L}
	\end{align}
	Hence, combining this with \eqref{eq:ell_tilde} and \eqref{eq:diff_L}, we proved
	\begin{align} \label{eq:ell_diff}
	\begin{split}
	\ell_t - \ell_{t-1} 
	& = \alpha \log t - (1-\alpha) \log \log t 
	- \log L(\alpha t \log t)
	- \log \widetilde{L}(\alpha t \log t) \cdot \left[ 1 + \frac{1}{\log t} 
	+ O \left(\frac{1}{t} \right) \right] \\
	& \quad {} - \left[ 1+(1-\alpha)\log \alpha +\eta \right] - \frac{1-\alpha}{\log t} 
	+ O \left(\frac{1}{t} \right).
	\end{split}
	\end{align} 
	Using $\log L(x) = o(\log x)$ and $\log \widetilde{L}(x) = o(\log x)$, this proves $\ell_t - \ell_{t-1} \sim \alpha \log t$.
	
	For Part \ref{it:log_ell_diff}, we simply have
	\[
	\log \ell_t - \log \ell_{t-1} 
	= \log \left( 1 - \frac{\ell_t - \ell_{t-1}}{\ell_{t-1}}\right) 
	\sim \frac{\ell_t - \ell_{t-1}}{\ell_{t-1}}
	\sim \frac{1}{t},
	\]
	using Parts \ref{it:ell} and \ref{it:ell_diff}.
	
	For Part \ref{it:log_ell_diff_diff}, we introduce $f(t) \coloneqq \alpha \log t - (1-\alpha) \log \log t - \left[ 1+(1-\alpha)\log \alpha +\eta \right] - \frac{1-\alpha}{\log t}$ and 
	$g(t) = \log L(\alpha t \log t)	- \log \widetilde{L}(\alpha t \log t) \cdot [1 + \frac{1}{\log t}]$, so that \eqref{eq:ell_diff} yields
	\begin{align*}
		\log(\ell_t - \ell_{t-1})
		= \log \left( f(t) +g(t) + O \left(\frac{1 + \log \widetilde{L}(\alpha t \log t)}{t} \right) \right)
		= \log f(t) 
		+ \log \left( 1 + \frac{g(t)}{f(t)} + o \left( \frac{1}{t} \right) \right),
	\end{align*}
	using that $f(t) \sim \alpha \log t$ and $\log \widetilde{L}(\alpha t \log t) = o(\log t)$. Note that $g(t)/f(t) \to 0$ so 
	\[
	\log \left( 1 + \frac{g(t)}{f(t)} + o \left( \frac{1}{t} \right) \right) 
	= \log \left( \left( 1 + \frac{g(t)}{f(t)} \right) \left( 1 + o \left( \frac{1}{t} \right) \right)\right)
	= \log \left( 1 + \frac{g(t)}{f(t)} \right) + o \left( \frac{1}{t} \right).
	\]
	Proceeding similarly, we also have 
	\[
	\log \left( 1 + \frac{g(t+1)}{f(t+1)} \right) 
	= \log \left( 1 + \frac{g(t)}{f(t)} \right) + O \left( \frac{g(t+1)}{f(t+1)} - \frac{g(t)}{f(t)} \right).
	\]
	combining what precedes yields
	\begin{align} \label{eq:diff_log_diff_ell}
	\log(\ell_{t+1} - \ell_t) - \log(\ell_t - \ell_{t-1}) 
	= \log \frac{f(t+1)}{f(t)} + O \left( \frac{g(t+1)}{f(t+1)} - \frac{g(t)}{f(t)} \right)
	+ o \left( \frac{1}{t} \right).
	\end{align}
	On the one hand, we have through explicit calculation
	\begin{align*}
	\log \frac{f(t+1)}{f(t)} 
	= \log \left( 1 + \frac{f(t+1)-f(t)}{f(t)} \right)
	= O \left( \frac{1}{t \log t} \right).
	\end{align*}
	On the other hand, it follows from Lemma~\ref{lem:slowly_varying}\ref{it:slowly_varying_1}
	that $\log L(\alpha (t+1) \log (t+1))  - \log L(\alpha t \log t) = o(1/t)$ and the same holds for $\widetilde{L}$ (because $\widetilde{L}$ satisfies \hyperref[hyp_SV_k]{$(\mathrm{SV}_1)$}).
	Combining these facts, we get
	\begin{align*} 
	\frac{g(t+1)}{f(t+1)} - \frac{g(t)}{f(t)} 
	= \frac{g(t+1) - g(t)}{f(t+1)} + g(t) \frac{f(t) - f(t+1)}{f(t)f(t+1)} 
	= o \left( \frac{1}{t} \right).
	\end{align*}
	Coming back to \eqref{eq:diff_log_diff_ell}, this proves Part \ref{it:log_ell_diff_diff}.
\end{proof}
As a consequence of Parts \ref{it:ell} and \ref{it:ell_diff} of Lemma~\ref{lem:def_ell_t}, there exists $s_0$ such that,
\begin{equation} \label{eq:def_i_t_1}
i_t \coloneqq 
\begin{cases}
t+1, & \text{if } 0 \leq t < s_0, \\
\lfloor \exp(\ell_t) \rfloor, & \text{if } t \geq s_0,
\end{cases}
\end{equation}
defines an increasing sequence of integers.
Indeed, we can choose $s_0$ large enough such that, for any $t \geq s_0$, $\lfloor \exp(\ell_t) \rfloor \geq t$ and $\ell_{t+1} - \ell_t \geq 1$.
Moreover, we will assume that $s_0$ is chosen large enough such that, for $t \geq s_0$, we have $a_{i_t}-a_{i_{t-1}} > 0$.
The fact that this last condition can be met is a consequence of the proof of the next lemma (see \eqref{eq:expansion_diff_a_i_t_1}). 
\begin{lemma} \label{lem:theta_t_1}
	Let $\eta \in \R$. Let $i_t$ be defined as in \eqref{eq:def_i_t_1} for some $s_0$ and
	\begin{equation} \label{eq:def_theta_t_1}
		\theta_t \coloneqq
		\begin{cases}
		0, & \text{if } 1 \leq t \leq s_0, \\
		-\log(a_{i_t}-a_{i_{t-1}}), & \text{if } t > s_0.
		\end{cases}
	\end{equation}
	Then, there exists $r_0 \geq s_0$ such that $(\theta_t)_{t\geq r_0}$ is a non-decreasing sequence of non-negative numbers.
	Moreover, as $t \to \infty$,
	\begin{equation} \label{eq:expansion_theta_t_1}
	\theta_t 
	= \alpha \log t - (1-\alpha) \log \log t - (1-\alpha) \log \alpha - \log L(\alpha t \log t) + o(1).
	\end{equation}
\end{lemma}

\begin{proof}
	Consider $t > s_0$, so that $i_t = \lfloor\exp(\ell_t)\rfloor$ and $i_{t-1} = \lfloor\exp(\ell_{t-1})\rfloor$.
	Recall also from Lemma~\ref{lem:def_ell_t} that $\ell_t \sim \alpha t \log t$ and $\ell_{t+1} - \ell_t \sim \alpha \log t$.
	It follows from \eqref{eq:assumption_a_n_1} that
	\begin{align}
	a_{i_t}-a_{i_{t-1}}
	& = \int_{\log i_{t-1}}^{\log i_t} x^{-\alpha} L(x) \diff x 
	+ o\left( (\log i_t)^{-1-\alpha} (\log \log i_t)^2 L(\log i_t) \right) \nonumber \\
	& = \int_{\ell_{t-1}}^{\ell_t} x^{-\alpha} L(x) \diff x  
	+ o\left( \frac{1}{t} (\ell_t-\ell_{t-1}) \ell_t^{-\alpha} L(\ell_t) \right), \label{eq:increment_a}
	\end{align}
	using that $\log i_t - \ell_t = O(1/i_t)$ to change the endpoints of the integral.
	Then, Lemma~\ref{lem:integral_slowly_varying} yields
	\begin{align}
	a_{i_t}-a_{i_{t-1}}
	& = \ell_t^{-\alpha} (\ell_t-\ell_{t-1}) L(\ell_t)
	\left( 1 + \frac{\alpha}{2} \cdot \frac{\ell_t-\ell_{t-1}}{\ell_t} (1+o(1)) + o\left( \frac{1}{t} \right) \right) \nonumber \\
	& = \ell_t^{-\alpha} (\ell_t-\ell_{t-1}) L(\ell_t)
	\left( 1 + \frac{\alpha}{2t} + o\left( \frac{1}{t} \right) \right). 
	\label{eq:expansion_diff_a_i_t_1}
	\end{align}
	It follows that
	\[
	\theta_t
	= \alpha \log \ell_t - \log(\ell_t-\ell_{t-1}) - \log L(\ell_t)
	- \frac{\alpha}{2t} + o\left( \frac{1}{t} \right).
	\]
	It follows from Lemma~\ref{lem:slowly_varying}\ref{it:slowly_varying_1} that $\log L(\ell_t) = \log L(\alpha t \log t) + o(1)$ and so the expansion \eqref{eq:expansion_theta_t_1} follows from Lemma~\ref{lem:def_ell_t}.\ref{it:ell}-\ref{it:ell_diff}.
	Hence, $\theta_t \sim \alpha \log t$ and so $\theta_t \geq 0$ for $t$ large enough.
	Moreover, applying again Lemma~\ref{lem:slowly_varying}\ref{it:slowly_varying_1}, we have 
	$\log L(\ell_{t+1}) - \log L(\ell_t) = o(1/t)$, so using Lemma~\ref{lem:def_ell_t}.\ref{it:log_ell_diff}-\ref{it:log_ell_diff_diff}, we get
	\[
	\theta_{t+1} - \theta_t
	= \frac{\alpha}{t} + o\left( \frac{1}{t} \right),
	\]
	which shows the sequence $(\theta_t)$ is eventually non-decreasing.
\end{proof}
We now estimate the asymptotic behavior of the quantities appearing in the assumptions of Proposition \ref{prop:upper_bound} and Proposition \ref{prop:lower_bound}.
\begin{lemma} \label{lem:estimation_key_quantity_1}
	Let $\eta \in \R$. Let $i_t$ be defined as in \eqref{eq:def_i_t_1} and $\theta_t$ as in \eqref{eq:def_theta_t_1}. 
	Then, as $t \to \infty$,
	\begin{equation} \label{eq:key_quantity_1}
	i_t \cdot \exp \left( \sum_{i=2}^{i_t} (e^{\theta_{t_i}} -1) \frac{w_i}{W_i} - \sum_{r=1}^t \theta_r \right) 
	= \exp \left( - \eta t + o(t) \right).
	\end{equation}
	Moreover, as $s \to \infty$, uniformly in $t \geq s$,
	\begin{equation} \label{eq:partial_key_quantity_1}
	\frac{i_t}{i_s} \cdot \exp \left( \sum_{i=i_s+1}^{i_t} (e^{\theta_{t_i}} -1) \frac{w_i}{W_i} - \sum_{r=s+1}^t \theta_r \right) 
	= \exp \left( - \eta (t-s) + o(t-s) \right).
	\end{equation}
\end{lemma}

\begin{proof}
	We focus on the proof of \eqref{eq:partial_key_quantity_1}. Then, \eqref{eq:key_quantity_1} follows from the same lines (or it can be seen as a consequence of \eqref{eq:partial_key_quantity_1} by letting $s \to \infty$ slowly enough in comparison with $t$). The case $s=t$ is obvious so we assume $t \geq s+1$, so that $o(1)$ terms can be included in $o(t-s)$. In the following, $o(\dots)$ terms are meant to hold as $s \to \infty$ uniformly in $t \geq s+1$.
	We first estimate $\sum_{r=s+1}^t \theta_r$, by summing the expansion of $\theta_r$ in \eqref{eq:expansion_theta_t_1}.
	Using standard integral comparison, we have
	\begin{align*}
	\sum_{r=s+1}^t \log r 
	& = t \log t - t - s\log s + s + o(t-s), \\
	\sum_{r=s+1}^t \log \log r 
	& = t \log \log t - s \log \log s + o(t-s).
	\end{align*}
	For the term involving $L$ we use a summation by part to get
	\begin{align*}
	& \sum_{r=s+1}^t \log L(\alpha r \log r) \\
	& = \sum_{r=s+2}^{t-1} r \left[ \log L(\alpha r \log r)  - \log L(\alpha (r+1) \log (r+1)) \right] + t \log L(\alpha t \log t) - s \log L(\alpha (s+1) \log (s+1)) \\
	& = t \log L(\alpha t \log t) - s \log L(\alpha s \log s) + o(t-s),
	\end{align*}
	where we used that $\log L(\alpha r \log r)  - \log L(\alpha (r+1) \log (r+1)) = o(1/r)$ by Lemma~\ref{lem:slowly_varying}\ref{it:slowly_varying_1}, to bound the sum and to replace $\log L(\alpha (s+1) \log (s+1))$ by $\log L(\alpha s \log s)$.
	Combining this yields
	\begin{align*}
	\sum_{r=s+1}^t \theta_r
	& = \alpha t \log t  - (1-\alpha) t \log \log t - t \log L(\alpha t \log t) - t[ \alpha + (1-\alpha) \log \alpha] \\
	& \quad {} - \left( \alpha s \log s  - (1-\alpha) s \log \log s - s \log L(\alpha s \log s) - s[ \alpha + (1-\alpha) \log \alpha] \right) + o(t-s).
	\end{align*}
	On the other hand, using the definition of $a_n$ and then recalling that $e^{\theta_r}  (a_{i_r} - a_{i_{r-1}}) = 1$, we have
	\begin{align*}
	\sum_{i=i_s+1}^{i_t} (e^{\theta_{t_i}} -1) \frac{w_i}{W_i}
	= \sum_{r=s+1}^{t} (e^{\theta_r} -1) (a_{i_r} - a_{i_{r-1}})
	= \sum_{r=s+1}^{t} (1+o(1))
	= t-s + o(t-s).
	\end{align*}
	Combining this with $i_t/i_s = \exp(\ell_t-\ell_s+o(1))$ and the definition of $\ell_t$ in \eqref{eq:def_ell_1} yields \eqref{eq:partial_key_quantity_1}.
\end{proof}

\begin{proof}[Proof of Theorem~\ref{thm:powers_of_log}]
	Let $i_t$ and $\theta_t$ defined as in \eqref{eq:def_i_t_1} and \eqref{eq:def_theta_t_1} for some fixed $\eta \in \R$.
	One can check that
	\begin{align*}
	t_n
	& = \frac{\log n}{\alpha \log \log n} 
	\left( 1 
	+ \frac{\log L(\log n) + \log \log \log n}{\alpha \log \log n - \log L(\log n) - \log \log \log n}
	+ \frac{1+\alpha +\log \alpha + \eta}{\alpha \log \log n}
	+ o \left( \frac{1}{\log \log n} \right)
	\right).
	\end{align*}
	
	\noindent
	\textit{Upper bound.} We take $\eta > 0$ so that it follows from \eqref{eq:key_quantity_1} that \eqref{eq:ass_upper_bound} is satisfied. Hence, we can apply Proposition~\ref{prop:upper_bound} to get $\sup_{n\geq 1}(\haut(\ttu_{n})- t_n)<\infty$ almost surely. Letting $\eta \to 0$ yields the upper bound.%
	\medskip 
	
	\noindent
	\textit{Lower bound.} We fix some $\eta < 0$ and we aim at checking that \eqref{eq:ass_lower_bound} is satisfied to apply Proposition~\ref{prop:lower_bound} with $T(n) \coloneqq (\log n)^{4/5}$.
	Using assumption \eqref{eq:ass_somme_carre}, we have
	\begin{align*}
	\delta_1(n) 
	& = \sum_{r \geq T(n)+1} e^{2\theta_{r}} \sum_{i=i_{r-1}+1}^{i_r} \frac{w_i^2}{W_i^2} 
	\leq \sum_{r \geq T(n)+1} e^{2\theta_{r}} \cdot O \left( \frac{1}{i_{r-1}} \right)
	= \sum_{r \geq T(n)+1} \exp \left( - \alpha r \log r (1+o(1)) \right),
	\end{align*}
	using $i_{r-1} = \exp(\alpha r \log r (1+o(1)))$ and $\theta_r \sim \alpha \log r$.
	It follows that 
	\begin{align*} 
	\sum_{t \geq 1} \delta_1(i_t) 
	= \sum_{t \geq 1} O \left( \exp \left( - \alpha T(i_t) \right) \right)
	= \sum_{t \geq 1} O \left( \exp \left( - \alpha t^{4/5} \right) \right)
	< \infty,
	\end{align*}
	proving the first part of \eqref{eq:ass_lower_bound}.
	We now aim at checking the second part of \eqref{eq:ass_lower_bound}.
	We have $T(i_t) \sim (\alpha t \log t)^{4/5}$ and, recalling $\theta_t \sim \alpha \log t$, $\theta_{T(i_t)} \sim \frac{4}{5} \alpha \log  t$.
	Hence, we get, for any $C > 0$,
	\[
	(\theta_t-\theta_{T(i_t)})\cdot (t-T(i_t))^{3/4} - C \log(t-T(i_t))
	\sim \frac{1}{5} \alpha t^{3/4} \log  t
	\]
	and it follows from the second part of Lemma~\ref{lem:E(n)} combined with \eqref{eq:ass_somme_carre} that
	\begin{align}  \label{eq:lower_bound_E(i_t)_1}
	E(i_t) 
	& \geq \exp\left(-\frac{1}{5} \alpha t^{3/4} \log  t (1+o(1)) \right) 
	- \frac{C}{i_{T(i_t)}}
	\geq \exp\left(-\frac{1}{5} \alpha t^{3/4} \log  t (1+o(1)) \right),
	\end{align}
	because $1/i_{T(i_t)} = \exp(-\alpha (\alpha t \log t)^{4/5} \cdot \frac{4}{5} \log  t \cdot (1+o(1)))$.
	On the other hand, using \eqref{eq:ass_somme_carre} and \eqref{eq:partial_key_quantity_1}, we have
	\begin{align}
	\delta_2(n) 
	& = \sum_{r \geq T(n)+1} e^{2\theta_{r}}
	\cdot \exp \left( \sum_{s=T(n)+1}^{r-1} \theta_s
	- \sum_{i=i_{T(n)}+1}^{i_{r-1}} (e^{\theta_{t_i}} -1) \frac{w_i}{W_i} \right)
	\cdot \sum_{i=i_{r-1}+1}^{i_r} \frac{w_i^2}{W_i^2} \nonumber \\
	& \leq \sum_{r \geq T(n)+1} e^{2\theta_{r}}
	\cdot \frac{i_{r-1}}{i_{T(n)}} \exp \left( \eta (r-1-T(n)) + o(r-1-T(n)) \right)
	\cdot O \left( \frac{1}{i_{r-1}} \right) \nonumber \\
	& = \frac{1}{i_{T(n)}} \sum_{s \geq 0} \exp \left( 2 \theta_{s+T(n)+1} + \eta s + o(s) + O(1) \right) \label{eq:bound_delta_2_1}.
	\end{align}
	Then, by \eqref{eq:expansion_theta_t_1}, for $n$ large enough, we have $\theta_{s+T(n)+1} \leq 2 \alpha \log (s+T(n)+1) \leq 2 \alpha (\log(s+1) + \log T(n))$. Summing over $s$, this yields $\delta_2(n) = O(T(n)^{2\alpha}/i_{T(n)})$.
	Combining this with \eqref{eq:lower_bound_E(i_t)_1}, we get
	\begin{align*}
	\sum_{t \geq 1} \frac{\delta_2(i_t)}{E(i_t)^2} 
	& = \sum_{t \geq 1} \frac{T(i_t)^{2\alpha}}{i_{T(i_t)}} \exp\left(-\frac{1}{5} \alpha t^{3/4} \log  t (1+o(1)) \right) \\
	& = \sum_{t \geq 1} \exp \left( -\alpha (\alpha t \log t)^{4/5} \cdot \frac{4}{5} \log  t \cdot (1+o(1)) \right)
	< \infty,
	\end{align*}
	using previous estmitates to note that $1/i_t$ is the dominating factor.
	This proves the second part of \eqref{eq:ass_lower_bound}.
	Hence, we can apply Proposition~\ref{prop:lower_bound} to get that almost surely, for $n$ large enough, $\haut(\ttT_n) \geq t_n - T(n) - 1$. Letting $\eta \to 0$ yields the lower bound.
\end{proof}

\subsection{Quickly converging sequences}
The goal of this section is to prove Theorem~\ref{thm:quickly}. 
We work under assumption \eqref{eq:assumption_a_n_2} and one of the assumptions \eqref{eq:ass_somme_carre}, \eqref{eq:ass_somme_carre_beta=1} or \eqref{eq:ass_somme_carre_beta>1}, depending on the value of $\beta$. 
We consider the cases $\beta\in\intervalleoo{0}{1}$, $\beta=1$ and $\beta\in \intervalleoo{1}{\infty}$ in separate subsections.
All those subsections are organized in the same fashion as the previous one.


\subsubsection{Case \texorpdfstring{$\beta\in \intervalleoo{0}{1}$}{beta in (0,1)}}

We start by introducing a sequence $(\ell_t)$ meant to satisfy $\ell_t \simeq \exp(\log^\beta i_t)$, which is the quantity appearing in assumption \eqref{eq:assumption_a_n_2} when $n=i_t$.
\begin{lemma} \label{lem:def_ell_t_2}
	Let $\kappa > 0$. 
	For $t \geq 1$, we define 
	\begin{equation} 
	\ell_t \coloneqq \exp\left( \kappa \cdot t^{\frac{\beta}{1-\beta}} \right).
	\end{equation}
	Then, as $t \to \infty$,
	\begin{enumerate}
		\item\label{it:quickly converging log lt - log lt-1} 
		$\log \ell_t - \log \ell_{t-1}
		= \frac{\kappa \beta}{1-\beta} \cdot t^{\frac{2\beta-1}{1-\beta}}\cdot (1+\grandO{\frac{1}{t}})$;
		\item\label{it:quickly converging lt - lt-1} If $\beta \in (0,\frac{1}{2})$, then
		$\ell_t - \ell_{t-1} =  \ell_t \cdot \frac{\kappa \beta}{1-\beta} \cdot t^{\frac{2\beta-1}{1-\beta}} 
		\cdot \left( 1 - \frac{\kappa \beta}{2(1-\beta)} t^{\frac{2\beta-1}{1-\beta}} 
		+ \petito{t^{\frac{2\beta-1}{1-\beta}}} \right)$.
	\end{enumerate}
\end{lemma}
\begin{proof}
	We start by noting that
	\begin{align*}
		t^\frac{\beta}{1-\beta} - (t-1)^{\frac{\beta}{1-\beta}}
		= t^\frac{\beta}{1-\beta} \cdot \Biggl(1 - \left(1-\frac{1}{t}\right)^{\frac{\beta}{1-\beta}}\Biggr) 
		= \frac{\beta}{1-\beta}\cdot t^{\frac{2\beta-1}{1-\beta}} \cdot \left(1 + \grandO{\frac{1}{t}}\right).
	\end{align*}
	Plugging the last display in the expression for $\ell_t$ proves Part \ref{it:quickly converging log lt - log lt-1}. 
	
	Now assume that $\beta \in (0,\frac{1}{2})$. Then, $\frac{2\beta-1}{1-\beta} < 0$, so it follows from \ref{it:quickly converging log lt - log lt-1} that $\log \ell_{t-1} -\log \ell_t \to 0$ or equivalently $(\ell_t-\ell_{t-1})/\ell_t \to 0$.
	Therefore, the following expansion holds
	\begin{align*}
	\frac{\ell_t-\ell_{t-1}}{\ell_t}&=1-\exp(\log \ell_{t-1} - \log \ell_t)\\
	&= -\left(\log \ell_{t-1} -\log \ell_t  \right) - \frac{1}{2}\left(\log \ell_{t-1} -\log \ell_t  \right)^2 + \petito {\left(\log \ell_{t-1} -\log \ell_t  \right)^2}\\ 
	&= \frac{\kappa \beta}{1-\beta} \cdot t^{\frac{2\beta-1}{1-\beta}} 
	\cdot \left( 1- \frac{\kappa \beta}{2(1-\beta)} t^{\frac{2\beta-1}{1-\beta}} 
	+ \petito{t^{\frac{2\beta-1}{1-\beta}}} \right),
	\end{align*}
	using \ref{it:quickly converging log lt - log lt-1} and that $O(t^{-1}) = o(t^{\frac{2\beta-1}{1-\beta}})$.
	This proves \ref{it:quickly converging lt - lt-1}.
\end{proof}
We now define
\begin{equation}\label{eq:def_i_t_2}
i_t \coloneqq 
\begin{cases}
t+1, & \text{if } 0 \leq t < s_0, \\
\left\lfloor \exp\left((\log \ell_t)^{\frac{1}{\beta}}\right) \right\rfloor
= \left\lfloor \exp\left(\kappa^\frac{1}{\beta} t^\frac{1}{1-\beta}\right) \right\rfloor, 
& \text{if } t \geq s_0,
\end{cases}
\end{equation}
by choosing some $s_0$ large enough so that $(i_t)_{t\geq 1}$ is increasing.
Moreover, we will assume that $s_0$ is chosen large enough such that, for $t \geq s_0$, we have $a_{i_t}-a_{i_{t-1}} > 0$.
As in the previous case, the fact that this last condition can be met is a consequence of the proof of the next lemma. 
\begin{lemma} \label{lem:theta_t_2}
	Let $\kappa > 0$. Let $i_t$ be defined as in \eqref{eq:def_i_t_2} and
	\begin{equation} \label{eq:def_theta_t_2}
	\theta_t \coloneqq
	\begin{cases}
	0, & \text{if } 1 \leq t \leq s_0, \\
	-\log(a_{i_t}-a_{i_{t-1}}), & \text{if } t > s_0.
	\end{cases}
	\end{equation}
	Then, there exists $r_0 \geq s_0$ such that $(\theta_t)_{t\geq r_0}$ is a non-decreasing sequence of non-negative numbers.
	Moreover, as $t \to \infty$,
	\begin{equation} \label{eq:expansion_theta_t_2}
	\theta_t \sim (\alpha -1) \kappa \cdot t^{\frac{\beta}{1-\beta}}
	\end{equation}
\end{lemma}

\begin{proof}
	First note that the $o(\dots)$ term appearing in assumption \eqref{eq:assumption_a_n_2} for $a_n$, when taking $n=i_{t-1}$, equals
	\[
		\petito{ \ell_{t-1}^{1-\alpha} \cdot J(\ell_{t-1}) 
			\cdot (\log i_{t-1})^{(4\beta-2)\wedge 0}}
		= \petito{ \ell_{t-1}^{1-\alpha} \cdot J(\ell_{t-1}) 
			\cdot \left( t^{2 \cdot \frac{2\beta-1}{1-\beta}} \wedge 1 \right)},
	\]
	and the one for $n=i_t$ can be included in this one. 	
	Moreover, note that the derivative of $x\mapsto \exp(\log^\beta x)$ is smaller than $1$ for $x\geq 1$ so that from the definition of $\ell_t$ and $i_t$ we have $\lvert \exp(\log^\beta(i_t))-\ell_t \rvert \leq 1$.
	Hence, we get
	\begin{align}\label{eq:ait moins ait-1 en fonction de lt}
		a_{i_t}-a_{i_{t-1}}
= \ell_{t-1}^{1-\alpha} \cdot J(\ell_{t-1}) -\ell_{t}^{1-\alpha} \cdot J(\ell_{t}) +   \petito{ \ell_{t-1}^{1-\alpha} \cdot J(\ell_{t-1}) 
	\cdot \left( t^{2 \cdot \frac{2\beta-1}{1-\beta}} \wedge 1 \right)}
	\end{align}
%
	We now distinguish the three cases $\beta\in \intervalleoo{0}{\frac12}$, $\beta=\frac12$ and $\beta\in\intervalleoo{\frac12}{1}$.
	\medskip
	
	\noindent 
	\textit{Case $\beta\in \intervalleoo{0}{\frac12}$}. 
	In this case, we first use Lemma~\ref{lem:from_J_to_L} and the assumption that $J$ satisfies \hyperref[hyp_SV_k]{$(\mathrm{SV}_{2})$} to deduce that
	there exists a function $L \colon [1,\infty) \to (0,\infty)$ satisfying \hyperref[hyp_SV_k]{$(\mathrm{SV}_1)$} such that $L(x) \sim (\alpha-1) J(x)$ and, for $x$ large enough, $x^{1-\alpha} J(x) = \int_x^\infty y^{-\alpha} L(y) \diff y$, so that we can rewrite 
	\begin{align*}
		\ell_{t-1}^{1-\alpha} \cdot J(\ell_{t-1}) -\ell_{t}^{1-\alpha} \cdot J(\ell_{t})= \int_{\ell_{t-1}}^{\ell_t} y^{-\alpha} L(y) \diff y.
	\end{align*}
	Now, since $\frac{2\beta-1}{1-\beta} < 0$ so it follows from Lemma~\ref{lem:def_ell_t_2}\ref{it:quickly converging log lt - log lt-1} that $\frac{\ell_t}{\ell_{t-1}}\rightarrow 1$ and we can apply Lemma~\ref{lem:integral_slowly_varying} and plug the result back into \eqref{eq:ait moins ait-1 en fonction de lt} to get
	\begin{align*}
		a_{i_t}-a_{i_{t-1}}
		&= \ell_t^{-\alpha} (\ell_t-\ell_{t-1}) L(\ell_t)
		\left( 1 + \frac{\alpha}{2} \cdot \frac{\ell_t-\ell_{t-1}}{\ell_t} (1+o(1))\right) + \petito{ \ell_{t-1}^{1-\alpha} \cdot J(\ell_{t-1}) 
			\cdot t^{2 \cdot \frac{2\beta-1}{1-\beta}}} \\
		&= \ell_t^{1-\alpha} \cdot L(\ell_t) \cdot \frac{\ell_t-\ell_{t-1}}{\ell_t} \cdot  \left( 1 + \frac{\alpha}{2} \cdot \frac{\ell_t-\ell_{t-1}}{\ell_t} + \petito{t^{\frac{2\beta -1}{1-\beta}}}\right),
	\end{align*} 
	where we used that $\frac{\ell_t-\ell_{t-1}}{\ell_t} \sim \frac{\kappa \beta}{1-\beta} \cdot t^{\frac{2\beta-1}{1-\beta}} $ by Lemma~\ref{lem:def_ell_t_2}\ref{it:quickly converging lt - lt-1}, and the fact that $L(x) \sim (\alpha-1) J(x)$ to factor in the error term.
	This ensures that 
	\begin{align*}
		\theta_{t}- \theta_{t-1} 
		& = (\alpha-1)(\log \ell_t - \log \ell_{t-1}) - (\log L(\ell_t)- \log L(\ell_{t-1})) - \log\left(\frac{\ell_t-\ell_{t-1}}{\ell_t}\cdot  \frac{\ell_{t-1}}{\ell_{t-1}-\ell_{t-2}}\right) \\
		& \quad {} - \left(\log\left( 1 + \frac{\alpha}{2} \cdot \frac{\ell_t-\ell_{t-1}}{\ell_t} + \petito{t^{\frac{2\beta -1}{1-\beta}}}\right) - 
		\log \left( 1 + \frac{\alpha}{2} \cdot \frac{\ell_{t-1}-\ell_{t-2}}{\ell_{t-1}} + \petito{t^{\frac{2\beta -1}{1-\beta}}}\right)\right).
	\end{align*}
	The first term is directly handled by Lemma~\ref{lem:def_ell_t_2}\ref{it:quickly converging log lt - log lt-1}. 
	For the second term, by Lemma~\ref{lem:slowly_varying}\ref{it:slowly_varying_1}, we have 
	\begin{align*}
	\log(L(\ell_t))-\log(L(\ell_{t-1}))
	&= \petito{\log(\ell_t)-\log(\ell_{t-1})}
	= \petito{t^\frac{2\beta -1}{1-\beta}}, \label{eq:diff_log_L_2}
	\end{align*}
	using again Lemma~\ref{lem:def_ell_t_2}\ref{it:quickly converging log lt - log lt-1} in the last equality.
	For the third term, Lemma~\ref{lem:def_ell_t_2}\ref{it:quickly converging lt - lt-1} yields
	\begin{align*}
		\log\left(\frac{\ell_t-\ell_{t-1}}{\ell_t}\cdot  \frac{\ell_{t-1}}{\ell_{t-1}-\ell_{t-2}}\right)
		= \log\left(\left(\frac{t}{t-1}\right)^{\frac{2\beta-1}{1-\beta}}
		\cdot 
		\frac{1-\frac{\kappa \beta}{2(1-\beta)} t^{\frac{2\beta-1}{1-\beta}} + \petito{t^{\frac{2\beta-1}{1-\beta}}}}{1-\frac{\kappa \beta}{2(1-\beta)} (t-1)^{\frac{2\beta-1}{1-\beta}} + \petito{t^{\frac{2\beta-1}{1-\beta}}}}\right)
		= \petito{t^{\frac{2\beta-1}{1-\beta}}},
	\end{align*}
	noting that $-1 < \frac{2\beta-1}{1-\beta} < 0$.
	In order to control the last term, we use the expansion $\log(1+x)=x+O(x^2)$  and the fact that the expansion in Lemma~\ref{lem:def_ell_t_2}\ref{it:quickly converging lt - lt-1} gives 
	\[
	\frac{\ell_t-\ell_{t-1}}{\ell_t} - \frac{\ell_{t-1}-\ell_{t-2}}{\ell_{t-1}}= \petito{t^{\frac{2\beta-1}{1-\beta}}}
	\qquad \text{and} \qquad  \left(\frac{\ell_t-\ell_{t-1}}{\ell_t}\right)^2=\petito{t^{\frac{2\beta-1}{1-\beta}}}.
	\]
	Putting everything together we get
	\begin{align*}
		\theta_t-\theta_{t-1} = (\alpha-1) \frac{\kappa \beta}{1-\beta} \cdot  t^{\frac{2\beta-1}{1-\beta}} + \petito{t^\frac{2\beta -1}{1-\beta}},
	\end{align*}
	which is eventually positive when $t$ is large, so the sequence $(\theta_t)_{t\geq s_0}$ is eventually non-decreasing as claimed. 
	The claim on the asymptotic expansion of $\theta_t$ might be obtained for example by summing the expansion that we have for the increments.  
	\medskip
	
	\noindent 
	\textit{Case $\beta=\frac12$}.
	Recalling $\ell_t = e^{\kappa} \ell_{t-1}$, we write
	\[
	 \ell_{t-1}^{1-\alpha} \cdot J(\ell_{t-1}) - \ell_t^{1-\alpha} \cdot J(\ell_t) 
	= \ell_{t-1}^{1-\alpha} \cdot J(\ell_{t-1}) \cdot 
	\left( 1 - e^{(1-\alpha)\kappa} \frac{J(\ell_t)}{J(\ell_{t-1})} \right).
	\]
	Since $J$ satisfies \ref{hyp_SV_0} and $\ell_t = e^{\kappa} \ell_{t-1}$ and $\ell_t\to \infty$ as $t\rightarrow \infty$ we have
	$J(\ell_t)/J(\ell_{t-1}) = 1 + o(1)$, so coming back to \eqref{eq:ait moins ait-1 en fonction de lt} we get
	\[
	a_{i_t}-a_{i_{t-1}}
	= \ell_{t-1}^{1-\alpha} \cdot J(\ell_{t-1}) \cdot 
	\left( 1 - e^{(1-\alpha)\kappa} \right)\cdot  (1+o(1)).
	\]
	Hence, it follows that
	\[
	\theta_t 
	= (\alpha-1) \log \ell_{t-1} - \log J(\ell_{t-1}) - \log \left( 1 - e^{(1-\alpha)\kappa} \right) + o(1).
	\]
	This proves the asymptotics of $\theta_t$ in \eqref{eq:expansion_theta_t_2}.
	It remains to prove that $(\theta_t)_{t\geq 1}$ is eventually non-decreasing:
	using again $\log J(\ell_t) - \log J(\ell_{t-1}) = o(1)$, we have
	$\theta_{t+1} - \theta_t = (\alpha-1) \kappa  + o(1)$,
	which is positive for $t$ large enough.%
	\medskip
	
	\noindent 
	\textit{Case $\beta\in \intervalleoo{\frac12}{1}$}.
In this case we can write
	\begin{align*}
	\ell_{t-1}^{1-\alpha} \cdot J(\ell_{t-1}) - \ell_t^{1-\alpha} \cdot J(\ell_t) 
	= \ell_{t-1}^{1-\alpha} \cdot J(\ell_{t-1}) \cdot (1 + o(1)),
	\end{align*}
	using that $\ell_{t-1} = o(\ell_t)$ and $\log J(x) = o(\log x)$ by Lemma~\ref{lem:slowly_varying}\ref{it:slowly_varying_0}.
	This yields, plugging this back into \eqref{eq:ait moins ait-1 en fonction de lt}, 
	\[
	\theta_t = -\log(a_{i_t}-a_{i_{t-1}})
	= (\alpha-1) \log \ell_{t-1} - \log J(\ell_{t-1}) + o(1).
	\]
	This proves the asymptotics \eqref{eq:expansion_theta_t_2}.
	The fact that $(\theta_t)_{t\geq 1}$ is eventually non-decreasing follows again from $\log J(\ell_t) - \log J(\ell_{t-1}) = o(\log\ell_t-\log\ell_{t-1})$ by Lemma~\ref{lem:slowly_varying}\ref{it:slowly_varying_0}, together with the fact that 
	$\log \ell_t - \log \ell_{t-1} \to \infty$ by Lemma~\ref{lem:def_ell_t_2}\ref{it:quickly converging log lt - log lt-1}.
\end{proof}

\begin{lemma} \label{lem:estimation_key_quantity_2}
	Let $\kappa > 0$ be such that $\eta \coloneqq \kappa^{\frac{1}{\beta}}-(\alpha-1)(1-\beta) \kappa$ is non-zero. 
	Let $i_t$ be defined as in \eqref{eq:def_i_t_2} and $\theta_t$ as in \eqref{eq:def_theta_t_2}. 
	Then, as $t \to \infty$,
	\begin{equation} \label{eq:key_quantity_2}
	i_t \cdot \exp \left( \sum_{i=2}^{i_t} (e^{\theta_{t_i}} -1) \frac{w_i}{W_i} - \sum_{r=1}^t \theta_r \right) 
	= \exp \left( \eta \cdot t^{\frac{1}{1-\beta}} \cdot (1+o(1)) \right).
	\end{equation}
	Moreover, as $s \to \infty$, uniformly in $t \geq s$,
	\begin{equation} \label{eq:partial_key_quantity_2}
	\frac{i_t}{i_s} \cdot \exp \left( \sum_{i=i_s+1}^{i_t} (e^{\theta_{t_i}} -1) \frac{w_i}{W_i} - \sum_{r=s+1}^t \theta_r \right) 
	= \exp \left( \eta \cdot \left( t^{\frac{1}{1-\beta}} - s^{\frac{1}{1-\beta}} \right) \cdot (1+o(1)) \right).
	\end{equation}
\end{lemma}

\begin{proof}
	We focus on the proof of \eqref{eq:partial_key_quantity_2}, because \eqref{eq:key_quantity_2} follows from the same lines. The case $s=t$ is obvious so we assume $t \geq s+1$. 
	The following asymptotics are meant to hold as $s \to \infty$, uniformly in $t \geq s+1$.
	Summing the asymptotic equivalent in \eqref{eq:expansion_theta_t_2}, we get 
	\begin{align*}
		\sum_{r=s+1}^t \theta_r
		\sim (\alpha - 1) \kappa \cdot \int_s^t x^{\frac{\beta}{1-\beta}} \diff x
		\sim (\alpha - 1) (1-\beta) \kappa \cdot
		\left( t^{\frac{1}{1-\beta}} - s^{\frac{1}{1-\beta}} \right).
	\end{align*}
	On the other hand, it follows from $e^{\theta_r} (a_{i_r} - a_{i_{r-1}})  = 1$ that $\sum_{i=i_s+1}^{i_t} (e^{\theta_{t_i}} -1) \frac{w_i}{W_i} \sim t-s = o(t^{\frac{1}{1-\beta}} - s^{\frac{1}{1-\beta}})$.
	Then, using the definition of $i_t$, we get \eqref{eq:partial_key_quantity_2}.
\end{proof}

\begin{proof}[Proof of Theorem~\ref{thm:quickly}\ref{it:asymptotic height beta in (0,1)}]
	Let $i_t$ and $\theta_t$ defined as in \eqref{eq:def_i_t_2} and \eqref{eq:def_theta_t_2} for some fixed $\kappa > 0$.
	It follows from the definition of $i_t$ that 
	\[
	t_n = \kappa^{-\frac{1-\beta}{\beta}} (\log n)^{1-\beta}.
	\]
	Moreover, we can check that the quantity $\eta = \kappa^{\frac{1}{\beta}}-(\alpha-1)(1-\beta) \kappa$ appearing in Lemma~\ref{lem:estimation_key_quantity_2} is of the same sign as $\kappa - (\alpha -1)^{\frac{\beta}{1-\beta}}(1- \beta)^{\frac{\beta}{1-\beta}}$. 
	\medskip 
	
	\noindent
	\textit{Upper bound.} We choose $\kappa <(\alpha -1)^{\frac{\beta}{1-\beta}}(1- \beta)^{\frac{\beta}{1-\beta}}$ so that $\eta < 0$ and therefore \eqref{eq:ass_upper_bound} follows from \eqref{eq:key_quantity_2}. Hence, we can apply Proposition~\ref{prop:upper_bound} to get 
	$\haut(\ttT_{n})\leq t_n + O(1)$ almost surely. Letting $\kappa \to (\alpha -1)^{\frac{\beta}{1-\beta}}(1- \beta)^{\frac{\beta}{1-\beta}}$ yields the upper bound.
	\medskip 
	
	\noindent
	\textit{Lower bound.} We fix some $\kappa > (\alpha -1)^{\frac{\beta}{1-\beta}}(1- \beta)^{\frac{\beta}{1-\beta}}$, so that $\eta > 0$, and we aim at checking that \eqref{eq:ass_lower_bound} is satisfied to apply Proposition~\ref{prop:lower_bound} with $T(n) \coloneqq (\log n)^{(1-\beta)^{3/2}}$.
	Using assumption \eqref{eq:ass_somme_carre}, we have
	\begin{align*}
	\delta_1(n) 
	& = \sum_{r \geq T(n)+1} e^{2\theta_{r}} \sum_{i=i_{r-1}+1}^{i_r} \frac{w_i^2}{W_i^2} 
	\leq \sum_{r \geq T(n)+1} e^{2\theta_{r}} \cdot O \left( \frac{1}{i_{r-1}} \right)
	= \sum_{r \geq T(n)+1} \exp \left( - \kappa^\frac{1}{\beta} r^\frac{1}{1-\beta} (1+o(1)) \right),
	\end{align*}
	using \eqref{eq:def_i_t_2} and \eqref{eq:expansion_theta_t_2}.
	It follows that 
	\begin{align*} 
	\sum_{t \geq 1} \delta_1(i_t) 
	= \sum_{t \geq 1} O \left( \exp \left( - \kappa^\frac{1}{\beta} T(i_t)^\frac{1}{1-\beta} \right) \right)
	= \sum_{t \geq 1} O \left( \exp \left( - \kappa^{(1+(1-\beta)^{1/2})/\beta} \cdot t^{(1-\beta)^{-1/2}} \right) \right)
	< \infty,
	\end{align*}
	proving the first part of \eqref{eq:ass_lower_bound}.
	We now aim at checking the second part of \eqref{eq:ass_lower_bound}.
	It follows from the first part of Lemma~\ref{lem:E(n)} that $E(i_t) \geq \exp(-Ct)$.
	For $\delta_2(n)$ we proceed as in \eqref{eq:bound_delta_2_1}: using \eqref{eq:ass_somme_carre} and \eqref{eq:partial_key_quantity_2}, we have
	\begin{align*}
	\delta_2(n) 
	& \leq \sum_{r \geq T(n)+1} e^{2\theta_{r}}
	\cdot \frac{i_{r-1}}{i_{T(n)}} 
	\exp \left( 
	-\eta \cdot \left( r^{\frac{1}{1-\beta}} - (T(n)+1)^{\frac{1}{1-\beta}} \right) \cdot (1+o(1)) 
	\right)
	\cdot O \left( \frac{1}{i_{r-1}} \right) \\
	& \leq \frac{O(1)}{i_{T(n)}} \sum_{s \geq 0} 
	\exp \left( 
	4 (\alpha-1)\kappa \cdot (s+T(n)+1)^{\frac{\beta}{1-\beta}}
	- \frac{\eta}{2} \cdot \left( (s+T(n)+1)^{\frac{1}{1-\beta}} - (T(n)+1)^{\frac{1}{1-\beta}} \right) 
	\right),
	\end{align*}
	for $n$ large enough, using \eqref{eq:expansion_theta_t_2} to bound $\theta_r$.
	Then, for any real $a,b >0$ and integer $M \geq 1$, we have
	\begin{align*}
	& \sum_{s \geq 0} 
	\exp \left( 
	a (s+M)^{\frac{\beta}{1-\beta}}
	- b \cdot \left( (s+M)^{\frac{1}{1-\beta}} - M^{\frac{1}{1-\beta}} \right) 
	\right) \\
	& \leq \sum_{s = 0}^{M-1} \exp \left( a (2M)^{\frac{\beta}{1-\beta}} \right)
	+ \sum_{s \geq M} 
	\exp \left( 
	a (2s)^{\frac{\beta}{1-\beta}}
	- b \cdot s^{\frac{1}{1-\beta}}
	\right) 
	\leq M \exp \left( a (2M)^{\frac{\beta}{1-\beta}} \right) + C(a,b),
	\end{align*}
	bounding the second sum by the same sum starting at $s=0$.
	It follows that
	\begin{align*}
	\delta_2(n) 
	& \leq \frac{1}{i_{T(n)}} \exp \left( \grandO{T(n)^{\frac{\beta}{1-\beta}}} \right)
	= \exp \left( \grandO{T(n)^{\frac{\beta}{1-\beta}}} - \kappa^{\frac{1}{\beta}} T(n)^{\frac{1}{1-\beta}} \right).
	\end{align*}
	Recalling that $E(i_t) \geq \exp(-Ct)$ and $T(i_t) = (\log i_t)^{(1-\beta)^{3/2}} \sim \kappa^{(1-\beta)^{3/2}/\beta} \cdot t^{(1-\beta)^{1/2}}$, the second part of \eqref{eq:ass_lower_bound} follows.
	Hence, we can apply Proposition~\ref{prop:lower_bound} to get that almost surely, for $n$ large enough, $\haut(\ttT_n) \geq t_n - T(n) - 1$. Letting $\kappa \to (\alpha -1)^{\frac{\beta}{1-\beta}}(1- \beta)^{\frac{\beta}{1-\beta}}$ yields the lower bound.
\end{proof}

\subsubsection{Case \texorpdfstring{$\beta=1$}{beta=1}}

Let $\kappa > 1$. 
For $t \geq 1$, we define 
\begin{equation} 
\ell_t \coloneqq \exp\left(\kappa^t \right)
\end{equation}
and 
\begin{equation}\label{eq:def_i_t beta=1}
i_t \coloneqq 
\begin{cases}
t+1, & \text{if } 0 \leq t < s_0, \\
\lfloor \ell_t \rfloor= \lfloor \exp \left(\kappa^t \right) \rfloor, 
& \text{if } t \geq s_0,
\end{cases}
\end{equation}
by choosing some $s_0$ large enough so that $(i_t)_{t\geq 1}$ is increasing.

\begin{lemma} \label{lem:theta_t beta=1}
	Let $\kappa > 1$. Let $i_t$ be defined as in \eqref{eq:def_i_t beta=1} and
	\begin{equation} \label{eq:def_theta_t beta=1}
	\theta_t \coloneqq
	\begin{cases}
	0, & \text{if } 1 \leq t \leq s_0, \\
	-\log(a_{i_t}-a_{i_{t-1}}), & \text{if } t > s_0.
	\end{cases}
	\end{equation}
	Then, there exists $r_0 \geq s_0$ such that $(\theta_t)_{t\geq r_0}$ is a non-decreasing sequence of non-negative numbers.
	Moreover, as $t \to \infty$,
	\begin{equation} \label{eq:expansion_theta_t beta=1}
	\theta_t \sim (\alpha -1) \kappa^{t-1}.
	\end{equation}
\end{lemma}
\begin{proof}
	The proof is similar to the case $\beta \in (\frac{1}{2},1)$. We start by writing
	\begin{align*}
	\ell_{t-1}^{1-\alpha} \cdot J(\ell_{t-1}) - \ell_t^{1-\alpha} \cdot J(\ell_t) 
	= \ell_{t-1}^{1-\alpha} \cdot J(\ell_{t-1}) \cdot (1 + o(1)),
	\end{align*}
	using that $\log J(x) = o(\log x)$ by Lemma~\ref{lem:slowly_varying}\ref{it:slowly_varying_0}.
	This allows us to write, using assumption \eqref{eq:assumption_a_n_2},
	$a_{i_t}-a_{i_{t-1}}= \ell_{t-1}^{1-\alpha} \cdot J(\ell_{t-1}) \cdot (1 + o(1))$.
	This entails that 
	\begin{align*}
	\theta_t = -\log(a_{i_t}-a_{i_{t-1}})
	= (\alpha -1) \kappa^{t-1} - \log J(\exp(\kappa^{t-1}))
 + \petito{1}
	\sim (\alpha -1) \kappa^{t-1}.
	\end{align*}
	This asymptotic expansion already ensures that the sequence $\theta_t$ is eventually non-decreasing so the lemma is proved.
\end{proof}

\begin{lemma} \label{lem:estimation_key_quantity beta=1}
	Let $\kappa > 1$ such that $\kappa \neq \alpha$.
	Let $i_t$ be defined as in \eqref{eq:def_i_t beta=1} and $\theta_t$ as in \eqref{eq:def_theta_t beta=1}. 
	Then, as $t \to \infty$,
	\begin{equation} \label{eq:key_quantity beta=1}
	i_t \cdot \exp \left( \sum_{i=2}^{i_t} (e^{\theta_{t_i}} -1) \frac{w_i}{W_i} - \sum_{r=1}^t \theta_r \right) 
	= \exp \left( \frac{\kappa -\alpha}{\kappa -1}\cdot  \kappa^t \cdot (1+o(1)) \right).
	\end{equation}
	Moreover, as $s \to \infty$, uniformly in $t \geq s$,
	\begin{equation} \label{eq:partial_key beta=1}
	\exp \left( \sum_{r=s+1}^t \theta_r - \sum_{i=i_s+1}^{i_t} (e^{\theta_{t_i}} -1) \frac{w_i}{W_i} \right) 
	= \exp \left( \frac{\alpha-1}{\kappa -1}\cdot (\kappa^t-\kappa^s) \cdot (1+\petito{1}) \right).
	\end{equation}
\end{lemma}
\begin{proof}
	This follows from the fact that 
	\begin{align*}
		\sum_{r=1}^{t}\theta_r = \sum_{r=1}^{t} (\alpha -1) \kappa^{r-1} (1+\petito{1})= \frac{\alpha-1}{\kappa -1}\cdot  \kappa^t \cdot(1+\petito{1}) 
	\end{align*}
	and similarly for $t\geq s+1$, (the case $t=s$ can be treated aside),
	\begin{align*}
		 \sum_{r=s+1}^{t}\theta_r = \sum_{r=s+1}^{t} (\alpha -1) \kappa^{r-1} (1+\petito{1})= \frac{\alpha-1}{\kappa -1}\cdot (\kappa^t-\kappa^s)\cdot (1+\petito{1}),
	\end{align*}
the definition of $i_t$ in \eqref{eq:def_i_t beta=1}, and the fact that $\sum_{i=i_s+1}^{i_t} (e^{\theta_{t_i}} -1) \frac{w_i}{W_i} =t-s + o(t-s) = o(\kappa^t-\kappa^s)$. 
\end{proof}

\begin{proof}[Proof of Theorem~\ref{thm:quickly}\ref{it:asymptotic height beta=1}]
	Let $i_t$ and $\theta_t$ defined as in \eqref{eq:def_i_t beta=1} and \eqref{eq:def_theta_t beta=1} for some fixed $\kappa > 1$.
It follows from the definition of $i_t$ that 
\[
t_n = \frac{\log \log n}{\log \kappa} + O(1).
\]

\noindent
\textit{Upper bound.} We choose $\kappa <\alpha$ and then \eqref{eq:ass_upper_bound} follows from \eqref{eq:key_quantity beta=1}. Hence, we can apply Proposition~\ref{prop:upper_bound} to get 
$\haut(\ttT_{n})\leq t_n + O(1)$ almost surely. Letting $\kappa \to \alpha$ yields the upper bound.
\medskip 

\noindent
\textit{Lower bound.} We fix some $\kappa > \alpha$ and we aim at checking that \eqref{eq:ass_lower_bound} is satisfied to apply Proposition~\ref{prop:lower_bound} with $T(n) \coloneqq (\log \log n)^{1/2}$.
Using \eqref{eq:ass_somme_carre_beta=1}, for any $\varepsilon>0$, we have
\begin{align*}
\delta_1(n) 
= \sum_{r \geq T(n)+1} e^{2\theta_{r}} \cdot O \left( \frac{1}{i_{r-1}^{2\alpha-1-\varepsilon}} \right)
= \sum_{r \geq T(n)+1} O \left( \exp( - (1-\varepsilon) \kappa^{r-1}(1+o(1)) \right),
\end{align*}
using \eqref{eq:expansion_theta_t beta=1}.
Choosing $\varepsilon = 1/4$, this yields $\delta_1(i_t) = O( \exp( - \kappa^{T(i_t)} /2))$, which is summable in $t$ since $T(i_t) \sim \sqrt{t \log \kappa}$. 
This proves the first part of \eqref{eq:ass_lower_bound}.
For the second part, we first have $E(i_t) \geq \exp(-Ct)$ by the first part of Lemma~\ref{lem:E(n)}.
Then, using \eqref{eq:ass_somme_carre_beta=1} as for $\delta_1(n)$ and \eqref{eq:partial_key beta=1}, we have, for any $\varepsilon > 0$,
\begin{align*}
\delta_2(n) 
	& = \sum_{r \geq T(n)+1} 
	\exp \left( \sum_{s=T(n)+1}^{r-1} \theta_s
	- \sum_{i=i_{T(n)}+1}^{i_{r-1}} (e^{\theta_{t_i}} -1) \frac{w_i}{W_i} \right)
	\cdot e^{2\theta_{r}} \sum_{i=i_{r-1}+1}^{i_r} \frac{w_i^2}{W_i^2} \\
	& = \sum_{r \geq T(n)+1} 
	\exp \left( \frac{\alpha-1}{\kappa -1}\cdot (\kappa^{r-1}-\kappa^{T(n)}) (1+\petito{1}) - (1-\varepsilon) \kappa^{r-1}(1+o(1)) \right) \\
	& \leq \sum_{r \geq T(n)+1} 
	\exp \left( -\left( \frac{\kappa-\alpha}{\kappa -1} - \varepsilon \right) \cdot \kappa^{r-1} (1+\petito{1}) \right).
\end{align*}
Hence, choosing $\varepsilon = \frac{1}{3} \cdot \frac{\kappa-\alpha}{\kappa -1}$, we get 
$\delta_2(n) = O(\exp (-\varepsilon \kappa^{T(n)}))$.
It follows easily that $\sum_{t\geq 1}\delta_2(i_t)/E(i_t) < \infty$, which proves the second part of \eqref{eq:ass_lower_bound}.
Applying Proposition~\ref{prop:lower_bound} and letting $\kappa \to \alpha$, we get the lower bound.

\end{proof}

\subsubsection{Case \texorpdfstring{$\beta>1$}{beta>1}}

Let $\kappa > 1$. 
For $t \geq 1$, we define 
\begin{equation} 
\ell_t \coloneqq \exp \left( \exp(\beta \kappa^t) \right)
\end{equation}
and 
\begin{equation}\label{eq:def_i_t beta>1}
i_t \coloneqq 
\begin{cases}
t+1, & \text{if } 0 \leq t < s_0, \\
\lfloor \exp(\log^{1/\beta} \ell_t) \rfloor= \lfloor \exp \left( \exp(\kappa^t) \right) \rfloor, 
& \text{if } t \geq s_0,
\end{cases}
\end{equation}
by choosing some $s_0$ large enough so that $(i_t)_{t\geq 1}$ is increasing.

\begin{lemma} \label{lem:theta_t beta>1}
	Let $\kappa > 1$. Let $i_t$ be defined as in \eqref{eq:def_i_t beta>1} and
	\begin{equation} \label{eq:def_theta_t beta>1}
	\theta_t \coloneqq
	\begin{cases}
	0, & \text{if } 1 \leq t \leq s_0, \\
	-\log(a_{i_t}-a_{i_{t-1}}), & \text{if } t > s_0.
	\end{cases}
	\end{equation}
	Then, there exists $r_0 \geq s_0$ such that $(\theta_t)_{t\geq r_0}$ is a non-decreasing sequence of non-negative numbers.
	Moreover, as $t \to \infty$,
	\begin{equation} \label{eq:expansion_theta_t beta>1}
	\theta_t 
	= (\alpha -1) \log \ell_{t-1} - \log J(\ell_{t-1}) +O(1).
	\end{equation}
\end{lemma}
\begin{proof}
	Proceeding as in the proof of Lemma~\ref{lem:theta_t beta=1}, we get $a_{i_t}-a_{i_{t-1}} = \ell_{t-1}^{1-\alpha} \cdot J(\ell_{t-1}) \cdot (1 + o(1))$. The result follows by taking the logarithm.
	In particular, $\theta_t \sim (\alpha -1) \exp(\beta \kappa^{t-1})$, so it is eventually non-negative and non-decreasing.
\end{proof}
\begin{lemma} \label{lem:estimation_key_quantity beta>1}
	Let $\kappa > 1$.
	Let $i_t$ be defined as in \eqref{eq:def_i_t beta>1} and $\theta_t$ as in \eqref{eq:def_theta_t beta>1}. 
	Then, as $t \to \infty$,
	\begin{equation} \label{eq:key_quantity beta>1}
	i_t \cdot \exp \left( \sum_{i=2}^{i_t} (e^{\theta_{t_i}} -1) \frac{w_i}{W_i} - \sum_{r=1}^t \theta_r \right) 
	= \exp \left(  e^{\kappa^t} - (\alpha -1) e^{\beta \kappa^{t-1}} (1+o(1)) \right).
	\end{equation}
	Moreover, as $s \to \infty$, uniformly in $t \geq s$,
	\begin{equation} \label{eq:partial_key beta>1}
	\exp \left( \sum_{r=s+1}^t \theta_r - \sum_{i=i_s+1}^{i_t} (e^{\theta_{t_i}} -1) \frac{w_i}{W_i} \right) 
	= \exp \left( (\alpha -1) e^{\beta \kappa^{t-1}} (1+o(1)) \1_{t \geq s+1} \right).
	\end{equation}
\end{lemma}
\begin{proof}
	Recall that, as for the other cases, we have $\sum_{i=i_s+1}^{i_t} (e^{\theta_{t_i}} -1) \frac{w_i}{W_i} = O(t)$ for any $s \leq t$.
	Then, the two claims are obtained by summing the asymptotic equivalent $\theta_r \sim (\alpha -1) \exp(\beta \kappa^{r-1})$.
\end{proof}
\begin{proof}[Proof of Theorem~\ref{thm:quickly}\ref{it:asymptotic height beta>1}]
	Let $i_t$ and $\theta_t$ defined as in \eqref{eq:def_i_t beta>1} and \eqref{eq:def_theta_t beta>1} for some fixed $\kappa > 1$.
	It follows from the definition of $i_t$ that 
	\[
	t_n = \frac{\log \log \log n}{\log \kappa} + O(1).
	\]
	
	\noindent
	\textit{Upper bound:} We choose $\kappa < \beta$ so that $e^{\kappa^t} = o( (\alpha -1) e^{\beta \kappa^{t-1}})$. Therefore, \eqref{eq:ass_upper_bound} follows from \eqref{eq:partial_key beta>1} with $s=1$. 
	Applying Proposition~\ref{prop:upper_bound} and then letting $\kappa \to \beta$ yields the upper bound.
	\medskip 
	
	\noindent
	\textit{Lower bound:} We fix some $\kappa > \beta$ and we aim at checking that \eqref{eq:ass_lower_bound} is satisfied to apply Proposition~\ref{prop:lower_bound} with $T(n) \coloneqq \log \log \log \log n$.
	Using assumption \eqref{eq:ass_somme_carre_beta>1}, there is $\varepsilon>0$ such that, as $r \to \infty$,
	\begin{equation} \label{eq:control_square}
		\sum_{i=i_{r-1}+1}^{i_r} \frac{w_i^2}{W_i^2} 
		= \grandO{\frac{1}{i_{r-1}^\varepsilon} 
			\ell_{r-1}^{-2(\alpha-1)} J^2\left( \ell_{r-1} \right)}
		= \grandO{ \frac{1}{i_{r-1}^\varepsilon} e^{-2\theta_r} }, 
	\end{equation}
	using \eqref{eq:expansion_theta_t beta>1}.
	Applying this, we get
	\begin{align*}
	\delta_1(n) 
	= \sum_{r \geq T(n)+1} \grandO{ \frac{1}{i_{r-1}^\varepsilon}}
	= \grandO{ \sum_{r \geq T(n)+1} \exp \left( -\varepsilon e^{\kappa^r} \right) } 
	= \grandO{ \exp \left( -\varepsilon e^{\kappa^{T(n)}} \right) }.
	\end{align*}
	Since $T(i_t) \sim \log t$, it follows that $\sum_{t\geq 1}\delta_1(i_t)<\infty$, which proves the first part of \eqref{eq:ass_lower_bound}.
	We now aim at checking the second part of \eqref{eq:ass_lower_bound}.
	It follows from the first part of Lemma~\ref{lem:E(n)} that $E(i_t) \geq \exp(-Ct)$.
	For $\delta_2(n)$, using \eqref{eq:partial_key beta>1} and \eqref{eq:control_square}, we get
	\begin{align*}
	\delta_2(n) 
	& = \sum_{r \geq T(n)+1} e^{2\theta_{r}}
	\cdot \exp \left( \sum_{s=T(n)+1}^{r-1} \theta_s
	- \sum_{i=i_{T(n)}+1}^{i_{r-1}} (e^{\theta_{t_i}} -1) \frac{w_i}{W_i} \right)
	\cdot \sum_{i=i_{r-1}+1}^{i_r} \frac{w_i^2}{W_i^2} \\
	& = \sum_{r \geq T(n)+1} \exp \left( (\alpha -1) e^{\beta \kappa^{r-2}} (1+o(1)) \1_{t \geq s+1} \right) \cdot \grandO{ \frac{1}{i_{r-1}^\varepsilon}} \\
	& = \sum_{r \geq T(n)+1} \exp \left( -\varepsilon e^{\kappa^{r-1}} (1+o(1)) \right),
	\end{align*}
	using that $e^{\beta \kappa^{r-2}} =o(e^{\kappa^{r-1}})$, since $\kappa > \beta$.
	It is then easy to check that $\sum_{t\geq 1}\delta_2(i_t)/E(i_t) < \infty$, proving the second part of \eqref{eq:ass_lower_bound}.
	Applying Proposition~\ref{prop:lower_bound} and letting $\kappa \to \beta$, we get the lower bound.
\end{proof}

\appendix

\section{Crude bounds for the height}\label{sec:crude_bounds}

\subsection{Crude upper bound}\label{sec:crude upper bound}

In this section, we use the many-to-one lemma in a naive way (without a barrier) in order to get some crude upper bound on the height of $\ttT_n$.
As explained below, this amounts to bounding $\haut(\ttT_n)$ by the maximum of $n$ independent copies of $\sum_{i=1}^{n} B_i$, where the $B_i$'s are independent $\mathrm{Bernoulli}(w_i/W_i)$ random variables.
We then compare the obtained upper bound with the actual behavior of $\haut(\ttT_n)$ in the regimes that are studied in Theorem~\ref{thm:powers_of_log} and Theorem~\ref{thm:quickly}.


\paragraph{Crude upper bound approach.}
We first work under no restriction on the weight sequence $(w_i)$. 
Applying the result contained in Remark~\ref{rem:consequence_many-to-one} to the function $f(y)=\ind{y\geq x}$, we get 
\[
	\Pp{\haut(\ttu_n)\geq x}= \Pp{H_{n-1}+1\geq x},
\]
where we recall that we defined for $k \geq 1$, $H_k=\sum_{i=2}^k B_i$ with $(B_i)_{i\geq 1}$ a sequence of independent random variables such that $B_i$ has distribution $\mathrm{Bernoulli}(w_i/W_i)$.
Using a union-bound on all the vertices yields
\begin{equation} \label{eq:comp_iid}
	\Pp{\haut(\ttT_n)\geq x} \leq \sum_{i=1}^{n}\Pp{\haut(\ttu_i)\geq x}=\sum_{i=1}^{n}\Pp{H_{i-1}+1\geq x} \leq n \cdot \Pp{H_{n}+1\geq x},
\end{equation}
where for the last inequality we used that $(H_k)_{k\geq 1}$ is non-decreasing.
Note that the right-hand side of \eqref{eq:comp_iid} is small if and only if $x$ is larger than the maximum of $n$ independent copies of $H_{n}+1 = \sum_{i=1}^{n} B_i$.
Even if the inequalities in \eqref{eq:comp_iid} are not optimal, there are not so far from the truth in the regimes considered in this paper (see Footnote \ref{footnote:crude}).

We now have to estimate the maximum of $n$ independent copies of $\sum_{i=1}^{n-1} B_i$, that we denote by~$M_n$, in the different regimes considered in Theorem \ref{thm:powers_of_log} and Theorem \ref{thm:quickly}.
This is done via standard techniques and we omit the calculations, even if they can be tedious in some of the cases.

\paragraph{Variance varying like a power of $\log n$.} 
We consider here the framework of Theorem \ref{thm:powers_of_log}.
First assume that $\alpha \in (0,1)$.
In that case one can check that
\begin{align*}
M_n 
& = \frac{\log n}{\alpha \log \log n} 
+ \frac{\log n}{(\alpha \log \log n)^2} 
\biggl( 
\sum_{k\geq 0} \left( \frac{\log L(\log n)}{\alpha \log \log n} \right)^k
(\log L(\log n) + (k+1) \log \log \log n)  \\
& \hspace{10.5cm} {} + 1 + \log \alpha - \log(1-\alpha) + o(1) \biggr).
\end{align*}
The single difference with the expansion of $\haut(\ttT_n)$ in Theorem \ref{thm:powers_of_log} appears in the coefficient of the term of order $\log n / (\log \log n)^2$, which is $1 + \log \alpha - \log(1-\alpha)$ instead of $1 + \log \alpha - \alpha$.

Now assume that $\alpha = 1$ and $a_n$ diverges. 
Assume also that $a_{\log n} = o(a_n)$ as $n \to \infty$%
\footnote{We do not discuss the opposite case here, for which it is tedious to get an expansion for $M_n$. However, it is easier to check that the claim in Remark \ref{rem:comparison_iid} still holds.\label{footnote:cas_chiant}}.
Let $J(x) \coloneqq \int_1^x \frac{L(u)}{u} \diff u$, so that $a_n \sim J (\log n)$.
One can check that $L(x) = o(J(x))$ as $x \to \infty$ and it follows that $J$ satisfies \hyperref[hyp_SV_k]{$(\mathrm{SV}_1)$}.
Then, the following expansion holds.
\begin{align*}
M_n & = \frac{\log n}{\log \log n} 
+ \frac{\log n}{(\log \log n)^2} 
\biggl( 
\sum_{k\geq 0} \left( \frac{\log J(\log n)}{\log \log n} \right)^k
(\log J(\log n) + (k+1) \log \log \log n) + 1 + o(1) \biggr).
\end{align*}
Therefore, the first order is the same for $M_n$ and $\haut(\ttT_n)$, but the order of magnitude of the difference is greater than $\log n/(\log \log n)^2$ because $\log J(x) - \log L(x) \to \infty$. 

Assume again $\alpha = 1$, but now with $a_n$ converging to some finite limit $a_\infty$.
Assume also that $a_\infty - a_n = o(a_\infty - a_{\log n})$ (the same comment as in Footnote \ref{footnote:cas_chiant} holds).
Let $I(x) \coloneqq \int_x^\infty \frac{L(u)}{u} \diff u$, one can check as for $J$ that $L(x) = o(I(x))$ as $x \to \infty$ and therefore $I$ satisfies \hyperref[hyp_SV_k]{$(\mathrm{SV}_1)$}. 
In this case, we get
\begin{align*}
M_n & = \frac{\log n}{\log \log n} 
+ \frac{\log n}{(\log \log n)^2} 
\biggl( 
\log I(\log \log n) + \log \log \log n + 1 + o(1) \biggr).
\end{align*}
Again, the difference with $\haut(\ttT_n)$ is not at the first order but is greater than $\log n/(\log \log n)^2$ because $\log I(\log x) \geq \log I(x)$ and $\log I(x) - \log L(x) \to \infty$.

Finally, in the case $\alpha > 1$, we can check that
\begin{equation} \label{eq:first_order_M_n}
M_n \sim \frac{\log n}{\log \log n},
\end{equation}
which differs from the first order of $\haut(\ttT_n)$ by a multiplicative constant.

\paragraph{Quickly converging variance.} We consider here the framework of Theorem \ref{thm:powers_of_log}.
If $\beta >1$, we get the same behavior as in \eqref{eq:first_order_M_n}.
If $\beta =1$, we have
\[
	M_n \sim \frac{\log n}{\alpha \log \log n},
\]
where $\alpha > 1$ is the constant appearing in assumption \eqref{eq:assumption_a_n_2}.
If $\beta > 1$, we get
\[
M_n \sim \frac{\log n}{(\alpha-1) (\log \log n)^\beta}.
\]
In all these cases, the upper bound is not even the right order of magnitude for $\haut(\ttT_n)$ since we have $\haut(\ttT_n)=o(M_n)$ almost surely.

\subsection{Crude lower bound}\label{sec:crude lower bound}

In this section, we present a simple method for proving lower bounds for the height of $\ttT_n$, which surprisingly gives a correct first order in the regimes studied in this paper.
Assume for simplicity that the sequence $(w_i)_{i\geq 1}$ is non-increasing and that
all weights $w_i$ are non-zero.

The strategy is the following. We build a path in the tree $\ttT_n$ starting from the root, by then considering its first child and then the first child of the latter, and so on.
More formally, we define a random sequence of indices $(I_r)_{r\geq 1}$ recursively as follows: let $I_0 = 1$ and, for any $r \geq 0$, $I_{r+1}$ is the index of the first vertex attached to $\ttu_{I_r}$. 
Note that this sequence is a.s.\@ well-defined: any vertex $\ttu_i$ has infinitely many children a.s.\@ by Borel--Cantelli lemma and the fact that $\sum_{k>i} w_i/W_k = \infty$ in the regime where $W_k$ grows sub-polynomially.
Finding a lower bound for the length of this path provides a lower bound for $\haut(\ttT_n)$.

To this end, we note that, for any $r \geq 1$ and $t \geq 0$,
\[
\Ppsq{I_{r+1} - I_r > t}{I_r} 
= \prod_{i=I_r+1}^{I_r+t} \left( 1 - \frac{w_{I_r}}{W_i} \right) \\
\leq \exp \left( - \sum_{i=I_r+1}^{I_r+t} \frac{w_{I_r}}{W_i} \right).
\]
Then, for any increasing sequence $(i_r)_{r\geq 0}$ such that $i_0 = 1$, we have, noting that $I_0 = 1 = i_0$,
\begin{align}
\Pp{\exists r \geq 0 : I_r > i_r}
& \leq \sum_{r \geq 0} \Pp{ I_{r+1} > i_{r+1}, I_r \leq i_r} 
= \sum_{r \geq 0} \Ec{\Ppsq{I_{r+1}-I_r > i_{r+1}-I_r}{I_r} \1_{I_r \leq i_r}} 
\nonumber \\
& \leq \sum_{r \geq 0} \Ec{\exp \left( - \sum_{i=I_r+1}^{i_{r+1}} \frac{w_{I_r}}{W_i} \right) \1_{I_r \leq i_r}} 
\leq \sum_{r \geq 0} \exp \left( - \sum_{i=i_r+1}^{i_{r+1}} \frac{w_{i_r}}{W_i} \right),
\label{eq:criterion_crude_lower_bound}
\end{align}
using that $(w_i)_{i\geq 1}$ is non-increasing.
Hence, if we find a sequence $(i_r)_{r\geq 0}$ such that the right-hand side of \eqref{eq:criterion_crude_lower_bound} is smaller than $\varepsilon$, then we can deduce that $\P(\haut(\ttT_n) \geq t_n-1) \geq \Pp{I_{t_n-1} \leq n} \geq 1-\varepsilon$, where $(t_n)_{n\geq 1}$ is defined in such a way that $i_{t_n-1} < n \leq i_{t_n}$ for all $n\geq 1$.

We now illustrate this method on with example: assume $w_i = i^{-\alpha}$ for some $\alpha > 1$ (this corresponds to the case $\beta = 1$ in Theorem~\ref{thm:quickly}).
Then, using the bound $W_i \geq W_1 = 1$, we have
\[
	\sum_{i=i_r+1}^{i_{r+1}} \frac{w_{i_r}}{W_i}
	\geq i_r^{-\alpha} \cdot (i_{r+1}-i_r).
\]
In order for this to be large, we want $i_r^\alpha = o(i_{r+1})$. This leads naturally to the choice $i_r = \exp(\kappa^r)$ with $\kappa > \alpha$ for which the series on the right-hand side of \eqref{eq:criterion_crude_lower_bound} is convergent.
But, to get it as small as we want, we rather take $i_r = \delta \cdot \exp(\kappa^r)$ for small enough $\delta > 0$.
This proves a lower bound with $t_n = (\log \log n)/\log \kappa + O(1)$, which yields the lower bound in Theorem \ref{thm:quickly} by letting $\kappa \to \alpha$  (at least in probability, but one can strengthen it to an almost sure result).

This method, despite being very crude, gives the correct first order for the lower bound in all cases of Theorem \ref{thm:powers_of_log} and Theorem \ref{thm:quickly}, at least in the case of non-increasing $(w_i)_{i\geq 1}$. 
One can weaken this assumption, but then it does not work to choose the first child of each vertex along the path, one has to chose the first child with sufficient weight.
Pushing this method to get the next orders in the framework of Theorem \ref{thm:powers_of_log}, one sees that the lower bound obtained does not match the truth at the term of order $\frac{\log n}{\log \log n} \log \log \log n$.
Moreover, in the case of polynomially-growing $W_n$ treated in \cite{painsenizergues2022}, this method would not give the correct first order. 

\section{Slowly varying functions}
\label{sec:slowly varying}

In this short section, we state and prove a few results concerning slowly varying functions that we need in the next section in order to prove Theorem~\ref{thm:powers_of_log} and Theorem~\ref{thm:quickly}.
\begin{lemma} \label{lem:slowly_varying}
	The following holds.
	\begin{enumerate}
		\item\label{it:slowly_varying_0} Assume $L \colon [1,\infty) \to (0,\infty)$ satisfies \ref{hyp_SV_0}. Then, as $x \to \infty$, $\log L(x) = o(\log x)$ and, for any $\lambda >1$, uniformly in $y > \lambda x$,
		\begin{equation} \label{eq:difference_log_L_SV0}
		\log L(y) - \log L(x) = \petito{\log \left( \frac{y}{x} \right)}.
		\end{equation}
		\item\label{it:slowly_varying_1} Assume $L \colon [1,\infty) \to (0,\infty)$ satisfies \hyperref[hyp_SV_k]{$(\mathrm{SV}_1)$}. Then, as $x \to \infty$, uniformly in $y > x$,
		\begin{equation} \label{eq:difference_log_L}
		\log L(y) - \log L(x) = \petito{\log \left( \frac{y}{x} \right)}.
		\end{equation}
		\item\label{it:slowly_varying_2} Assume $L \colon [1,\infty) \to (0,\infty)$ satisfies \hyperref[hyp_SV_k]{$(\mathrm{SV}_2)$}.
		Let $\widetilde{L}(x) \coloneqq \exp(xL'(x)/L(x))$.
		Then, $\widetilde{L}$ satisfies \hyperref[hyp_SV_k]{$(\mathrm{SV}_1)$} and, uniformly in $y > x$,
		\begin{equation} \label{eq:difference_log_L_precis}
		\log L(y) - \log L(x) 
		= \log \widetilde{L}(y) \cdot \log \left( \frac{y}{x} \right) 
		+ \petito{\log^2 \left( \frac{y}{x} \right)},
		\end{equation}
	\end{enumerate}
\end{lemma}
\begin{proof}
	\textit{Part \ref{it:slowly_varying_0}.} This follows easily from Karamata's Representation Theorem \cite[Theorem 1.3.1]{binghamgoldieteugels1989}, which shows that, for some $A$ large enough, $L$ can be written as
	\[
	L(x) = \exp \left( f(x) + \int_A^x \frac{g(u)}{u} \diff u \right), 
	\qquad x \geq A,
	\]
	where $f$ and $g$ are measurable functions such that $f(x)$ has a finite limit as $x \to \infty$ and $g(x) \to 0$.

	\noindent 
	\textit{Part \ref{it:slowly_varying_1}.} 
	When $L$ satisfies \hyperref[hyp_SV_k]{$(\mathrm{SV}_1)$}, the representation above holds with $f = 0$ and $g(x) = xL'(x)/L(x)$.
	In particular, we can write, using $xL'(x)/L(x) \to 0$,
	\begin{align} \label{eq:diff_log_L_as_an_integral}
	\log L(y) - \log L(x) 
	= \int_x^y \frac{1}{u} \cdot \frac{u L'(u)}{L(u)} \diff u
	= \petito{\int_x^y \frac{1}{u} \diff u}
	= \petito{\log \left( \frac{y}{x} \right)},
	\end{align}
	which proves \eqref{eq:difference_log_L}.

	\noindent
	\textit{Part \ref{it:slowly_varying_2}.}
	It follows from a derivative calculation that $\widetilde{L}$ satisfies \hyperref[hyp_SV_k]{$(\mathrm{SV}_1)$}.
	Moreover, we have
	\begin{align*}
	\log L(y) - \log L(x)
	& = \int_x^y \frac{1}{u} \log \widetilde{L}(u) \diff u 
	= \log \widetilde{L}(y)
	\cdot \log \left( \frac{y}{x} \right)
	+ \int_x^y \frac{1}{u} \cdot 
	\left( \log \widetilde{L}(u) - \log \widetilde{L}(y) \right) \diff u \nonumber \\
	& = \log \widetilde{L}(y)
	\cdot \log \left( \frac{y}{x} \right)
	+ \int_x^y \frac{1}{u} \cdot \petito{\log \left( \frac{y}{u} \right)} \diff u,
	\end{align*}
	using \eqref{eq:difference_log_L} but for $\widetilde{L}$. This yields \eqref{eq:difference_log_L_precis}.
\end{proof}

\begin{lemma} \label{lem:integral_slowly_varying}
	Let $\alpha > 0$.
	Assume $L \colon [1,\infty) \to (0,\infty)$ satisfies \hyperref[hyp_SV_k]{$(\mathrm{SV}_1)$}.
	Let $(\ell_t)_{t \geq 1}$ be an increasing sequence such that $\ell_{t} \to \infty$ and $\ell_{t-1}/\ell_t \to 1$.
	Then, as $t \to \infty$, 
	\begin{equation} \label{eq:integral_slowly_varying}
	\int_{\ell_{t-1}}^{\ell_t} x^{-\alpha} L(x) \diff x 
	= \ell_t^{-\alpha} (\ell_t-\ell_{t-1}) L(\ell_t)
	\left( 1 + \frac{\alpha}{2} \cdot \frac{\ell_t-\ell_{t-1}}{\ell_t} (1+o(1))\right).
	\end{equation}
\end{lemma}

\begin{proof}
	We decompose
	\begin{align} \label{eq:decompo}
	\int_{\ell_{t-1}}^{\ell_t} x^{-\alpha} L(x) \diff x 
	= (\ell_t-\ell_{t-1}) \ell_t^{-\alpha} L(\ell_t)
	+ \int_{\ell_{t-1}}^{\ell_t} (x^{-\alpha}-\ell_t^{-\alpha}) L(x) \diff x 
	+ \ell_t^{-\alpha} \int_{\ell_{t-1}}^{\ell_t} (L(x)-L(\ell_t)) \diff x.
	\end{align}
	We first deal with the third term on the right-hand side of \eqref{eq:decompo}.
	Using the same representation for $L$ as in \eqref{eq:diff_log_L_as_an_integral} and the fact that $xL'(x)/L(x) \to 0$, we have, uniformly in $x \in [\ell_{t-1},\ell_t]$,
	\begin{align} \label{eq:difference_L}
	\frac{L(x)-L(\ell_t)}{L(\ell_t)}
	= \exp \left( - \int_x^{\ell_t} \frac{1}{y} \cdot \frac{y L'(y)}{L(y)} \diff y  \right) - 1
	= \exp \left( \petito{\log \frac{\ell_t}{\ell_{t-1}}} \right) - 1
	= \petito{\frac{\ell_t-\ell_{t-1}}{\ell_t}},
	\end{align}
	using that $(\ell_t-\ell_{t-1})/\ell_t \to 0$ as a consequence of the assumptions of the lemma.
	Therefore, we get
	\begin{align*}
	\ell_t^{-\alpha}\int_{\ell_{t-1}}^{\ell_t} (L(x)-L(\ell_t)) \diff x
	& = \petito{\ell_t^{-\alpha} \frac{\ell_t-\ell_{t-1}}{\ell_t} L(\ell_t) \int_{\ell_{t-1}}^{\ell_t}  \diff x}
	= \petito{\ell_t^{-\alpha} \frac{(\ell_t-\ell_{t-1})^2}{\ell_t} L(\ell_t)},
	\end{align*}
	proving that this term can be included in the $o(1)$ term on the right-hand side of \eqref{eq:integral_slowly_varying}.
	We now deal with the second term on the right-hand side of \eqref{eq:decompo}.
	By \eqref{eq:difference_L}, we have $L(x) = L(\ell_t) (1+o(1))$ uniformly in $x \in [\ell_{t-1},\ell_t]$.
	Therefore, with the change of variable $x = \ell_t-u$, we have
	\begin{align*}
	\int_{\ell_{t-1}}^{\ell_t} (x^{-\alpha}-\ell_t^{-\alpha}) L(x) \diff x 
	& = \ell_t^{-\alpha} L(\ell_t) (1+o(1)) \int_{0}^{\ell_t-\ell_{t-1}} \left( \left( 1 - \frac{u}{\ell_t} \right)^{-\alpha}-1 \right) \diff u \\
	& \sim \ell_t^{-\alpha} L(\ell_t) 
	\int_{0}^{\ell_t-\ell_{t-1}} \frac{\alpha u}{\ell_t} \diff u \\
	& \sim (\ell_t-\ell_{t-1}) \ell_t^{-\alpha} L(\ell_t) 
	\cdot \frac{\alpha}{2} \cdot \frac{\ell_t-\ell_{t-1}}{\ell_t}.
	\end{align*}
	Coming back to \eqref{eq:decompo}, this concludes the proof.
\end{proof}


\begin{lemma} \label{lem:from_J_to_L}
	Assume $J \colon [1,\infty) \to (0,\infty)$ satisfies \hyperref[hyp_SV_k]{$(\mathrm{SV}_k)$} for some $k \geq 1$.
	\begin{enumerate}
		\item\label{it:from_J_to_L_1} For any $\alpha \in (0,1)$, there exists a function $L \colon [1,\infty) \to (0,\infty)$ satisfying \hyperref[hyp_SV_k]{$(\mathrm{SV}_{k-1})$} such that $L(x) \sim (1-\alpha) J(x)$ and, for $x$ large enough,
		\begin{equation} \label{eq:function_J_alpha<1}
		x^{1-\alpha} J(x) = \int_1^x y^{-\alpha} L(y) \diff y.
		\end{equation}
		\item\label{it:from_J_to_L_2} For any $\alpha > 1$, there exists a function $L \colon [1,\infty) \to (0,\infty)$ satisfying \hyperref[hyp_SV_k]{$(\mathrm{SV}_{k-1})$} such that $L(x) \sim (\alpha-1) J(x)$ and, for $x$ large enough,
		\begin{equation} \label{eq:function_J_alpha>1}
		x^{1-\alpha} J(x) = \int_x^\infty y^{-\alpha} L(y) \diff y.
		\end{equation}
	\end{enumerate}
\end{lemma}
\begin{proof}
	Let $\alpha \in \R$.
	Let $L_0(x) \coloneqq (1-\alpha)J(x) + x J'(x)$, so that $x^{-\alpha} L_0(x) = \frac{\diff}{\diff x} (x^{1-\alpha} J(x))$.
	Then, if $\alpha \neq 1$, it follows from \hyperref[hyp_SV_k]{$(\mathrm{SV}_1)$} for $J$ that $L_0(x) \sim (1-\alpha)J(x)$.
	Using this fact and \hyperref[hyp_SV_k]{$(\mathrm{SV}_k)$} for $J$, we get, for any $i \in \llbracket 1,k-1 \rrbracket$,
	\begin{equation} \label{eq:derivative_L_0}
	\frac{x^i L_0^{(i)}(x)}{L_0(x)}
	= (1-\alpha) \frac{x^i J^{(i)}(x)}{L_0(x)} + \frac{x^{i+1} J^{(i+1)}(x)}{L_0(x)}
	\xrightarrow[x \to \infty]{} 0.
	\end{equation}
	\textit{Part \ref{it:from_J_to_L_1}.} Assume that $\alpha \in (0,1)$.
	We write 
	\[
	x^{1-\alpha} J(x) = J(1) + \int_{1}^x y^{-\alpha} L_0(y) \diff y,
	\]
	Since $L_0(x) \sim (1-\alpha)J(x)$, there exists $x_0$ such that for $x \geq x_0$, $L_0(x) > 0$.
	We define $L(x) = L_0(x)$ for $x \geq x_0$.
	This ensures that $L$ satisfies \hyperref[hyp_SV_k]{$(\mathrm{SV}_{k-1})$} by \eqref{eq:derivative_L_0}.
	For $x < x_0$, we can choose $L(x)$ such that $L$ is $k-1$ times differentiable positive and 
	\[
	J(1) + \int_1^{x_0} y^{-\alpha} L_0(y) \diff y = \int_1^{x_0} y^{-\alpha} L(y) \diff y.
	\]
	Then \eqref{eq:function_J_alpha<1} holds for any $x \geq x_0$.
	
	\noindent
	\textit{Part \ref{it:from_J_to_L_2}.} Assume that $\alpha > 1$.
	Then, there exists $x_0$ such that for $x \geq x_0$, we have $L_0(x) < 0$. We choose $L(x) = -L_0(x)$ for $x \geq x_0$.
	Moreover, we have
	\[
	x^{1-\alpha} J(x) = - \int_x^\infty y^{-\alpha} L_0(y) \diff y,
	\]
	so \eqref{eq:function_J_alpha>1} holds for any $x \geq x_0$.
\end{proof}

%
%
%

\section{Converting assumptions on \texorpdfstring{$a_n$}{an} to \texorpdfstring{$W_n$}{Wn}}\label{sec:correspondence behavior an and log Wn}

\begin{proof}[Proof of \eqref{eq:transfer assum Wn an}]
We denote $\delta(x) \coloneqq \log(1-x)+x$ and then we write
\[
\frac{W_1}{W_n}=\prod_{i=2}^{n}\frac{W_{i-1}}{W_{i}}
=\exp\left(\sum_{i=2}^{n} \log\left(1 - \frac{w_i}{W_i}\right)\right)
=\exp \left(- \sum_{i=2}^{n} \frac{w_i}{W_i} + \sum_{i=2}^{n}  \delta\left(\frac{w_i}{W_i}\right)\right).
\]
This yields
\[
a_n= 1+\sum_{i=2}^{n} \frac{w_i}{W_i} 
= 1+ \log \left(\frac{W_n}{W_1}\right) 
- \sum_{i=2}^{\infty} \delta\left(\frac{w_i}{W_i}\right) 
+ \sum_{i=n+1}^{\infty }  \delta\left(\frac{w_i}{W_i}\right)
= \log(W_n) +K + O\left(\sum_{i=n+1}^{\infty} \frac{w_i^2}{W_i^2} \right),
\]
as $n \to \infty$, since $\delta(x)=O(x^2)$ as $x\rightarrow 0$.
\end{proof}

\section{Application to specific weight sequences}\label{sec:application}

This section contains the computations needed to verify that we can indeed apply Theorem~\ref{thm:powers_of_log} and Theorem~\ref{thm:quickly} to the examples that we presented in Section~\ref{sec:powers_of_log_intro} and Section~\ref{sec:quickly converging intro} in the introduction.
Before diving into computations, let us note that if $a_n$ satisfies 
\begin{align*}
a_n = \int_{1}^{\log n} x^{-\alpha} L(x) \diff x + C + o\left( (\log n)^{-1-\alpha} (\log \log n)^2 L(\log n) \right).
\end{align*}
for some $\alpha>0$, some constant $C$, and some function $L$ that satisfies \hyperref[hyp_SV_k]{$(\mathrm{SV}_2)$}, we can change the values of $L$ on a finite interval without changing its regularity in such a way that assumption \eqref{eq:assumption_a_n_1} holds.
Hence, in the examples, we only check without loss of generality that the above display holds for some constant $C$. 
To do that we additionally use \eqref{eq:transfer assum Wn an}, which allows us to only check the asymptotics of $W_n$.
The main difficulty is that, in order to get a small enough error term, we sometimes need to include corrective terms in the definition of $L$.

\medskip

\noindent
\textbf{Case $w_n=\frac{\lambda (1-\alpha)}{n(\log n)^\alpha}\cdot \exp(\lambda (\log n)^{1-\alpha})$ with $\alpha \in \intervalleoo{0}{1}$ and $\lambda >0$.}
First, using a sum-integral comparison we get that $W_n=\exp(\lambda (\log n)^{1-\alpha})+\grandO{1}$.
It is easy to check from there that $\sum_{i \geq n} w_i^2/W_i^2 = O(n^{-1} (\log n)^{-2\alpha})$ so that \eqref{eq:ass_somme_carre} holds and also that 
\[
\log(W_n)
= \lambda (\log n)^{1-\alpha} + \grandO{\exp(-\lambda (\log n)^{1-\alpha})}
= \int_{1}^{\log n} x^{-\alpha}\lambda (1-\alpha)  \diff x + \grandO{(\log n)^{-1-\alpha}}.
\]
Using the last display and \eqref{eq:transfer assum Wn an}, we get that $a_n$ satisfies \eqref{eq:assumption_a_n_1} with a function $L$ chosen such that $L$ is twice differentiable and $L(x)=\lambda (1-\alpha)$ for $x$ sufficiently large. This ensures that we can apply Theorem~\ref{thm:powers_of_log}.

\medskip 

\noindent \textbf{Case $w_n=\frac{1}{n}$.}
We have $W_n = \sum_{i=1}^{n}\frac{1}{i}=\log n +\gamma+\grandO{\frac{1}{n}}$, where $\gamma$ is the Euler--Mascheroni constant.
We can easily check that $\sum_{i \geq n} w_i^2/W_i^2 = O(n^{-1} (\log n)^{-2})$, which ensures that \eqref{eq:ass_somme_carre} holds.
We also get
\begin{align*}
\log W_n = \log \left(\log n +\gamma+\grandO{\frac{1}{n}}\right) &= \log \log n + \frac{\gamma}{\log n} + \grandO{\frac{1}{(\log n)^2}}\\
&= \int_{1}^{\log n} \frac{1}{x}\left(1-\frac{\gamma}{x}\right) \diff x + \gamma+\grandO{\frac{1}{(\log n)^2}}.
\end{align*} 
Now we use \eqref{eq:transfer assum Wn an} and the last display to ensure that $a_n$ satisfies \eqref{eq:assumption_a_n_1} with a function $L$ chosen such that $L$ is twice differentiable and $L(x)=1-\frac{\gamma}{x}$ for $x$ sufficiently large. This ensures that we can apply Theorem~\ref{thm:powers_of_log}.

\medskip 

\noindent \textbf{Case $w_n=\frac{1}{n(\log n)^\alpha}$ with $\alpha>1$.}
In that case, $W_n= W_\infty - \frac{(\log n)^{1-\alpha}}{\alpha-1} + \grandO{\frac{1}{n}}$.
We can first easily check that $\sum_{i \geq n} w_i^2/W_i^2 = O(n^{-1} (\log n)^{-2\alpha})$, which ensures that \eqref{eq:ass_somme_carre} holds. 
Also for any $k\geq 1$ we have
\begin{align*}
\log W_n &= \log\left(W_\infty - (W_\infty - W_n)\right)\\
&=\log W_\infty + \log \left(1 - \frac{(\log n)^{1-\alpha}}{W_\infty(\alpha-1)}+ \grandO{\frac{1}{n}}\right)\\
&= \log W_\infty - \sum_{i=1}^{k} \frac{1}{i}\frac{(\log n)^{i(1-\alpha)}}{W_\infty^i(\alpha-1)^i} + \grandO{(\log n)^{(k+1)(1-\alpha)}}\\
&=\log W_\infty - \int_{\log n}^\infty x^{-\alpha} \left( \sum_{i=1}^{k} \frac{x^{(i-1)(1-\alpha)}}{W_\infty^i(\alpha-1)^{i-1}} \right) \diff x 
+ \grandO{(\log n)^{(k+1)(1-\alpha)}}\\
&= \int_1^{\log n} x^{-\alpha} \left( \sum_{i=0}^{k-1} \frac{x^{i(1-\alpha)}}{W_\infty^{i+1}(\alpha-1)^i} \right) \diff x 
+ C + \grandO{(\log n)^{(k+1)(1-\alpha)}}.
\end{align*}
Now choosing $k$ so that $(k+1)(1-\alpha)\leq -1-\alpha$, and using \eqref{eq:transfer assum Wn an} again we get that 
$a_n$ satisfies \eqref{eq:assumption_a_n_1} with a function $L$ chosen such that $L$ is twice differentiable and $L(x) = \sum_{i=0}^{k-1} \frac{x^{i(1-\alpha)}}{W_\infty^{i+1}(\alpha-1)^i}$ for $x$ sufficiently large. This ensures that we can apply Theorem~\ref{thm:powers_of_log}.

\medskip
\noindent
\textbf{Case $w_n=n^{-\alpha}$ for $\alpha>1$.} 
It is easy to check that we then have $\sum_{i\geq n} w_i^2/W_i^2 = \grandO{n^{1-2\alpha}}$ so that \eqref{eq:ass_somme_carre_beta=1} holds. 
We have $W_n=W_\infty - \sum_{i=n+1}^{\infty}n^{-\alpha}= W_\infty - \frac{n^{1-\alpha}}{\alpha-1} + \petito{n^{1-\alpha}}$ so that using \eqref{eq:transfer assum Wn an} we get 
\[
a_\infty - a_n 
= \log W_\infty - \log W_n + \grandO{\sum_{i\geq n+1}\frac{w_i^2}{W_i^2}}
= \frac{n^{1-\alpha}}{W_\infty(\alpha-1)}\cdot \left(1 + \petito{1}\right).
\]
Hence, the hypotheses of Theorem~\ref{thm:quickly} are satisfied with $\beta=1$ and $J(x) = 1/(W_\infty(\alpha-1))$.

\addcontentsline{toc}{section}{References}
\bibliographystyle{abbrv}

\end{document}